\documentclass[a4paper,11pt]{article}

\usepackage{color}
\usepackage[latin1]{inputenc}
\usepackage[T1]{fontenc}
\usepackage[normalem]{ulem}
\usepackage[english]{babel}
\usepackage{verbatim}
\usepackage{graphicx}
\usepackage{enumerate,enumitem}
\usepackage{amsmath,amssymb,amsfonts,amsthm,mathrsfs}
\usepackage{rotating,relsize}
\usepackage{float}
\usepackage{array}
\usepackage{MnSymbol}
\usepackage[all,cmtip]{xy}
 \usepackage[hidelinks]{hyperref}
\xyoption{all}
\setlist[enumerate]{leftmargin=.7cm,label=\roman*)}

\newtheorem{theorem}{Theorem}[section]
\newtheorem*{theorem*}{Theorem}
\newtheorem*{cor*}{Corollary}

\newtheorem{lemma}[theorem]{Lemma}
\newtheorem{prop}[theorem]{Proposition}
\newtheorem{cor}[theorem]{Corollary}

\theoremstyle{definition}
\newtheorem*{remark*}{Remark}
\newtheorem{defn}[theorem]{Definition}
\newtheorem{rem}[theorem]{Remark}
\newtheorem{example}[theorem]{Example}

\DeclareMathOperator{\id}{id}

\DeclareMathOperator{\tr}{tr}

\DeclareMathOperator{\map}{Map}

\DeclareMathOperator{\Top}{Top}

\DeclareMathOperator*{\colim}{colim}
\DeclareMathOperator*{\hocolim}{hocolim}

\DeclareMathOperator{\THH}{THH}
\DeclareMathOperator{\THR}{THR}

\DeclareMathOperator{\GW}{GW}

\DeclareMathOperator{\res}{res}
\DeclareMathOperator{\KR}{KR}

\DeclareMathOperator{\Sp}{Sp}

\DeclareMathOperator{\ev}{ev}

\DeclareMathOperator{\Hom}{Hom}
\DeclareMathOperator{\Sh}{sh}
\DeclareMathOperator{\Inj}{Inj}

\DeclareMathOperator{\Z}{\mathbb{Z}}
\DeclareMathOperator{\F}{\mathbb{F}}
\DeclareMathOperator{\sd}{sd}
\DeclareMathOperator{\Tor}{\underline{Tor}}
\DeclareMathOperator{\proj}{proj}
\DeclareMathOperator{\tran}{tran}
\DeclareMathOperator{\R}{\mathbb{R}}
\DeclareMathOperator{\incl}{incl}
\DeclareMathOperator{\Ho}{Ho}
\DeclareMathOperator{\Ab}{Ab}
\DeclareMathOperator{\Ext}{Ext}
\DeclareMathOperator{\Torr}{Tor}
\DeclareMathOperator{\red}{red}
\DeclareMathOperator{\Fun}{Fun}
\DeclareMathOperator{\sk}{sk}
\DeclareMathOperator{\Mon}{Mon}

\begin{document}
\begin{center}\LARGE{Real topological Hochschild homology}
\end{center}

\begin{center}\Large{Emanuele Dotto, Kristian Moi, Irakli Patchkoria, Sune Precht Reeh}
\end{center}

\vspace{.05cm}

\abstract{
This paper interprets Hesselholt and Madsen's real topological Hochschild homology functor $\THR$ in terms of the multiplicative norm construction.
We show that $\THR$ satisfies cofinality and Morita invariance, and that it is suitably multiplicative. We then calculate its geometric fixed points and its Mackey functor of components, and show a decomposition result for group-algebras. Using these structural results we determine the homotopy type of $\THR(\F_p)$ and show that its bigraded homotopy groups are polynomial on one generator over the bigraded homotopy groups of $H\F_p$. We then calculate the homotopy type of $\THR(\Z)$ away from the prime $2$, and the homotopy ring of the geometric fixed-points spectrum $\Phi^{\mathbb{Z}/2}\THR(\Z)$.
}

\vspace{.05cm}


\tableofcontents

\section*{Introduction}
\phantomsection\addcontentsline{toc}{section}{Introduction} 
The topological Hochschild homology spectrum $\THH$ of a ring spectrum was introduced in \cite{Bok} and \cite{BHM} as a tool to study the algebraic $K$-theory of rings and connective ring spectra by means of the trace map. The trace induces an equivalence between a stabilized version of $K$-theory and $\THH$ \cite{DM}, and an equivalence between the relative $K$-theory and the relative topological cyclic homology of a nilpotent extension \cite{McCarthy}, \cite{Dundas}, \cite{DGM}.
It was later discovered that $\THH$ itself carries arithmetic information, having important relations to Witt vectors \cite{Wittvect}, to the deRham-Witt complex \cite{IbLarsLocalfields}, \cite{IbLarsDeRhamMixed}, and to the Hasse-Weil zeta function \cite{LarsZeta}.

The real topological Hochschild homology spectrum $\THR$ is a genuine $\Z/2$-equivariant spectrum associated to a genuine $\Z/2$-spectrum with a compatible multiplicative structure, introduced in \cite{IbLars}, \cite{thesis} and \cite{Amalie}. It is a genuine equivariant refinement of the antipodal map of the Hopf algebra structure on $\THH$ of commutative ring spectra defined by Angeltveit and Rognes \cite{angrog}, of the involution on $\THH$ defined in \cite{Kroinv}, and of the involution on the Hochschild complex of a ring with anti-involution of \cite{Loday}. It receives a trace map  $\tr\colon \KR\to \THR$ from the real $K$-theory spectrum $\KR$ of \cite{IbLars}, and therefore its derived fixed points spectrum $\THR^{\Z/2}$ is the target of a trace map from the Hermitian $K$-theory spectrum or Grothendieck-Witt spectrum $\GW$ of \cite{Schlichting,IbLars}. The trace map induces an equivalence of equivariant spectra after a suitable stabilization of $\KR$, at least when $2$ is invertible \cite{thesis}, \cite{Gcalc}, showing that $\THR$ carries information about Hermitian $K$-theory of rings with anti-involution.

The current definition of the $\THR$ spectrum uses a construction based on the original model of \cite{Bok}. Although this construction is homotopy invariant (Theorem \ref{invBok}) and receives a trace map from $\KR$, it is not suitable for calculations. The main purpose of this paper is to give an alternative construction of the $\THR$ spectrum in terms of the multiplicative norm construction of \cite{HHR} and \cite{Sto}, and use this description to carry out calculations.
\\

The input for $\THR$ is a ring spectrum with anti-involution $(A,w)$. A point-set model for their homotopy theory is the category of orthogonal $\Z/2$-spectra $A$ which are equipped with the structure of a ring spectrum, such that the involution $w$ defines a map of ring spectra $w\colon A^{op}\to A$ (see \S\ref{flatMS}). Alternatively, these can be described as algebras in genuine $\Z/2$-spectra over a certain $E_\sigma$-operad, where $\sigma$ is the sign representation of $\Z/2$ (see Remark \ref{operad}).
Let $N^{\mathbb{Z}/2}_e A$ be the norm construction of the underlying ring spectrum of $A$. Then the $\Z/2$-spectrum $A$ is a module over the $\Z/2$-ring spectrum $N^{\mathbb{Z}/2}_e A$ via the map
\[
N^{\mathbb{Z}/2}_e A \wedge A=A \wedge A\wedge A\xrightarrow{1\wedge w\wedge 1}A \wedge A^{op}\wedge A\xrightarrow{1\wedge \tau}A \wedge A\wedge A^{op}\stackrel{\mu}{\longrightarrow}A,
\]
where $\tau\in\Sigma_2$ is the flip permutation and $\mu$ is the multiplication of $A$. This also provides an $A$-module structure on the geometric fixed points of $A$
\[
A\wedge\Phi^{\Z/2}A\stackrel{\simeq}{\longrightarrow} \Phi^{\Z/2}(N^{\mathbb{Z}/2}_e A \wedge A)\longrightarrow \Phi^{\Z/2}A
\]
under suitable cofibrancy assumptions. Similar module structures are defined for an $(A,w)$-bimodule $M$. The following theorem is the main outcome of \S\ref{secTHR}.

\begin{theorem*}
Let $(A,w)$ be a flat ring spectrum with anti-involution, and let $M$ be a flat  $(A,w)$-bimodule. There is a stable equivalence of genuine $\Z/2$-spectra
\[
\THR(A;M)\simeq M \wedge^{\mathbf{L}}_{N^{\mathbb{Z}/2}_e A} A,
\]
where $\THR(A;M)$ is the B\"{o}kstedt model of real topological Hochschild homology defined in  \cite{IbLars}, \cite{thesis} and \cite{Amalie} (see also \S\ref{secBok}).
In particular, on geometric fixed points there is a natural stable equivalence of spectra 
\[
\Phi^{\Z/2}\THR(A;M)\simeq \Phi^{\Z/2}M \wedge^{\mathbf{L}}_{A} \Phi^{\Z/2}A.
\]
\end{theorem*}
From this geometric fixed points formula one can recover the calculation of $\pi_0\Phi^{\Z/2}\THR(A)$ of \cite{AmalieGeom}.
We prove this theorem first by comparing $\THR(A;M)$ to the dihedral Bar construction $B^{di}_\wedge (A;M)$ in a way analogous to  \cite{Sh00} and \cite{PatchkSagave} in Theorem \ref{comparison}, and then by showing that $B^{di}_\wedge (A;M)$ computes the derived smash product. We wish to emphasize that our equivalence is implemented by a zig-zag of only two simplicial equivalences, of the form
\[M\wedge A^{\wedge k}\stackrel{\simeq}{\longrightarrow}\hocolim_{I^{\times k+1}}\Omega^{i_0+\dots+i_k}(\Sh^{i_0} M\wedge \Sh^{i_1} A\wedge\dots\wedge \Sh^{i_k} A)\stackrel{\simeq}{\longleftarrow}\THR_k(A;M).
\]

The real algebraic $K$-theory functor admits as input Wall antistructures, or more generally exact categories with duality. It is therefore desirable to define $\THR$ on these more general objects. This is addressed in \S\ref{secgen}, where we extend the $\THR$ functor to a set-up of spectrally enriched categories with duality which include Wall antistructures, their category of modules, as well as a spectral version of Wall antistructures. We also extend the comparison between $\THR$ and the dihedral Bar construction and the calculation of its geometric fixed points spectra to this categorical context.

In \S\ref{secprop} we discuss multiplicativity, cofinality, and Morita invariance for the functor $\THR$. If the underlying ring spectrum $A$ is commutative, the ring spectrum with anti-involution $(A,w)$ is nothing but a genuine $\Z/2$-equivariant commutative ring spectrum. The following results are proved in \S\ref{secmult}.

\begin{theorem*}
Let $(A,w)$ be a commutative $\Z/2$-equivariant ring spectrum. Then $\THR(A):=\THR(A;A)$ is an associative $\Z/2$-equivariant ring spectrum, and the dihedral Bar construction $B^{di}_\wedge A$ is a $\Z/2$-equivariant augmented commutative $A$-algebra. If $(A,w)$ is flat there is a stable equivalence of associative genuine $\Z/2$-equivariant ring spectra
\[
\THR(A)\simeq B^{di}_\wedge A\simeq A \wedge^{\mathbf{L}}_{N^{\mathbb{Z}/2}_e A} A,
\]
and a stable equivalence of associative ring spectra $\Phi^{\Z/2}\THR(A)\simeq \Phi^{\Z/2}A \wedge^{\mathbf{L}}_{A}\Phi^{\Z/2}A $.
\end{theorem*}

When $A$ is commutative $\THH(A)$ and the cyclic Bar construction $B^{cy}_\wedge A$ have natural structures of $E_{\infty}$-ring spectra. We are however not aware of a comparison of these objects as $E_\infty$-algebras. In our situation, $B^{di}_\wedge A$ is strictly commutative and therefore has an action of the genuine $\Z/2$-equivariant $E_\infty$-operad. It is plausible that $\THR(A)$ also has an action of this operad which is compatible with the one on $B^{di}_\wedge A$. We leave this as an open question. In our calculations we will only use the comparison as associative rings.
\\

Given a ring spectrum with anti-involution $(A,w)$ we let $M^{\vee}_nA:=\bigvee_{n\times n}A$ be the non-unital matrix ring. This has an anti-involution $w^T$ obtained by applying $w$ to the wedge summands and by transposing, that is by applying the involution of $n\times n$ that flips the factors. Now suppose that $R$ is a discrete ring with anti-involution. The category of finitely generated projective right $R$-modules inherits a duality $D(P):=\hom_R(P,R)$, where the left $R$-module $\hom_R(P,R)$ is regarded as a right $R$-module via the anti-involution.
The following is proved in \S\ref{seccof} and \S\ref{secmorita}.

\begin{theorem*}[Morita Invariance]
If $(A,w)$ is flat, the inclusion $(A,w)\to (M^{\vee}_nA,w^T)$ of the $(1,1)$-wedge summand induces a stable equivalence of genuine $\Z/2$-spectra
\[
\THR(A)\stackrel{\simeq}{\longrightarrow}\THR(M^{\vee}_nA)
\]
for every integer $n\geq 1$. Moreover $\THR$ satisfies cofinality, and it follows that there is a stable equivalence of genuine $\Z/2$-spectra
\[\THR(HR)\stackrel{\simeq}{\longrightarrow}\THR(H\mathcal{P}_R),\]
where $H$ denotes the Eilenberg-MacLane construction, with fixed-points spectrum $(HR)^{\Z/2}=H(R^{\Z/2})$. This equivalence in fact holds for every Wall antistructure $(R,w,\epsilon)$, in the sense of \cite{Wall}.
\end{theorem*}

The construction of the trace map from $\KR(R,w,\epsilon)$ of \cite{IbLars} and \cite{thesis} lands most naturally in the spectrum $\THR(H\mathcal{P}_R)$. Thus the previous theorem is instrumental for describing a trace map to the spectrum $\THR(HR)$.

\begin{remark*}
If $(R,w)$ is a \emph{discrete} ring with anti-involution, our geometric fixed points formula shows that, after inverting $2$, the geometric fixed points of $\THR$ vanish:
\[
(\Phi^{\Z/2}\THR(HR;M))[\tfrac{1}{2}]\simeq (\Phi^{\Z/2}M \wedge^{\mathbf{L}}_{HR} \Phi^{\Z/2}HR)[\tfrac{1}{2}]\simeq \Phi^{\Z/2}M \wedge^{\mathbf{L}}_{HR} \Phi^{\Z/2}(HR[\tfrac{1}{2}])\simeq \ast
\]
since $\Phi^{\Z/2}(HR[\tfrac{1}{2}])$ is contractible (see Corollary \ref{geomzero}). As a consequence the trace map
\[
L(R)\otimes\mathbb{Q}\simeq \Phi^{\Z/2}\KR(R)\otimes \mathbb{Q}\stackrel{\tr}{\longrightarrow}\Phi^{\Z/2}\THR(HR)\otimes \mathbb{Q}\simeq\ast
\]
cannot detect elements in rational $L$-theory. The first equivalence follows from \cite{jl} and \cite{BF}. This result is consistent with Corti{\~n}as' vanishing result for the Chern Character to dihedral homology on the $L$-theory summand of \cite{gc}.
For non-discrete ring spectra the geometric fixed points of $\THR$ need not vanish, even after inverting $2$. In fact, Dotto and Ogle \cite{DO} have shown that the trace for the spherical group-ring $\mathbb{S}[G]$ is non-trivial on rational geometric fixed points.
\end{remark*}

We now turn our attention to some calculations. From the description of $\THR$ as a derived smash product we obtain a strongly convergent spectral sequence of $\Z/2$-Mackey functors
\[ \Tor_{p,q}^{{\underline{\pi}}_\ast(N^{\mathbb{Z}/2}_e A)}({\underline{\pi}}_\ast A, {\underline{\pi}}_\ast A) \Rightarrow {\underline{\pi}}_{p+q}\THR(A),\]
where $\Tor$ is the left derived functor of the $\Box$-product of Mackey functors, and ${\underline{\pi}}_\ast$ is the $\Z/2$-Mackey functor of homotopy groups. From this spectral sequence and a calculation of $\underline{\pi}_0N^{\mathbb{Z}/2}_e A$ we obtain the following in \S\ref{secpi0}.

\begin{theorem*} Let $(A,w)$ be a ring spectrum with anti-involution whose underlying orthogonal $\Z/2$-spectrum is flat and connective. Then the Mackey functor $\underline{\pi}_0\THR(A)$ is naturally isomorphic to the Mackey functor
\[
\xymatrix@C=40pt{
\pi_0A/[\pi_0A, \pi_0A] \ar@<6ex>@(dl,dr)_-{w}\ar@<.5ex>[r]^-{\tran}& (\pi^{\mathbb{Z}/2}_0 A \otimes \pi^{\mathbb{Z}/2}_0 A)/T \ar@<.5ex>[l]^-{\res}
},
\]
where $[\pi_0A, \pi_0A]$ is the commutator subgroup, and $T$ is the subgroup generated by the relations:
\begin{itemize}
\item[i)] $x \otimes a \cdot y-\omega (a) \cdot x \otimes y$, for $x,y \in \pi^{\mathbb{Z}/2}_0 A$ and $a \in \pi_0A$, where $\cdot$ is a multiplicative action of $\pi_0A$ on $\pi^{\Z/2}_0A$ induced by the module structure over $\pi_0N^{\Z/2}_eA$,
\item[ii)] $x \otimes \tran(a \res(y) w(b))-\tran(w(b)\res(x)a) \otimes y$, for $x,y \in \pi^{\mathbb{Z}/2}_0 A$ and $a, b \in \pi_0A$.
\end{itemize}
The transfer and the restriction are defined respectively by
\[\res(x \otimes y)=\res(x)\res(y)\ \ \ \ \ \ \ \ \ \ \  \mbox{and} \ \ \ \ \ \ \ \ \ \ \ \ \  \tran(a)= \tran(a) \otimes 1,\]
where $1 \in \pi^{\mathbb{Z}/2}_0 A$ is defined by the unit $\mathbb{S}\to A$.

If $A$ is moreover commutative these relations generate an ideal, $(\pi^{\mathbb{Z}/2}_0 A \otimes \pi^{\mathbb{Z}/2}_0 A)/T$ is a commutative ring, and $\underline{\pi}_0\THR(A)$ is a Tambara functor. The multiplicative norm of $\underline{\pi}_0\THR(A)$ is given by $a\mapsto N(a) \otimes 1$, where $N$ is the multiplicative norm of $\underline{\pi}_0A$.
\end{theorem*}

A ring with anti-involution with particular geometric relevance is the group-ring $\Z[G]$ where $G$ is a discrete group, and the anti-involution is given on generators by $g \mapsto g^{-1}$. In general if $M$ is a monoid with anti-involution $\iota\colon M^{op}\to M$ and $(A,w)$ is a ring spectrum with anti-involution, one can form the monoid-ring $A[M]:=A\wedge M_+$ with anti-involution $w\wedge\iota$. In \cite{Amalie} H\o genhaven shows that there is an equivalence
\[
\THR(\mathbb{S}[G])\simeq \Sigma^{\infty}B^{di}_\times G_+
\]
where the dihedral Bar construction $B^{di}_\times G$ is a model for the free loop space $\map(S^{\sigma},B^\sigma G)$ with respect to the sign-representation $\sigma$. We generalize this result to arbitrary coefficients in \S\ref{grouprings}.

\begin{theorem*}
Let $(A,w)$ be a flat ring spectrum with anti-involution, and $M$ a well-pointed topological monoid with anti-involution. Then the assembly map
\[\THR(A)\wedge (B_{\times}^{di}M)_+\stackrel{\simeq}{\longrightarrow}\THR(A[M])\]
is a stable equivalence of genuine $\Z/2$-spectra. If $G$ is a discrete group and $\iota$ is inversion, there is an isomorphism of Mackey functors
\[
\underline{\pi}_0\THR(\Z[G])\cong \big(\xymatrix@C=40pt{
\mathbb{Z}[G_{conj}]\ar@(dl,dr)_-{(-)^{-1}}\ar@<.5ex>[r]^-{\tran}& (\Z[G_{conj}]_{\Z/2} \oplus \Z[G^{\Z/2} \times_G G^{\Z/2}])/D \ar@<.5ex>[l]^-{\res}
}\big),\]
where $G_{conj}$ is the set of conjugacy classes of elements of $G$, and $D$ is the subgroup generated by the elements $2[g,g'] - [gg']$, for all $[g,g'] \in G^{\Z/2} \times_G G^{\Z/2}$. The transfer and restriction maps are defined in Corollary \ref{cor:grpring}.
\end{theorem*}
In Examples \ref{ZZ} and \ref{ZZ2} we give explicit calculations for the cases $G=\Z$ and $G=\Z/2$ based on these formulas.
\\

We recall that the homotopy groups of a $\Z/2$-equivariant ring spectrum $A$ form a bigraded ring
\[
\pi_{n,k}A:=[S^{n,k},A]^{\Z/2}
\]
where $S^{n,k}=S^{n-k}\wedge S^{k\sigma}$ and $\sigma$ denotes the sign representation of $\Z/2$. We write $\Sigma^{n,k}$ for the corresponding suspension functor. Let $p$ be a prime and let $\F_p$ have the trivial involution. We let $T_{H\F_p}(S^{2,1}):=\bigvee^\infty_{n=0} \Sigma^{2n,n} H\F_p$ denote the free $H\F_p$-algebra generated by $S^{2,1}$.
The following is proved in \S\ref{secFp}.

\begin{theorem*}
There is a stable equivalence of genuine $\Z/2$-equivariant ring spectra $T_{H\F_p}(S^{2,1})\stackrel{\simeq}{\to} \THR(\F_p)$ for every prime $p$, and therefore an isomorphism of bigraded rings
\[\pi_{\ast,\ast}\THR(\F_p)\cong {H\F_p}_{\ast,\ast}[\tilde{x}],\]
where $\tilde{x}$ has bidegree $(2,1)$. We deduce that there are isomorphisms of graded rings
\[
\pi_\ast \THR(\F_p)^{\Z/2} \cong\left\{
\begin{array}{lll}
\F_p [y]&,\  |y|=4&, \ \mbox{for $p$ odd}
\\
\F_2[\bar{x},y]&, \  |y|=1, |\bar{x}|=2 &,\ \mbox{for $p=2$}.

\end{array}
\right.
\]
Under the restriction map  $\pi_\ast \THR(\F_p)^{\Z/2} \to \pi_\ast \THH(\F_p)\cong \F_p[x]$, the generator $y$ maps to $x^2$ for $p$ odd, and for $p=2$ the element $\bar{x}$ maps to $x$ and $y$ maps to $0$.
\end{theorem*}

To prove the theorem we lift the equivalence $\bigvee^\infty_{n=0} \Sigma^{2n} H\F_p\stackrel{\simeq}{\to} \THH(\F_p)$ of B{\"o}kstedt \cite{Bok} and Breen \cite{breen} to a map of $\Z/2$-spectra, and test that it induces an equivalence on geometric fixed points. If $p$ is odd the source has trivial geometric fixed points, and so does the target by the geometric fixed points formula for $\THR$ of \S\ref{secgeom}. When $p=2$ we use the calculation of \cite{HuKriz} of $\Phi^{\Z/2}H\F_2$ and the same formula to show that the map is an equivalence on geometric fixed points.


Let $\Z$ have the trivial involution. Using a similar strategy we calculate $\THR(\Z)$ localized at an odd prime $p$, or equivalently $\THR(\Z_{(p)})$, in \S\ref{secTHRZ}.
\begin{theorem*} Let $p$ be an odd prime. There is a stable equivalence of genuine $\Z/2$-spectra
\[\THR(\Z)_{(p)} \simeq \THR(\Z_{(p)})\simeq H\Z_{(p)} \vee \bigvee_{k \geq 1} \Sigma^{2k-1, k} H(\Z/p^{\nu_p(k)}),\]
where $\nu_p(k)$ is the $p$-adic valuation of $k$.
\end{theorem*}

If one could show that a similar equivalence holds at the prime $2$, we would get an equivalence
\[\THR(\Z) \stackrel{?}{\simeq} H\Z \vee \bigvee_{k \geq 1} \Sigma^{2k-1, k} H\Z/k.\]
In Theorem \ref{thm:phiTHR(Z)} we calculate the homotopy ring $\pi_\ast\Phi^{\Z/2}\THR(\Z)$, showing in particular that the geometric fixed-points on both sides of this expression have isomorphic homotopy groups.

\begin{theorem*} There is an isomorphism of graded rings
\[\pi_\ast \Phi^{\Z/2} \THR(\Z) \cong \F_2[b_1,b_2, e]/e^2, \]
where $|b_1| = |b_2| = 2$ and $|e| = 1$. The element $b_1+b_2$ lifts to an element of $\pi_2 (\THR(\Z)^{\Z/2}) \cong \Z/2$ of infinite multiplicative order.
\end{theorem*}

This theorem shows in particular that the multiplication on $\THR( \Z)$ is non-trivial already at the level of homotopy groups, as opposed to the multiplication on $\THH(\Z)$ which is trivial on the homotopy groups of positive degree. The remaining piece in the computation of $\THR(\Z)$ at the prime $2$ is to show that the isomorphism between the homotopy groups of the geometric fixed points can be realized by a map of $\Z/2$-spectra. The techniques used for the calculation at odd primes do not allow us to construct such a map. This will be the subject of future work. 

\subsection*{Acknowledgements}
The authors would like to thank the Hausdorff Research Institute for Mathematics in Bonn for its hospitality during the Junior Trimester Program in Topology in 2016. Much of the work on this paper was carried out during this program. The authors also acknowledge the support of the Danish National Research Foundation through the Centre for Symmetry and Deformation (DNRF92). The second author was supported by the Max Planck institute for Mathematics and thanks the Mittag-Leffler Institute for their hospitality. The third author was supported by the German Research Foundation Schwerpunktprogramm 1786. The fourth author was supported by Independent Research Fund Denmark's Sapere Aude program (DFF--4002-00224) and by the Max Planck Institute for Mathematics.

We would like to thank Lars Hesselholt and Ib Madsen for their support for this project. We thank Bj{\o}rn Dundas, Mike Hill, Amalie H{\o}genhaven, Thomas Nikolaus, John Rognes, Steffen Sagave, Peter Scholze, Stefan Schwede, Sean Tilson and Christian Wimmer for useful discussions.

\subsection*{Notation and conventions}
By a space, we will always mean a compactly generated weak Hausdorff topological space. We denote by $\Top$ the category of spaces, by $\Top^G$ the category of $G$-spaces for a finite group $G$, and by $\Top_\ast$ and $\Top_{\ast}^G$ the associated categories of pointed objects.

The category of orthogonal spectra will be denoted by $\Sp$ and the category of $G$-equivariant orthogonal spectra will be denoted by $\Sp^G$. We will mostly work with the flat stable (positive) model structure on the category $\Sp^G$ based on a complete $G$-universe, which we will call genuine $G$-spectra. This model structure is constructed in \cite{Sto, BrDuSt} and is referred to as the $\mathbb{S}$-model structure. We prefer the terminology flat since the cofibrant objects in this model structure behave like flat modules in algebra. The cofibrations in the flat model structure on $\Sp^G$ are called flat cofibrations and the cofibrant objects are just called flat.
We will occasionally use the stable model structure on $\Sp^G$ from \cite{HHR} which is Quillen equivalent to the flat model structure and has fewer cofibrant objects. An equivalence of $G$-spectra will always be understood as a stable equivalence with respect to a complete $G$-universe.
 We will mostly be concerned with the case $G=\Z/2$.

\section{Background}

\subsection{Equivariant diagrams and real simplicial objects}
In this section we recall some of the constructions from \cite{Gdiags}, \cite{IbLars}.

\begin{defn}
Let $J$ be a small category with an involution $\omega\colon J\to J$, and $\mathscr{C}$ a category. A $\mathbb{Z}/2$-diagram in $\mathscr{C}$ is a functor $X\colon J\to \mathscr{C}$ together with a natural transformation $w\colon X\longrightarrow X\circ \omega$ such that the composite
\[
X\stackrel{w}{\longrightarrow}X\circ \omega\stackrel{w|_\omega}{\longrightarrow}X\circ \omega^2=X
\]
is the identity natural transformation. A morphism of $\mathbb{Z}/2$-diagram is a natural transformation of underlying functors $f\colon X\to Y$ such that $w_Y\circ f=f|_{\omega}\circ w_X$.
\end{defn}

\begin{example}
Let $\Delta$ be the standard skeleton for the category of non-empty finite totally ordered sets and order-preserving maps. This category has an involution $\omega$ that is constant on objects, and that sends a morphism $\alpha\colon [n]\to [k]$ to
\[
\omega(\alpha)(i)=k-\alpha(n-i).
\]
This induces a similar involution on the opposite category $\Delta^{op}$. A $\mathbb{Z}/2$-diagram $X\colon \Delta^{op}\to \mathscr{C}$ is called a real simplicial object of $\mathscr{C}$.
Explicitly, this consists of a simplicial object $X$ together with a map $w\colon X^{op}\to X$ of order two, that is, involutions $w_k\colon X_k\to X_k$ such that for every $\alpha\colon [n]\to [k]$
\[\alpha^{\ast}w_k=w_n(\omega(\alpha))^{\ast}.\]
A morphism of real simplicial objects is a morphism of $\mathbb{Z}/2$-diagrams.
\end{example}

The cosimplicial object $\Delta^{\bullet}\colon \Delta\to \Top$ that sends $[n]$ to the topological $n$-simplex $\Delta^n$ has a canonical structure of $\mathbb{Z}/2$-diagram, defined by the map $w\colon \Delta^n\to \Delta^n$
\[
w(t_0,t_1,\dots,t_n)=(t_n,\dots, t_1, t_0).
\]

\begin{defn}\label{coend}
Let $X\colon \Delta^{op}\to \mathscr{C}$ be a real simplicial object in a category $\mathscr{C}$ which is tensored over $\Top$. The geometric realization of $X$ is the $\mathbb{Z}/2$-object of $\mathscr{C}$ defined as the coend
\[|X|:=X\otimes_{\Delta}\Delta^{\bullet}\]
 with the diagonal $\mathbb{Z}/2$-action defined in \cite[\S1.2]{Gdiags}.
\end{defn}

\begin{rem}\label{subdivision}
Let $\sd_e\colon \Delta^{op}\to \Delta^{op}$ denote the functor that sends $[n]$ to the join $[n]\ast [n]=[2n+1]$, and a morphism $\alpha$ to $\alpha\ast \omega(\alpha)$. Precomposition with this functor defines an endofunctor on the category of simplicial objects in $\mathscr{C}$, which is called the Segal edgewise subdivision and it is still denoted $\sd_e$. This construction was introduced by Segal in \cite{Segalsub}. It satisfies $\sd_e (X^{op})=\sd_eX$, and thus a $\mathbb{Z}/2$-structure $w\colon X\to X^{op}$ on a simplicial object $X$ induces a simplicial involution
\[\sd_eX\xrightarrow{\sd_ew} \sd_e (X^{op})=\sd_eX\]
on the Segal edgewise subdivision. The geometric realization $|\sd_eX|$ inherits an involution, and it is readily verified that the canonical isomorphism $|\sd_eX|\cong |X|$ is equivariant with respect to the $\mathbb{Z}/2$-action on $|X|$ of Definition \ref{coend}.
\end{rem}

We recall that the homotopy colimit of a $\Z/2$-diagram $X\colon J\to \Top$ inherits a $\Z/2$-action. This can be explicitly defined by expressing the homotopy colimit as the realization of the simplicial space
\[
\hocolim_JX:=|\coprod_{\underline{j}\in N_\bullet J}X_{j_0}|
\]
as in \cite{BK}. The involution is then the geometric realization of the simplicial involution that sends $(\underline{j},x)$ to $(\omega(\underline{j}),w(x))$. The homotopical properties of this involution have been studied in \cite{Gdiags}.

For a finite dimensional $\Z/2$-representation $V$, let $\ev_V \colon \Sp^{\Z/2} \to \Top_\ast^{\Z/2}$ be the evaluation at level $V$ functor.
\begin{defn}
A real simplicial spectrum $X\colon \Delta^{op}\to \Sp$ is good if for each $\Z/2$-representation $V$ the simplicial $\Z/2$-space $(\ev_V)_\ast \sd_e X $ is good, that is, if for all $n \geq 0 $ the map of $\mathbb{Z}/2$-spaces
\[(s_{2n+2-i}s_i)_V\colon (X_{2n+1})_V\longrightarrow (X_{2n+3})_V\]
is an $h$-cofibration of (unpointed) $\Z/2$-spaces.
\end{defn}

The following result follows immediately from Remark \ref{subdivision}, and the fact that fixed points of finite groups commute with geometric realizations.

\begin{lemma}\label{good}
Let $f\colon X\to Y$ be a map of good real simplicial orthogonal spectra, such that the map $f_n\colon X_n\to Y_n$ is a stable equivalence of orthogonal $\mathbb{Z}/2$-spectra. Then $|f|\colon |X|\to |Y|$ is a stable equivalence of orthogonal $\mathbb{Z}/2$-spectra.
\end{lemma}

\subsection{Real \texorpdfstring{$I$}{I}-spaces}\label{sec:realI}
Let $I$ be the category whose objects are the non-negative integers, and whose morphisms $i\to j$ are the injective maps $\alpha\colon \{1,\dots, i\}\to\{1,\dots, j\} $. This category has a $\mathbb{Z}/2$-action $\omega\colon I\to I$ which is trivial on objects and that sends a map $\alpha\colon i\to j$ to
\[\omega(\alpha)(s)=j-\alpha(i-s+1)+1.\]

\begin{defn}
A real $I$-space is a $\mathbb{Z}/2$-diagram $X\colon I\to \Top$.
\end{defn}
The class of examples of real $I$-spaces we are the most concerned with comes from $\mathbb{Z}/2$-spectra.

\begin{example}
Let $E$ be an orthogonal  $\mathbb{Z}/2$-spectrum. There is an associated $I$-space $\Omega_I E\colon I\to \Top_\ast$ which is defined by
\[(\Omega_I E)(i)=\Omega^iE_i\]
on objects, and as in \cite[\S 2.3]{SchlichtkrullUnits} on morphisms. The $I$-space $\Omega_I E$ has a $\mathbb{Z}/2$-diagram structure defined by the maps
\[
\Omega^iE_i\stackrel{()\circ\tau_i}{\longrightarrow}\Omega^iE_i \stackrel{w}{\longrightarrow}\Omega^iE_i\stackrel{\Omega^i\tau_i}{\longrightarrow}\Omega^iE_i
\]
where $w$ denotes the  $\mathbb{Z}/2$-action on $E$. Here $\tau_i\in \Sigma_i$ is the permutation that reverses the order on $\{1,\dots, i\}$, which is applied first to the smash factors of the sphere $S^i$, and then to $E_i$ through the orthogonal structure of $E$. We let $\Omega^{\infty}_IE$ be the pointed $\mathbb{Z}/2$-space defined as the homotopy colimit of this $\mathbb{Z}/2$-diagram
\[\Omega^{\infty}_IE:=\hocolim_I\Omega_IE\]
(where the homotopy colimit is computed in the category of unpointed spaces). The involution is the one defined in the previous section, and by abuse of notation we still denote it by $w$.\end{example}

We recall that the functor $\Omega^{\infty}_I\colon \Sp\to \Top_\ast$ on non-equivariant spectra is lax-monoidal with respect to the smash product of orthogonal spectra and the cartesian product of pointed spaces, and therefore it preserves associative monoids. The equivariant lift $\Omega^{\infty}_I\colon \Sp^{\mathbb{Z}/2}\to \Top^{\mathbb{Z}/2}_\ast$ fails to be lax symmetric monoidal. This is because the natural transformation
\[
\phi\colon (\Omega^{\infty}_IE)\times(\Omega^{\infty}_IF)\to  \hocolim_{I\times I}\Omega^{i+j}(E_i\wedge F_j)\to  \hocolim_{I\times I}\Omega^{i+j}(E\wedge F)_{i+j}\stackrel{+_\ast}{\to}  \hocolim_{I}\Omega^{i}(E\wedge F)_{i}
\]
uses the disjoint union functor $+ \colon I\times I\to I$ , which is not strictly equivariant. This functor does however satisfy a compatibility condition with the $\mathbb{Z}/2$-action on $I$, namely
\[
\omega(\alpha+\beta)=\omega(\beta)+\omega(\alpha)
\]
for every pair of morphisms $\alpha,\beta$ in $I$, and this allows us describe its monoidal properties. Given a monoid $(M,\mu)$ in a symmetric monoidal category $(\mathscr{C},\otimes)$, we let $M^{op}$ denote the object $M$ equipped with the multiplication
\[
M\otimes M\stackrel{\tau}{\longrightarrow}M\otimes M\stackrel{\mu}{\longrightarrow}M
\]
where $\tau$ is the symmetry isomorphism of the symmetric monoidal structure.
\begin{defn}\label{defmonanti}
A monoid with anti-involution in a symmetric monoidal category $(\mathscr{C},\otimes)$ is a monoid $M$ in $\mathscr{C}$ equipped with a morphism of monoids $w\colon M^{op}\to M$ which satisfies $w^2=\id$.
\end{defn}

\begin{prop}\label{antisymmon}
Let $E$ and $F$ be orthogonal $\mathbb{Z}/2$-spectra with $\Z/2$-action maps $w$ and $w'$, respectively. The diagram of spaces
\[
\xymatrix{
(\Omega^{\infty}_IE)\times(\Omega^{\infty}_IF)\ar[d]_{w\times w'}\ar[r]^-{\phi}&\Omega^{\infty}_I(E\wedge F)\ar[d]^{w\wedge w'}\\
(\Omega^{\infty}_IE)\times(\Omega^{\infty}_IF)\ar[d]_{\tau}&\Omega^{\infty}_I(E\wedge F)\ar[d]^{\Omega^{\infty}_I(\tau)}\\
(\Omega^{\infty}_IF)\times(\Omega^{\infty}_IE)\ar[r]_-{\phi}&\Omega^{\infty}_I(F\wedge E)
}
\]
commutes, where $w$, $w'$ and $w \wedge w'$ denote the induced $\Z/2$-action maps on the $\Omega^{\infty}_I$ construction.  In particular, the functor $\Omega^{\infty}_I\colon \Sp^{\mathbb{Z}/2}\to \Top^{\mathbb{Z}/2}_\ast$ preserves monoids with anti-involution.
\end{prop}

\begin{proof}
We use the Bousfield-Kan formula to express the homotopy colimits as geometric realizations of simplicial spaces, and we show that the diagram of the statement commutes simplicially. An $n$-simplex in the simplicial space realizing to $\Omega^{\infty}_IE$ is a pair $(\underline{i},x)$ where $\underline{i}=(i_0\to\dots\to i_n)$ is an $n$-simplex of the nerve of $I$, and $x\in \Omega^{i_0}E_{i_0}$. The map $\phi$ is then defined by
\[
\phi((\underline{i},x),(\underline{j},y))=(\underline{i}+\underline{j},\iota(x\wedge y))
\]
where $x\wedge y\colon S^{i_0+j_0}\to E_{i_0}\wedge F_{j_0}$ is the smash of the maps $x$ and $y$, $\underline{i}+\underline{j}$ is the value of the simplicial map $+\colon NI\times NI\to NI$ induced by the disjoint union functor $+$ (where $N$ stands for the nerve), and $\iota\colon E_i\wedge F_j\to (E\wedge F)_{i+j}$ is the canonical map. The upper composite of the diagram is then
\[
(\tau\circ (w\wedge w')\circ \phi)((\underline{i},x),(\underline{j},y))=(\omega(\underline{i}+\underline{j}),\iota(w'(y)\wedge w(x))).
\]
This follows from the definitions using the properties of the canonical map $\iota\colon E_i\wedge F_j\to (E\wedge F)_{i+j}$ and from the fact that the identity
\[\chi_{i,j}\tau_{i+j}=\tau_i+\tau_j\]
holds, where $\tau_k\in \Sigma_k$ is the permutation that reverses the order on $\{1,\dots, k\}$ and $\chi_{i,j}$ is the block permutation in $\Sigma_{i+j}$ that swaps the first block of size $i$ with the last block of size $j$.
The lower composite is
\[
(\phi\circ\tau\circ (w\times w'))((\underline{i},x),(\underline{j},y))=\phi((\omega(\underline{j}),w'(y)),(\omega(\underline{i}),w(x)))=(\omega(\underline{j})+\omega(\underline{i}),\iota(w'(y)\wedge w(x))).
\]
These agree because $\omega(\alpha+\beta)=\omega(\beta)+\omega(\alpha)$. Now suppose that $E$ is a monoid with anti-involution, and consider the diagram
\[
\xymatrix{
(\Omega^{\infty}_IE)\times(\Omega^{\infty}_IE)\ar[d]_{w\times w}\ar[r]^-{\phi}&\Omega^{\infty}_I(E\wedge E)\ar[d]^{w\wedge w}\ar[r]^{\mu}&\Omega^{\infty}_IE\ar[dd]^{w}\\
(\Omega^{\infty}_IE)\times(\Omega^{\infty}_IE)\ar[d]_{\tau}&\Omega^{\infty}_I(E\wedge E)\ar[d]^{\tau}\\
(\Omega^{\infty}_IE)\times(\Omega^{\infty}_IE)\ar[r]_-{\phi}&\Omega^{\infty}_I(E\wedge E)\ar[r]_-{\mu}&\Omega^{\infty}_IE
}
\]
The left square commutes by the previous argument, and the right square commutes because $E$ is a monoid with anti-involution. More precisely,  the right hand diagram commutes by applying the (non-equivariant) functor $\Omega^{\infty}_I$ to the diagram which exhibits $E$ as a ring spectrum with anti-involution and by observing that the conjugation by $\tau_i$ is compatible with the map $\mu \colon E \wedge E \to E$, by naturality of the conjugation by $\tau_i$. Thus the outer rectangle also commutes, showing that the space $\Omega^{\infty}_IE$ with its involution $w$ and the monoid structure $\mu\circ\phi$ is a monoid with anti-involution in the category of spaces.
\end{proof}

\begin{rem} There is another, more conceptual explanation of why $\Omega^{\infty}_I$ preserves monoids with anti-involution. The functor $\Omega_I$ from orthogonal spectra to $I$-spaces is lax symmetric monoidal, where the category of $I$-spaces is equipped with the Day convolution product \cite{SagSchlichtdiag}. Thus a ring spectrum $E$ with an anti-involution $w$ gives an $I$-space monoid $\Omega_I E$ with an anti-involution $\Omega_I (w)$.  After passing to homotopy colimits this gives a map
\[ \hocolim_I \Omega_I(w)  \colon \hocolim_I  ((\Omega_I E)^{op} )  \to  \hocolim_I  (\Omega_I E).\]
The homotopy colimit functor from $I$-spaces to spaces is not lax symmetric monoidal but just lax monoidal. Thus, given an $I$-space monoid $X$, there is no immediate relation between $\hocolim_I (X^{op})$ and $(\hocolim_I X)^{op}$. However, the involution $\omega\colon I \to I$ provides such an identification.  We recall that this involution satisfies $\omega(\alpha+\beta)= \omega(\beta)+\omega(\alpha)$, or in other words it is a strictly monoidal isomorphism of monoidal categories $\omega \colon (I, +)^{op} \to (I, +)$, where $ (I, +)^{op}$ is the category $I$ with the monoidal structure $n+^{op}m:=m+n$. Moreover the $(I,+)$-space monoid $X$ defines an $(I,+)^{op}$-space monoid $\hat{X}$, which is $X$ as a functor and with the monoid structure
\[\hat{X}_n \times \hat{X}_m = X_n \times X_m \cong  X_m \times X_n \to X_{m+n}= \hat{X}_{n+^{op}m}.\]
It follows from the definitions that there are isomorphisms of monoids
\[ (\hocolim_{I} X)^{op}=(\hocolim_{(I,+)} X)^{op} \cong \hocolim_{(I,+)^{op}} (\hat{X})^{op}\stackrel{\omega_\ast}{\longrightarrow} \hocolim_{(I,+)} ((\hat{X})^{op}\circ\omega)\cong\hocolim_{(I,+)}X^{op},\]
where the monoidal products in the indexing refer to the induced monoid structures on the Bousfield-Kan construction over $I$. The last isomorphism is induced by the isomorphism of $I$-space monoids
$(\hat{X})^{op} \circ \omega\cong X^{op}$, given by the isomorphisms $X(\tau_i) \colon X_i \to X_i$. It is worth remarking that the morphism $X(\tau_i) \colon X_i \to X_i$ defines a $\Z/2$-diagram structure on any $I$-space $X$ (not necessarily a monoid). By inspection, one sees that the action coming from this canonical $\Z/2$-diagram structure combined with the involution of $E$ is exactly the involution of $\Omega^{\infty}_IE$.
\end{rem}

The aim of the remainder of this section is to prove that $\Omega^{\infty}_I$ has the equivariant homotopy type of the genuine infinite loop space functor.
\begin{theorem}\label{genuineloop}
Let $E$ be an orthogonal $\mathbb{Z}/2$-spectrum. There is a natural equivalence of $\mathbb{Z}/2$-spaces
\[
\Omega^{\infty \rho}E:=\hocolim_{n\in \mathbb{N}}\Omega^{n\rho}E_{n\rho}\stackrel{\simeq}{\longrightarrow}\Omega^{\infty}_IE
\]
where $\rho$ denotes the regular representation of $\mathbb{Z}/2$. In particular if $A$ is a ring spectrum with anti-involution, then $\Omega^{\infty}_IA$ is a monoid with anti-involution model for $\Omega^{\infty \rho}A$.
\end{theorem}

As in the classical theory of $I$-spaces (see e.g. \cite[2.2.9]{Sh00}) Theorem \ref{genuineloop} will follow from a comparison of the homotopy colimit of a real $I$-space with the homotopy colimit of its restriction to the subcategory of natural numbers $\mathbb{N}$. The difference between these homotopy colimits is measured by an equivariant version of the injection monoid.

Given a natural number $n$, we write $[-n,n]=\{-n,\dots,-1,0,1,\dots n\}$ for the symmetric interval with $2n+1$ elements. We will also write $\underline{n}$ for the set $\{1,\cdots, n\}$. For convenience, we will systematically identify $\underline{2n+1}$ in the order preserving way with $[-n,n]$ so that the involution on the morphisms of $I$ becomes conjugation by $-1$. We let $\iota\colon\mathbb{N}\to I$ be the functor that sends $n$ to $2n+1$, and the unique morphism $n\leq m$ to the canonical inclusion $[-n,n]\subset [-m,m]$. Clearly $\iota$ is $\Z/2$-equivariant with respect to the trivial involution on $\mathbb{N}$, thus providing a $\Z/2$-equivariant map
\[
\iota_\ast\colon \hocolim_{\mathbb{N}}\iota^\ast X\longrightarrow\hocolim_{I}X
\]
for every real $I$-space $X\colon I\to \Top$.
Now let $\mathcal{M}R=\Inj(\mathbb{Z},\mathbb{Z})$ be the monoid of self injections of the integers, and $\mathcal{M}R^{\Z/2}$ the submonoid of $\mathbb{Z}/2$-equivariant maps with respect to the involution on $\Z$ given by multiplication by $-1$. For every element $f$ of $\mathcal{M}R$ we define
\[f\cdot n:=\max |f([-n,n])|,\]
which is a map of posets $f\colon (\mathbb{N},\leq)\to (\mathbb{N},\leq)$.
The injection $f\colon[-n,n]\to [-(f\cdot n),f\cdot n]$ induces a  map
\[f\cdot (-)\colon \hocolim_{\mathbb{N}}\iota^\ast X\xrightarrow{X(f)}\hocolim_{\mathbb{N}}f^\ast\iota^\ast X\stackrel{f_\ast}{\longrightarrow}\hocolim_{\mathbb{N}}\iota^\ast X,\]
which is equivariant if $f$ lies in $\mathcal{M}R^{\Z/2}$.
\begin{prop}\label{cofinal}
Let $X\colon I\to \Top$ be a real $I$-space whose underlying $I$-space is semi-stable (see e.g. \cite[\S 2.5]{SagSchCpletion}). Suppose that for every injection $f$ in $\mathcal{M}R^{\Z/2}$ the map $f\cdot(-)$
is an equivalence on $\Z/2$-fixed points. Then the canonical map
\[\iota_\ast\colon \hocolim_{\mathbb{N}}\iota^\ast X\stackrel{\simeq}{\longrightarrow}  \hocolim_{I}X\]
is an equivalence of $\Z/2$-spaces.
\end{prop}

\begin{proof}
We start by showing that $\iota_\ast$ is an equivalence on fixed points, by adapting the argument of \cite[2.2.9]{Sh00}. The idea of the proof is to replace $\mathbb{N}$ by a category $I^o/\omega$ with an action of $\mathcal{M}R^{\mathbb{Z}/2}$, and the fixed points category $I^{\mathbb{Z}/2}$ with the Grothendieck construction $\mathcal{M}R^{\mathbb{Z}/2}\wr(I^o/\omega)$.

We let $I^o/\omega$ be the category whose objects are the pairs $(n,\alpha)$, where $n$ is a non-negative integer and $\alpha\colon [-n,n]\rightarrowtail \mathbb{Z}$ is an equivariant injective map. A morphism $(n,\alpha)\to (n',\alpha')$ in $I^o/\omega$ is an injective equivariant map $\beta\colon [-n,n]\rightarrowtail[-n',n']$ such that $\alpha=\alpha'\circ\beta$. The category $I^o/\omega$ has a strict left action of $\mathcal{M}R^{\mathbb{Z}/2}$, which is given on objects by $f(n,\alpha) = (n, f \circ \alpha)$ and on morphisms by $f(\beta)=\beta$. The functor $\iota\colon \mathbb{N}\to I^{\mathbb{Z}/2}$ factors as
\[
\xymatrix@C=70pt@R=15pt{
\mathbb{N}\ar[r]^-\iota\ar[d]_\gamma& I^{\mathbb{Z}/2}\\
I^o/\omega\ar[r]_-{\delta}&\mathcal{M}R^{\mathbb{Z}/2}\wr (I^o/\omega)\rlap{\ .}\ar[u]_\pi
}
\]
Here $\gamma$ sends an object $n$ to the inclusion $[-n,n]\subset \mathbb{Z}$, and the morphism $n\leq m$ to the inclusion $[-n,n]\subset [-m,m]$. The functor $\delta$ is the canonical inclusion of the fiber over the unique object of $\mathcal{M}R^{\mathbb{Z}/2}$ and $\pi$ sends an object $(n,\alpha)$ to $2n+1$. A morphism $(n,\alpha)\to (n',\alpha')$ in the Grothendieck construction is a pair $(f,\beta)$ where $f\in \mathcal{M}R^{\mathbb{Z}/2}$ and $\beta\colon [-n,n]\rightarrowtail [-n',n']$ is equivariant and such that $f\circ\alpha=\alpha'\circ\beta$. This is sent to $\beta$ by $\pi$.

We prove in Lemma \ref{gammaandpi} below that $\gamma$ and $\pi$ are homotopy right cofinal. Thus the vertical maps in the commutative diagram
\[
\xymatrix{\displaystyle
\hocolim_{\mathbb{N}}\iota^\ast X^{\mathbb{Z}/2}\ar[r]^-{\iota_\ast}\ar[d]_{\gamma_\ast}^{\simeq}
&
\displaystyle
\hocolim_{I^{\mathbb{Z}/2}} X^{\mathbb{Z}/2}\rlap{$\cong (\hocolim_{I} X)^{\mathbb{Z}/2}$}
\\
\displaystyle
\hocolim_{I^o/\omega}\delta^\ast \pi^\ast X^{\mathbb{Z}/2}\ar[r]_-{\delta_\ast}
&
\displaystyle
\hocolim_{\mathcal{M}R^{\mathbb{Z}/2}\wr(I^o/\omega)}\pi^\ast X^{\mathbb{Z}/2}\ar[u]_{\pi_\ast}^{\simeq}
\ar[r]
&
B\mathcal{M}R^{\mathbb{Z}/2}\simeq\ast
}
\]
are weak equivalences. The classifying space of the injection monoid $B\mathcal{M}R^{\mathbb{Z}/2}$ is contractible by Lemma \ref{BMR}. The bottom sequence is equivalent to the sequence
\[
\hocolim_{I^o/\omega}\delta^\ast \pi^\ast X^{\mathbb{Z}/2}\stackrel{}{\longrightarrow}  \hocolim_{\mathcal{M}R^{\mathbb{Z}/2}}(\hocolim_{I^o/\omega} \delta^\ast \pi^\ast X^{\mathbb{Z}/2})\longrightarrow B\mathcal{M}R^{\mathbb{Z}/2}\simeq \ast.
\]
We now show that $\mathcal{M}R^{\mathbb{Z}/2}$ acts by equivalences on $\hocolim_{I^o/\omega} \delta\ast \pi^\ast X^{\mathbb{Z}/2}$. The action of an element $f$ of $\mathcal{M}R^{\mathbb{Z}/2}$ is by definition the bottom horizontal composite of the diagram
\[
\xymatrix@C=38pt{
\displaystyle
\hocolim_{\mathbb{N}}\gamma^\ast\delta^\ast\pi^\ast X^{\mathbb{Z}/2}\ar@/^5ex/[rrr]^-{f\cdot(-)}_-\simeq\ar[rr]^-{X(f)}\ar[dd]^-{\gamma_\ast}_\simeq\ar[drr]_-{X(\pi(f,\id))|_{\gamma}}
&&\displaystyle
\hocolim_{\mathbb{N}}f^\ast\gamma^\ast\delta^\ast\pi^\ast X^{\mathbb{Z}/2}\ar[r]^{f_{\ast}}\ar[ddr]^{(\gamma f)_\ast}
&\displaystyle
\hocolim_{\mathbb{N}}\gamma^\ast\delta^\ast\pi^\ast X^{\mathbb{Z}/2}\ar[dd]^-{\gamma_\ast}_\simeq
\\
&&\displaystyle
\hocolim_{\mathbb{N}}\gamma^\ast f^\ast\delta^\ast\pi^\ast X^{\mathbb{Z}/2}\ar[d]^{\gamma_\ast}\ar[u]^{X(\pi\delta(\xi))}\ar[dr]_-{(f\gamma)_\ast}
\\
\displaystyle
\hocolim_{I^o/\omega}\delta^\ast\pi^\ast X^{\mathbb{Z}/2}\ar[rr]_-{X(\pi(f,\id))}
&&\displaystyle
\hocolim_{I^o/\omega}f^\ast\delta^\ast\pi^\ast X^{\mathbb{Z}/2}\ar[r]_-{f_{\ast}}
&\displaystyle
\hocolim_{I^o/\omega}\delta^\ast\pi^\ast X^{\mathbb{Z}/2}
}
\]
and the composite of the top row is an equivalence by assumption. Here the bottom left horizontal map 
is induced by the morphism $(f,\id)\colon \alpha\to f\alpha$ of $\mathcal{M}R^{\Z/2}\wr(I^o/\omega)$. The upward map in the middle column is induced by the natural transformation
\[
\xi_n\colon f\cdot \gamma(n)=f\cdot(n,[-n,n]\subset \mathbb{Z})=(n,f|_{[-n,n]})\xrightarrow{f|_{[-n,n]}}(f\cdot n,[-f\cdot n,f\cdot n]\subset \mathbb{Z})=\gamma(f\cdot n)
\]
in $I^o/\omega$. Thus $(\gamma f)_\ast\circ X(\pi\delta(\xi))$ is homotopic to $(f\gamma)_\ast$. The rest of the diagram commutes strictly. A version of Quillen's theorem B \cite[page 98]{Quillen} then applies to show that the sequence above is a fiber sequence, showing that $\iota_\ast$ is an equivalence on fixed points.

We must show that $\iota_\ast$ is a non-equivariant equivalence. Let $c\colon\mathbb{N}\to I$ be the standard inclusion that sends $n\leq m$ to the inclusion $\underline{n}\subset\underline{m}$. Since the $I$-space $X$ is semi-stable, the map
\[c_\ast\colon \hocolim_{\mathbb{N}}c^\ast X\longrightarrow \hocolim_{I}X\]
is a non-equivariant equivalence by \cite[Prop. 2.10]{SagSchCpletion}.
Thus it remains to compare the maps $c_\ast$ and $\iota_\ast$. Let $\phi\colon \mathbb{N}\to \mathbb{Z}$ be the bijection that sends $2i+1$ to $-i$ and $2i$ to $i$. The restriction of $\phi$ to the subset $\underline{n}$ defines a natural transformation
\[\lambda_n:=\phi|_{\underline{n}}\colon c(n)=n\longrightarrow 2n+1=\iota(n),\]
and this gives rise to a homotopy commutative diagram
\[
\xymatrix@C=50pt{
\hocolim_{\mathbb{N}}c^\ast X\ar[r]^-{c_\ast}_-\simeq\ar[d]_{\lambda}
&
\hocolim_{I}X\\
\hocolim_{\mathbb{N}}\iota^\ast X\ar[ur]_{\iota_\ast}\rlap{\ .}
}
\]
We now show that $\lambda$ is an equivalence on homotopy colimits. The restriction $\phi_{n}^{-1}$ of $\phi^{-1}\colon \mathbb{Z}\to \mathbb{N}$ to the interval $[-n,n]$ defines a commutative diagram
\[\xymatrix{
\dots\ar[r]
&
X_n\ar[r]^-c\ar[d]_-{\lambda_n}
&
X_{n+1}\ar[r]^-c
&
\dots\ar[r]^-c
&
X_{2n+1}\ar[r]^-c\ar[d]^{\lambda_{2n+1}}
&\dots
\\
\dots\ar[r]
&
X_{2n+1}\ar[r]_-\iota\ar[urrr]^-{\phi^{-1}_n}
&
X_{2n+3}\ar[r]_-\iota
&
\dots\ar[r]_-\iota
&
X_{4n+3}\ar[r]_-\iota
&\dots & .
}
\]
The upper triangle commutes by definition, and it is not hard to see that the lower triangle commutes as well. Thus $\phi^{-1}$ provides an ind-inverse for $\lambda$, which is therefore a homeomorphism.
\end{proof}

\begin{lemma}\label{BMR}
Let $\mathcal{M}R_0\subset \mathcal{M}R$ be the $\mathbb{Z}/2$-submonoid of injections that send zero to zero. Both $B\mathcal{M}R_0$ and $B\mathcal{M}R$ are weakly $\mathbb{Z}/2$-contractible.
\end{lemma}

\begin{proof} Our argument is an adaptation of \cite[5.2]{Schwedehtpygps}.
We first consider the monoid $\mathcal MR$ of all injections $\mathbb Z\to \mathbb Z$, and we construct a functor $F\colon\mathcal MR \to \mathcal MR$ defined on each injection $f\colon \mathbb Z \to \mathbb Z$ by
\[F(f)(t) := \begin{cases}2f(\tfrac 12t) & \text{if $t$ is even,}\\ t & \text{if $t$ is odd.}\end{cases}\]
The injective map $\theta(t):=2t$ then defines a natural transformation $\theta\colon Id_{\mathcal MR_0} \Rightarrow F$. At the same time, the injective map $\psi(t) :=2t+1$ defines a natural transformation $\psi\colon C_{\id_{\mathbb Z}}\Rightarrow F$ to $F$ from the constant functor that sends every $f$ to the identity map on $\mathbb Z$. The two natural transformations $\theta$ and $\psi$ together give a zigzag that contracts $B\mathcal MR$ non-equivariantly.
We still need to argue that the fixed points $(B\mathcal MR)^{\mathbb Z/2}\simeq B(\mathcal MR^{\mathbb Z/2})$ are contractible as well. Note however that $\mathcal MR^{\mathbb Z/2}$ is the monoid of $\mathbb Z/2$-equivariant injections $\mathbb Z\to \mathbb Z$, and any equivariant map has to send zero to zero, hence $\mathcal MR^{\mathbb Z/2}=\mathcal MR_0^{\mathbb Z/2}$. We will study the fixed points $\mathcal MR_0^{\mathbb Z/2}$ below.

Next we consider the submonoid $\mathcal MR_0$ of injections that send zero to zero. The functor $F$ restricts to an endofunctor on $\mathcal MR_0$, and $\theta\in \mathcal MR_0$ still gives a natural transformation $\theta\colon Id_{\mathcal MR_0} \Rightarrow F$. However, the map $\psi$ does not send zero to zero, so we need a different transformation. We define $\varphi\in \mathcal MR_0$ by
\[\varphi(t) :=\begin{cases}2t-1 & \text{if $t>0$,}\\ 0 & \text{if $t=0$,}\\ 2t+1 &\text{if $t<0$.}\end{cases}\]
Because we are working in $\mathcal MR_0$ where all maps send zero to zero, $\varphi$ defines a natural transformation $\varphi\colon C_{\id_{\mathbb Z}}\Rightarrow F$ in $\mathcal MR_0$. As above, the two natural transformations $\theta$ and $\varphi$ give a zigzag that contracts $B\mathcal MR_0$. Since $F$, $\theta$ and $\varphi$ are $\mathbb Z/2$-equivariant, $B\mathcal MR_0$ is even equivariantly contractible, and therefore the fixed point space $B\mathcal MR_0^{\mathbb Z/2}$ is contractible as well by the same natural transformations $\theta$ and $\varphi$.
\end{proof}

\begin{lemma}\label{gammaandpi}
The functors $\gamma\colon \mathbb{N}\to I^o/\omega$ and $\pi\colon \mathcal{M}R^{\mathbb{Z}/2}\wr(I^o/\omega) \to I^{\Z/2}$ from the proof of Proposition \ref{cofinal} are homotopy right cofinal.
\end{lemma}

\begin{proof}
For each object $(n,\alpha)$ of $I^o/\omega$, the under-category $(n,\alpha)/\gamma$ is isomorphic to the poset of non-negative integers $m$ such that $\alpha([-n,n])\subset [-m,m]$. This has an initial object $\max_{[-n,n]} |\alpha|$, and it is therefore contractible.

The under category $n/\pi$ has objects the triples $(k,\alpha,\gamma)$ where $\alpha\colon [-k,k]\rightarrowtail\Z$ and $\gamma\colon \underline{n}\rightarrowtail \underline{2k+1}$ are equivariant injections.
A morphism $(k,\alpha,\gamma)\to (k',\alpha',\gamma')$ is pair $(f,\beta)$ where $f\in \mathcal{M}R^{\Z/2}$ and $\beta\colon [-k,k]\rightarrowtail [-k',k']$ is equivariant and such that $f\circ\alpha= \alpha'\circ\beta$ and $\gamma'=\beta\gamma$.

First let us suppose that $n=2j+1$ is odd. We show that the inclusion of the endomorphisms of the object $(j,[-j,j]\subset \Z,\id_n)$ into $n/\pi$ is a homotopy retract. The endomorphisms of $(j,[-j,j]\subset \Z,\id_n)$ form the submonoid of $\mathcal{M}R^{\Z/2}$ of those injections that restrict to the identity on the interval $[-j,j]$.
There is a retraction
\[
R\colon n/\pi\longrightarrow End_{(j,[-j,j]\subset \Z,\id_n)}
\]
defined as follows. For every equivariant injection $\alpha\colon [-j,j]\rightarrowtail \Z$ we choose an equivariant bijection $f_{\alpha}\colon\mathbb{Z}\to \mathbb{Z}$ which extends $\alpha$. If $\alpha$ is the canonical inclusion $[-j,j]\subset \Z$, then we choose $f_{\alpha}=\id$.  Now we define $R$ by sending a morphism $(f,\beta)\colon (k,\alpha,\gamma)\to (k',\alpha',\gamma')$ to
\[R(f):=f^{-1}_{\alpha'\gamma'}f f_{\alpha\gamma}.\]
This is a well-defined endomorphism since on the interval $[-j,j]$ we have that
\[f^{-1}_{\alpha'\gamma'}f f_{\alpha\gamma}|_{[-j,j]}=f^{-1}_{\alpha'\gamma'} f\alpha\gamma=f^{-1}_{\alpha'\gamma'} \alpha'\beta\gamma=f^{-1}_{\alpha'\gamma'} \alpha'\gamma'=\id.\]
It is easy to verify that $R$ is a functor and a retraction for the inclusion of the endomorphisms. Moreover, the morphism
\[(f_{\alpha\gamma},\gamma)\colon(j,[-j,j]\subset \Z,\id_n)\longrightarrow (k,\alpha,\gamma)\]
defines a natural transformation from $R$ to the identity. Thus the inclusion of the endomorphisms induces a homotopy equivalence on classifying spaces
\[\ast\simeq B\mathcal{M}R^{\Z/2}\cong BEnd_{(j,[-j,j]\subset \Z,\id_n)}\stackrel{\simeq}{\longrightarrow} B(n/\pi)\]
where the isomorphism between the endomorphisms and $\mathcal{M}R^{\Z/2}$ collapses the interval $[-j,j]$ to zero.

If $n=2j$ is even, we consider the object $(j,[-j,j]\subset\Z,c)$ where $c\colon \underline{2j}\cong [-j,j]\backslash 0\subset  [-j,j]\cong\underline{2j+1}$ is the equivariant inclusion which restricts to the standard inclusion $\underline{j}\to \underline{j+1}$. Similarly to the odd case, the inclusion $End_{(j,[-j,j]\subset\Z,c)}\to n/\pi$ admits a homotopy retraction. The endomorphisms of $(j,[-j,j]\subset\Z,c)$ are again the submonoid of $\mathcal{M}R^{\Z/2}$ of the injections that restrict to the identity on $[-j,j]$.
The retraction is defined in a similar matter, by choosing to extend each equivariant map $\alpha\colon  [-j,j]\backslash 0\rightarrowtail\Z$ to an equivariant bijection on $\mathbb{Z}$. The key observation is that the map $c$ can be extended to the identity of $\mathbb{Z}$. Thus there is an equivariant homotopy equivalence
\[\ast\simeq B\mathcal{M}R^{\Z/2}\cong BEnd_{(j,[-j,j]\subset\Z,c)}\stackrel{\simeq}{\longrightarrow} B(n/\pi).\qedhere\]
\end{proof}

\begin{proof}[Proof of \ref{genuineloop}]
Let $E$ be an orthogonal $\mathbb{Z}/2$-spectrum with associated real $I$-space $\Omega_IE\colon I\to \Top$.
Let us start by identifying the restriction $\iota^\ast \Omega_IE$ along the equivariant inclusion $\iota\colon\mathbb{N}\to I$. The involution on the spaces
\[(\iota^\ast \Omega_IE)_n=\Omega^{2n+1}E_{2n+1}\]
is defined by conjugating the levelwise involution of $E$ with the flip permutations on $S^{2n+1}$ and $E_{2n+1}$. These are precisely the $\Z/2$-spaces $\Omega^{n\rho+1}E_{n\rho+1}$, where $\rho$ is the regular representation of $\mathbb{Z}/2$. Therefore suspension by the trivial representation defines a weak equivalence
\[
\Omega^{\infty \rho}E\stackrel{\sim}{\longrightarrow} \hocolim_{n\in\mathbb{N}}\Omega^{n\rho+1}E_{n\rho+1}=\hocolim_{\mathbb{N}}\iota^\ast \Omega_IE.
\]
It remains to show that $\iota_\ast\colon \hocolim_{\mathbb{N}}\iota^\ast \Omega_IE\to \hocolim_{I}\Omega_IE$ is an equivariant equivalence.
We will prove this using Proposition \ref{cofinal}. The underlying $I$-space is equivalent to the $I$-space associated to the orthogonal spectrum $E$, and it is therefore semi-stable. An equivariant injection $f\in\mathcal{M}R^{\Z/2}$ acts trivially on the homotopy groups of $(\hocolim_{n\in\mathbb{N}}\Omega^{n\rho+1}E_{n\rho+1})^{\mathbb{Z}/2}$ at the canonical basepoint. This is because every orthogonal $\mathbb{Z}/2$-spectrum is semi-stable, see \cite[3.4]{Schwede} for a proof. To show that $f$ acts by equivalences it therefore suffices to show that it acts by a loop map. Let us recall that $f\cdot n$ is the smallest integer such that $f([-n,n])\subset [-f\cdot n,f\cdot n]$. The action of $f$ is then the top row of the commutative diagram
\[
\xymatrix@C=35pt{
\displaystyle\hocolim_{n\in\mathbb{N}}\Omega^{n\rho+1}E_{n\rho+1}\ar[r]^-{X(f)}\ar[d]_-{\simeq}
&
\displaystyle\hocolim_{n\in\mathbb{N}}\Omega^{(f\cdot n)\rho+1}E_{(f\cdot n)\rho+1}\ar[r]^-{f_\ast}\ar[d]_-{\simeq}
&
\displaystyle\hocolim_{n\in\mathbb{N}}\Omega^{n\rho+1}E_{n\rho+1}\ar[d]_-{\simeq}
\\
\Omega\displaystyle\hocolim_{n\in\mathbb{N}}\Omega^{n\rho}E_{n\rho+1}\ar[r]^-{\Omega X(f_0)}
&
\Omega\displaystyle\hocolim_{n\in\mathbb{N}}\Omega^{(f\cdot n)\rho}E_{(f\cdot n)\rho+1}\ar[r]^-{\Omega(f_0)_\ast}
&
\Omega\displaystyle\hocolim_{n\in\mathbb{N}}\Omega^{n\rho}E_{n\rho+1}\rlap{\ ,}
}
\]
where $f_0\colon [-n,n]-0\to [-f\cdot n,f\cdot n]-0$ is the restriction of $f$ (we observe that since $f$ is equivariant it sends zero to zero).
\end{proof}

\begin{rem}
A statement similar to \ref{genuineloop} holds for symmetric  $\mathbb{Z}/2$-spectra if one assumes suitable semi-stability conditions.
\end{rem}

\section{Real topological Hochschild homology}\label{secTHR}

\subsection{Ring spectra with anti-involution}

A ring spectrum with anti-involution is a genuine $\Z/2$-equivariant spectrum with a compatible multiplicative structure. It is a direct generalization of the notions of discrete and simplicial rings with anti-involution \cite{BF}, and a spectral version of ``Hermitian Mackey functors'' (see \ref{defHermMackey}).

\begin{defn} A ring spectrum with anti-involution is a monoid with anti-involution in the symmetric monoidal category of orthogonal spectra and smash product (see Definition \ref{defmonanti}). Explicitly, this is a pair $(A,w)$ consisting of an orthogonal ring spectrum $A$ and a morphism of orthogonal ring spectra $w\colon A^{op}\to A$ such that $w^2=\id$.
\end{defn}

We remark that since the involution on $A$ is strict, the map $w$ defines a $\Z/2$-action on the underlying orthogonal spectrum of $A$, and thus it determines a genuine $\Z/2$-equivariant homotopy type. The constructions of this paper will take this underlying genuine $\Z/2$-equivariant homotopy type into account.

\begin{defn}
A morphism of ring spectra with anti-involution $f\colon (A,w)\to (B,\sigma)$ is a morphism of orthogonal ring spectra $f\colon A\to B$ which commutes strictly with the involutions.
We say that $f$ is a stable $\Z/2$-equivalence if the underlying map of $\Z/2$-spectra is an equivalence in the stable model structure of orthogonal $\Z/2$-spectra induced by a complete $\Z/2$-universe.
\end{defn}

\begin{rem}\label{operad}
The resulting category of ring spectra with anti-involution can be described as a category of algebras over a certain equivariant version of the associative operad $Ass$. Let $Ass^{\sigma}$ be the symmetric operad in the category of $\Z/2$-sets whose underlying operad of sets is $Ass$, and with the involution on $Ass_n=\Sigma_n$ defined by sending $\sigma$ to $\tau_n\circ\sigma$, where $\tau_n$ is the permutation of $\{1,\dots,n\}$ that reverses the order. Then the category of $Ass^{\sigma}$-algebras in $\Z/2$-spectra is isomorphic to the category of ring spectra with anti-involution. Moreover $\Sigma_n$ with this involution is a cofibrant $\Z/2\times\Sigma_n$-space with respect to the family of graph subgroups of $\Z/2\times\Sigma_n$ (see \cite{BP}). Thus the theory of algebras over $Ass^{\sigma}$ is homotopically meaningful.

We further observe that there is a map of symmetric operads of $\mathbb{Z}/2$-spaces $Ass^{\sigma}\to E^{\mathbb{Z}/2}_{\infty}$, where $E^{\mathbb{Z}/2}_{\infty}$ is the model for the $\mathbb{Z}/2$-genuine $E_\infty$-operad of \cite{MonaEG}. The $n$-th space of $E^{\mathbb{Z}/2}_{\infty}$ is the classifying space of the functor category $\underline{Fun}(\mathcal{E}\Z/2,\mathcal{E}\Sigma_n)$, where $\mathcal{E}G=G\wr G$ is the translation groupoid of the left $G$-action on the finite group $G$. The map
\[Ass^{\sigma}_n\longrightarrow \underline{Fun}(\mathcal{E}\Z/2,\mathcal{E}\Sigma_n)\]
sends an element $\sigma\in \Sigma_n$ to the functor $\mathcal{E}\Z/2\to\mathcal{E}\Sigma_n$ that sends the unit $1$ to $\sigma$ and $-1$ to $\tau_n\circ\sigma$. Thus genuinely commutative $\Z/2$-equivariant orthogonal ring spectra give rise to ring spectra with anti-involution by forgetting structure.
\end{rem}

\begin{example}
\begin{enumerate}
\item[i)] Let $R$ be a discrete or simplicial ring with an anti-involution $w\colon R^{op}\to R$. The Dold-Thom construction (see e.g.\cite[Example 2.13]{Schwede}) of the Eilenberg-MacLane spectrum $H \colon \Ab \to \Sp$ is lax symmetric monoidal, and hence $HR$ is a ring spectrum with anti-involution. The underlying $\Z/2$-spectrum is the Eilenberg-MacLane spectrum of the $\Z/2$-Mackey functor with underlying abelian group $R$ and fixed points group $R^{\Z/2}$. Throughout the paper, we will always consider the Eilenberg-MacLane spectrum $HR$ as a ring spectrum with anti-involution in this manner.
\item[ii)] Any $\Z/2$-equivariant commutative orthogonal ring spectrum is a ring spectrum with anti-involution, for example the Eilenberg-MacLane spectrum of a Tambara functor.
\item[iii)] The suspension spectrum $\mathbb{S}[G]:=\mathbb{S}\wedge G_+$ of a topological group $G$ is an orthogonal ring spectrum, where the multiplication is defined from the multiplication on $G$. The inversion map $\iota\colon G^{op}\to G$ defines anti-involution on $\mathbb{S}[G]$.
\end{enumerate}
\end{example}

Let $(A,w)$ be a ring spectrum with anti-involution and $M$ an $A$-bimodule. We let $M^{op}$ be the $A$-bimodule with left module structure
\[
A\wedge M\stackrel{\tau}{\longrightarrow}M\wedge A\stackrel{\id\wedge w}{\longrightarrow}M\wedge A\stackrel{\mu_r}{\longrightarrow}M
\]
where $\tau$ is the symmetry isomorphism and $\mu_r$ is the right $A$-module action, and with right $A$-module structure
\[
M\wedge A\stackrel{\tau}{\longrightarrow}A\wedge M\stackrel{w\wedge\id}{\longrightarrow}A\wedge M\stackrel{\mu_l}{\longrightarrow}M.
\]
where $\mu_l$ is the left $A$-module action.

\begin{defn}
Let $(A,w)$ be a ring spectrum with anti-involution. An $(A,w)$-bimodule is a pair $(M,j)$ of an $A$-bimodule $M$ and a map of $A$-bimodules $j\colon M^{op}\to M$ that satisfies $j^2=\id$.

A morphism from an $(A,w)$-bimodule $(M,j)$ to a $(B,\sigma)$-bimodule $(N,k)$ is a pair $(f,\phi)$ where $f\colon (A,w)\to (B,\sigma)$ is a morphism of ring spectra with anti-involution and $\phi\colon M\to f^{\ast}N$ is a map of $A$-bimodules which commutes with the involutions $j$ and $k$ (where $f^\ast$ denotes restriction of scalars).
We say that $(f,\phi)$ is a stable $\Z/2$-equivalence if both $f$ and $\phi$ are stable $\Z/2$-equivalences of underlying orthogonal $\Z/2$-spectra.
\end{defn}

\begin{defn}\label{defflat}
We say that a ring spectrum with anti-involution $(A,w)$ is flat if its underlying orthogonal $\Z/2$-spectrum is flat (see \S\ref{appendixflat}). Similarly, an $(A,w)$-bimodule $(M,j)$ is flat if its underlying orthogonal $\Z/2$-spectrum is flat.
\end{defn}

\begin{rem}
We will build real topological Hochschild homology out of indexed smash products, and flatness is a convenient condition that makes these constructions homotopically well-behaved. In Proposition \ref{antimodelstructure} we construct a model structure on the category of ring spectra with anti-involution where the equivalences are the stable equivalence of genuine $\Z/2$-spectra. In particular this will show that any ring spectrum with anti-involution can be replaced by a flat one, up to stable $\Z/2$-equivalence.
\end{rem}

\begin{rem}\label{flatreplcomm}
In our calculations of $\THR(\Z)$ and $\THR(\F_p)$ we will need to know that if a ring spectrum with anti-involution happens to be commutative, then a flat replacement can be carried out without loosing commutativity. This can be achieved by taking a cofibrant replacement in the flat positive model structure on $\Z/2$-equivariant commutative orthogonal ring spectra of \cite{Sto} and \cite{BrDuSt}.

In particular, any $\Z/2$-equivariant commutative ring $R$ admits a $\Z/2$-equivariant commutative orthogonal ring spectrum model for its Eilenberg-MacLane spectrum $HR$ (even though the classical construction of the Eilenberg-MacLane spectrum provides a flat $\Z/2$-equivariant \emph{symmetric} spectrum, by \cite[Example 2.7]{Schwedesym}, it is unclear to the authors if it is flat as a $\Z/2$-equivariant \emph{orthogonal} spectrum).
\end{rem}

\subsection{The dihedral Bar construction}\label{secdihedral}

We use the norm construction of \cite{HHR} and \cite{Sto} to show that the cyclic Bar construction on a ring spectrum $A$ inherits the structure of a $\mathbb{Z}/2$-spectrum from an anti-involution on $A$.

For every non-negative integer $k$, we let ${\bf k}=\{1,\dots,k\}$ denote the $\mathbb{Z}/2$-set with the involution that sends $i$ to $k+1-i$. Given a ring spectrum with anti-involution $(A,w)$, we let $A^{\wedge {\bf k}}$ be the corresponding indexed smash product as defined in \cite{HHR}. The following is a straightforward generalization of the involution on the Hochschild complex of a ring with anti-involution of \cite[5.2.1]{Loday}.

\begin{defn}
Let $(M,j)$ be an $(A,w)$-bimodule. The dihedral nerve of $(A,w)$ with coefficients in $(M,j)$ is the real simplicial orthogonal spectrum $N_{\wedge}^{di}(A;M)$ with $k$-simplices
\[N_{\wedge}^{di}(A;M)_k=M\wedge A^{\wedge {\bf k}},\]
and the levelwise involution given by the involution $j$ on $M$ and by the involution on the indexed smash product $A^{\wedge {\bf k}}$. The simplicial structure is the standard simplicial structure on the cyclic nerve (see e.g. \cite{Loday}). We let $B_{\wedge}^{di}(A;M)$ be the orthogonal $\mathbb{Z}/2$-spectrum defined by the geometric realization of $N_{\wedge}^{di}(A;M)$. When $A=M$ we write $B_{\wedge}^{di}A$ for $B_{\wedge}^{di}(A;A)$, or  $B_{\wedge}^{di}(A,w)$ if we want to emphasize the anti-involution.
\end{defn}

\begin{rem}
When $M=A$ the simplicial spectrum $N_{\wedge}^{di}A$ has a dihedral structure, and therefore its realization $B_{\wedge}^{di}A$ is an $S^{1}\rtimes \Z/2$-spectrum. In the present paper we will focus on the $\Z/2$-action and neglect the cyclic structure.
\end{rem}

We note that the dihedral bar construction in orthogonal spectra, together with its na\"{i}ve $\Z/2$-equivariant homotopy type, has been already considered by Kro in \cite{Kroinv}. We will focus on the genuine $\Z/2$-equivariant homotopy type of the dihedral bar construction $B_{\wedge}^{di}(A;M)$.

We now give a homotopical interpretation of the dihedral nerve of a ring spectrum with anti-involution $(A,w)$.
Let $N^{\mathbb{Z}/2}_e A$ be the multiplicative norm of the underlying orthogonal spectrum $A$ \cite{HHR}. We recall that this is the ring spectrum $A\wedge A$ with the componentwise multiplication, and the $\Z/2$-action defined by the flip permutation $\tau$. We note that $N^{\mathbb{Z}/2}_e A$ does not depend on the involution $w$, but that it is isomorphic to the enveloping algebra $A^e:=A \wedge A^{op}$ via the map
\[1\wedge w\colon N^{\mathbb{Z}/2}_e A\stackrel{\cong}{\longrightarrow} A^e.\]
Now let $(M,j)$ be an $(A,w)$-bimodule. The $\Z/2$-spectrum $(M,j)$ has the structures of a left and a right $\Z/2$-equivariant $N^{\mathbb{Z}/2}_e A$-module, defined respectively by the maps
\[
N^{\mathbb{Z}/2}_e A \wedge M=A \wedge A\wedge M\xrightarrow{1\wedge w\wedge 1}A \wedge A^{op}\wedge M\stackrel{1\wedge \tau}{\longrightarrow}A \wedge M\wedge A^{op}\stackrel{\mu}{\longrightarrow}M
\]
and
\[
M\wedge N^{\mathbb{Z}/2}_e A =M \wedge A\wedge A\xrightarrow{1\wedge 1\wedge w}M\wedge A \wedge A^{op}\xrightarrow{(\tau\wedge 1)\circ(1\wedge \tau)}A^{op} \wedge M\wedge A\stackrel{\mu}{\longrightarrow}M,
\]
respectively. In particular the $(A,w)$-bimodule $(A,w)$ becomes a left and a right $N^{\mathbb{Z}/2}_e A$-module.

\begin{prop} \label{dihedral-two-sided} There is a natural isomorphism of simplicial $\Z/2$-equivariant orthogonal spectra
\[
\sd_e N^{di}_{\wedge}(A;M)\cong N_{\wedge}(M, N^{\mathbb{Z}/2}_e A, A)_{\bullet}
\]
between the Segal edgewise subdivision of the dihedral bar construction and the two-sided bar construction
in the category of $\Z/2$-spectra.
\end{prop}

\begin{proof}
 Define for every $k\geq 0$ an isomorphism
\[M \wedge A^{\wedge {\bf 2k+1}} \longrightarrow M \wedge (A\wedge A^{op})^{\wedge k} \wedge A \xrightarrow{1\wedge (1\wedge w)^{\wedge k}\wedge 1} M \wedge (N^{\mathbb{Z}/2}_e A)^{\wedge  k} \wedge A,\]
where the first map is the permutation
\[x_0 \wedge x_1 \wedge \cdots \wedge x_{2k+1} \longmapsto x_0 \wedge (x_1 \wedge x_{2k+1} ) \wedge (x_2 \wedge x_{2k}) \wedge \cdots \wedge (x_k \wedge x_{k+2}) \wedge x_{k+1}.\]
These maps are $\Z/2$-equivariant isomorphisms and they are compatible with the simplicial structure.
\end{proof}

\begin{cor} \label{derived smash} Let $(A,w)$ be a flat ring spectrum with anti-involution and $(M,j)$ a flat $(A,w)$-bimodule. Then there is a stable equivalence of $\Z/2$-spectra
\[B^{di}_{\wedge}(A;M)\simeq M \wedge^{\mathbf{L}}_{N^{\mathbb{Z}/2}_e A} A
\]
where the right-hand term is the derived smash product of left and right $N^{\mathbb{Z}/2}_e A$-modules. In particular the functor $B^{di}_{\wedge}$ preserves stable $\Z/2$-equivalences of flat ring spectra with anti-involution.
\end{cor}

\begin{proof}  Since the unit map $\mathbb{S} \to A$ is an underlying flat cofibration, it follows from \cite[Theorem 3.4.22-23]{Sto} and \cite[Theorem 3.2.14]{BrDuSt} that the unit $\mathbb{S} \to N^{\mathbb{Z}/2}_e A$ is a $\Z/2$-equivariant flat cofibration. In particular, the spectrum $N^{\mathbb{Z}/2}_e A$ is flat. It follows that $B_{\wedge}(M, N^{\mathbb{Z}/2}_e A, A)$ and the derived smash product $M \wedge^{\mathbf{L}}_{N^{\mathbb{Z}/2}_e A} A$ are equivalent, essentially by \cite[Lemma 4.1.9]{Sh00} and \cite[Proposition IX.2.3]{EKMM}, which we recall and reprove in Lemma \ref{barderived} below.
\end{proof}

\begin{lemma}\label{barderived} Let $(\mathcal{C}, \wedge, \mathbb{I})$ be a cofibrantly generated simplicial monoidal model category. Suppose that $R$ is a monoid object in $\mathcal{C}$ such that the unit map $\mathbb{I} \to R$ is a cofibration and $R$ is underlying cofibrant in $\mathcal{C}$. Further, let $M$ be a right $R$-module and $N$ be a left $R$-module and assume that $M$ and $N$ are cofibrant as underlying objects of $\mathcal{C}$. Then the two-sided bar construction $B_{\wedge}(M,R,N)$ is weakly equivalent to the derived smash product $M \wedge^{\mathbf{L}}_RN$.
\end{lemma}


\begin{proof} The bar construction $B_{\wedge}(M,R,N)$ is the geometric realization of the simplicial object $[k] \mapsto B_{\wedge}(M,R,N)_k= M \wedge R^{\wedge k} \wedge N$. This simplicial object is Reedy cofibrant. Indeed, using that $\mathbb{I} \to R$ is a cofibration and $M$ and $N$ are cofibrant, the iterated pushout-product axiom implies that the latching morphism $L_kB_{\wedge}(M,R,N) \to B_{\wedge}(M,R,N)_k$ is a cofibration. The geometric realization $B_{\wedge}(M,R,N)$ is isomorphic to the colimit $\colim_n \sk_n$, where $sk_0=M \wedge N$ and $\sk_{n-1}$ and $\sk_n$ are related by the pushout square
\[\xymatrix{\Delta^n \otimes L_nB_{\wedge}(M,R,N) \coprod_{\partial \Delta^n \otimes L_nB_{\wedge}(M,R,N)} \partial \Delta^n \otimes B_{\wedge}(M,R,N)_n \ar[r] \ar[d] & \sk_{n-1} \ar[d] \\ \Delta^n \otimes B_{\wedge}(M,R,N)_n \ar[r] & \sk_n, }\]
where the left vertical morphism is a cofibration. Under our assumptions this implies that $B_{\wedge}(M,R,N)$ is cofibrant and homotopy invariant.
Now let
\[M^c \stackrel{\sim}{\longrightarrow} M\]
be a cofibrant replacement of $M$ as a right $R$-module. Since $R$ is underlying cofibrant, so is $M^c$. The homotopy invariance of the bar construction gives that $B_{\wedge}(M,R,N)$  and $B_{\wedge}(M^c,R,N)$ are weakly equivalent. On the other hand $B_{\wedge}(M^c,R,N)$ is isomorphic to $M^c \wedge_R B_{\wedge}(R,R,N)$. Using the same skeletal filtration argument as above one sees that in fact $B_{\wedge}(R,R,N)$ is a cofibrant left $R$-module. Moreover, there is an augmentation $B_{\wedge}(R,R,N)_{\bullet} \longrightarrow N$ and this augmented simplicial object has an extra degeneracy in $\mathcal{C}$, coming from the unit map $\mathbb{I} \to R$. Hence we get a weak equivalence
\[B_{\wedge}(R,R,N) \stackrel{\sim}{\longrightarrow} N\]
of left $R$-modules and we conclude that $B_{\wedge}(R,R,N)$ is a cofibrant replacement of $N$. Finally, this implies that
\[B_{\wedge}(M^c,R,N) \cong M^c \wedge_R B_{\wedge}(R,R,N) \]
models the derived smash product $M \wedge^{\mathbf{L}}_RN$ which finishes the proof.
\end{proof}

We will need to ensure that the geometric realization of $N_{\wedge}^{di}(A;M)$ is homotopically well-behaved.

\begin{lemma}\label{properness}
Suppose $(A,w)$ and $(M,j)$ are flat. Then the real simplicial spectrum $N_{\wedge}^{di}(A;M)$ is good.
\end{lemma}

\begin{proof}
Since a flat cofibration of orthogonal $\Z/2$-spectra is a levelwise $h$-cofibration, it is sufficient to show that the degeneracy maps of the subdivision of $N_{\wedge}^{di}(A;M)$ are  flat cofibrations. By the isomorphism of \ref{dihedral-two-sided} these are the degeneracies
\[s_i\colon M\wedge (N^{\Z/2}_eA)^{\wedge k}\wedge A \longrightarrow M\wedge (N^{\Z/2}_eA)^{\wedge k+1}\wedge A\]
of the two-sided Bar construction, which insert the unit $N^{\Z/2}_e\eta\colon \mathbb{S}\cong N^{\Z/2}_e\mathbb{S}\to N^{\Z/2}_eA$ in the $(i+1)$-st component. Since the underlying ring spectrum of $A$ is flat, the unit $\eta\colon \mathbb{S}\to A$ is automatically a flat cofibration of orthogonal spectra, and thus  $N^{\Z/2}_e\eta$ is a flat cofibration of orthogonal $\Z/2$-spectra. Finally, since $(A,w)$ and $(M,j)$ are flat, smashing with them preserves flat cofibrations.
\end{proof}

\subsection{B\"{o}kstedt's model of real topological Hochschild homology}\label{secBok}

Let $I$ be the category of finite sets and injective maps with the $\mathbb{Z}/2$-action $\omega\colon I\to I$ defined in \S\ref{sec:realI}.
The dihedral Bar construction of $I$ is the real simplicial category $N_{\bullet}^{di}I$ with $k$-simplices
\[N^{di}_kI=I^{\times k+1}\]
and the standard faces and degeneracy maps of the cyclic Bar construction for the symmetric monoidal structure $(I,+, 0)$. The real structure on $N_{\bullet}^{di}I$ is defined by the natural transformation
\[
\xymatrix@C=40pt{\omega_k\colon I^{\times k+1}\ar[r]^{\prod_{k+1} \omega}
&
I^{\times k+1}\ar[r]^{1\times \tau_k}
&
I^{\times k+1}
}
\]
where $\tau_k\colon I^{\times {\bf k}}\to I^{\times {\bf k}}$ is induced by the permutation in $\Sigma_k$ that reverses the order. We denote the resulting categories with $\Z/2$-action by $I^{\times 1+{\bf k}}$.
We let $\Delta^{op}\wr N^{di}I$ be the Grothendieck construction of the  $\mathbb{Z}/2$-diagram  $N_{\bullet}^{di}I\colon \Delta^{op}\to Cat$ (see \cite{Thomason}). This category inherits an evident $\mathbb{Z}/2$-action from the  $\mathbb{Z}/2$-structure of $N_{\bullet}^{di}I$ (see \cite[Section 2.5]{Gdiags} for the details).
We remark that for every integer $k\geq 0$ the inclusion $\iota_k\colon I^{\times 1+{\bf k}}\to \Delta^{op}\wr N^{di}I$ is $\mathbb{Z}/2$-equivariant.

\begin{prop}\label{THRofRMCat}
Let $X\colon \Delta^{op}\wr N^{di}I\to\Sp$ be a $\mathbb{Z}/2$-diagram of orthogonal spectra. Then the collection of $\mathbb{Z}/2$-spectra
\[
\THR_k(X)=\hocolim_{I^{1+{\bf k}}}\iota_{k}^\ast X,
\]
together with the simplicial structure induced by $N_{\bullet}^{di}I$, define a real simplicial orthogonal spectrum $\THR_\bullet(X)$. A natural transformation $f\colon X\to Y$ of $\mathbb{Z}/2$-diagrams $\Delta^{op}\wr N^{di}I\to\Sp$ induces a natural map of real simplicial orthogonal spectra
\[f_\ast\colon \THR_\bullet(X)\longrightarrow  \THR_\bullet(Y).\]
\end{prop}

\begin{proof}
The spectra $\THR_\bullet(X)$ admit a simplicial structure by \cite{PatchkSagave}. The map induced by a morphism $\alpha\colon [k]\to [l]$ in $\Delta$ is
\[
\alpha^{\ast}\colon \hocolim_{I^{1+{\bf l}}}\iota_{l}^\ast X\xrightarrow{\lambda_\alpha} \hocolim_{I^{1+{\bf l}}}\alpha^{\ast}\iota_{k}^\ast X\xrightarrow{\alpha_\ast} \hocolim_{I^{1+{\bf k}}}\iota_{k}^\ast X
\]
where $\alpha_\ast$ is the pushforward along the functor $\alpha \colon I^{1+{\bf l}}\to I^{1+{\bf k}}$ from the simplicial structure of $N^{di}_\bullet I$. The first map is the natural transformation $\lambda_\alpha\colon X\circ\iota_l\to X\circ\iota_{k}\circ\alpha$ induced by the morphisms $(\alpha,\id)\colon \underline{i}\to \alpha\underline{i}$ in $\Delta^{op}\wr N^{di}I$, for $\underline{i}\in I^{1+{\bf l}}$.
The levelwise involutions are the maps
\[
w_k\colon \hocolim_{I^{1+{\bf k}}}\iota_{k}^\ast X\xrightarrow{w} \hocolim_{I^{1+{\bf k}}}\omega_{k}^\ast\iota_{k}^\ast X\xrightarrow{(\omega_k)_\ast}  \hocolim_{I^{1+{\bf k}}}\iota_{k}^\ast X
\]
where the first map is induced by the natural transformation $w\colon X\to X\circ \omega$ from the $\mathbb{Z}/2$-structure of $X$, and the second map is the pushforward along the involution $\omega_k\colon I^{1+{\bf k}}\to I^{1+{\bf k}}$.

We verify that $w_k\alpha^{\ast}=\omega(\alpha)^{\ast}w_l$, that is that the outer rectangle in the diagram
\[\xymatrix@C=50pt{
\displaystyle
 \hocolim_{I^{1+{\bf l}}}\iota_{l}^\ast X\ar[d]_{w}\ar[r]^-{\lambda_\alpha}
 &
 \displaystyle
 \hocolim_{I^{1+{\bf l}}}\alpha^{\ast}\iota_{k}^\ast X\ar[r]^-{\alpha_\ast}\ar[d]_{w|_{\alpha}}
 &
 \displaystyle\hocolim_{I^{1+{\bf k}}}\iota_{k}^\ast X\ar[d]^{w}
\\
\displaystyle
\hocolim_{I^{1+{\bf l}}}\omega_{l}^\ast\iota_{l}^\ast X\ar[d]^-{(\omega_l)_\ast}\ar[r]_-{\lambda_{\omega(\alpha)}|_{\omega_l}}
&
\displaystyle
 \hocolim_{I^{1+{\bf l}}}\omega_{l}^\ast(\omega(\alpha))^{\ast}\iota_{k}^\ast X\ar[r]^-{\alpha_\ast}\ar[d]_{(\omega_l)_\ast}
&
\displaystyle
\hocolim_{I^{1+{\bf k}}}\omega_{k}^\ast\iota_{k}^\ast X\ar[d]^-{(\omega_k)_\ast}
\\
\displaystyle\hocolim_{I^{1+{\bf l}}}\iota_{l}^\ast X\ar[r]_{\lambda_{\omega(\alpha)}}
&
\displaystyle
 \hocolim_{I^{1+{\bf l}}}\omega(\alpha)^{\ast}\iota_{k}^\ast X\ar[r]_-{\omega(\alpha)_\ast}
&
\displaystyle
\hocolim_{I^{1+{\bf k}}}\iota_{k}^\ast X
}\]
commutes. The vertical map $w|_\alpha$ and the horizontal map $\alpha_\ast$, respectively in and out of the central term, are well-defined since $\omega_{l}^\ast(\omega(\alpha))^{\ast}=\alpha^\ast\omega^{\ast}_k$. The upper right and bottom left squares commute because of the interaction between the push forward maps on homotopy colimits and the maps induced by natural transformations. The bottom right square commutes because $\omega_k\alpha=\omega(\alpha)\omega_l$ as functors $I^{1+{\bf l}}\to I^{1+{\bf k}}$. The upper left square commutes because $X$ is a $\Z/2$-diagram, and the involution of $\Delta^{op}\wr N^{di}I$ sends the morphism $(\alpha,\id)$ to $(\omega(\alpha),\id)$.
\end{proof}

\begin{rem}\label{ptedvsunptedhocolim}
The homotopy colimit of Proposition \ref{THRofRMCat} is taken in the category of orthogonal spectra, and this is computed levelwise as a homotopy colimit of \textit{pointed} spaces. In order to use the results of \S\ref{sec:realI}, which are about the homotopy type of the homotopy colimit taken in the category of \textit{unpointed}  spaces, we need to observe that if $X\colon I\to \Top_\ast$ is a $\Z/2$-diagram where $X_i$ is a well-pointed $\Z/2$-space for every $i\in I$, then the comparison map
\[
|\coprod_{\underline{i}\in N_\bullet I}X_{i_0}|\longrightarrow|\bigvee_{\underline{i}\in N_\bullet I}X_{i_0}|
\]
from the unpointed homotopy colimit to the pointed homotopy colimit is an equivariant equivalence. Indeed, this map is the cofiber of the $h$-cofibration $|N_\bullet I|\hookrightarrow |\coprod_{\underline{i}\in N_\bullet I}X_{i_0}|$, and the source is equivariantly contractible since $I$ has a $\Z/2$-fixed initial object. We will implicitly use the results of \S\ref{sec:realI} for the pointed homotopy colimit, under this well-pointedness assumption.
\end{rem}

\begin{example}\label{Bokstedtmodel1}
Let $(A,w)$ be an orthogonal spectrum with anti-involution and $(M,j)$ an $(A,w)$-bimodule. There is a $\mathbb{Z}/2$-diagram $\Omega^{\bullet}_I(A;M;\mathbb{S})\colon  \Delta^{op}\wr N^{di}I\to\Sp$ defined on objects by sending $([k],\underline{i})$ to
\[\Omega^{\bullet}_I(A;M;\mathbb{S})_{\underline{i}}:=\Omega^{i_0+\dots+i_k}(\mathbb{S}\wedge M_{i_0}\wedge A_{i_1}\wedge\dots\wedge A_{i_k}),\]
see \cite{PatchkSagave}. The associated simplicial spectrum is B\"{o}kstedt's model for topological Hochschild homology \cite{Bok}.
This diagram has a $\mathbb{Z}/2$-structure $\Omega^{\bullet}_I(A;M;\mathbb{S})_{\underline{i}}\to \Omega^{\bullet}_I(A;M;\mathbb{S})_{\omega_k\underline{i}}$ which postcomposes a loop with the map
\[
\xymatrix@C=40pt{M_{i_0}\wedge A_{i_1}\wedge\dots\wedge A_{i_k}\ar[r]^{j\wedge w^{\wedge {\bf k}}}\ar@{-->}[drr]
&
M_{i_0}\wedge A_{i_1}\wedge\dots\wedge A_{i_k}\ar[r]^{1\wedge \tau_k}
&
M_{i_0}\wedge A_{i_k}\wedge\dots\wedge A_{i_1}\ar[d]^{\tau_{i_0}\wedge\tau_{i_k}\wedge\dots\wedge\tau_{i_1}}
\\
&&M_{i_0}\wedge A_{i_k}\wedge\dots\wedge A_{i_1}
}
\]
and precomposes it with
\[\xymatrix@C=60pt{
S^{i_0}\wedge S^{i_1}\wedge\dots\wedge S^{i_k}\ar[r]^{1\wedge \tau_{k}}&
S^{i_0}\wedge S^{i_k}\wedge\dots\wedge S^{i_1}\ar[r]^{\tau_{i_0}\wedge\tau_{i_k}\wedge\dots\wedge\tau_{i_1}}
&
S^{i_0}\wedge S^{i_k}\wedge\dots\wedge S^{i_1},
}
\]
where $\tau_j$ is the permutation of $\{1,\dots, j\}$ that reverses the order. The associated real simplicial spectrum was introduced by Hesselholt and Madsen in \cite[\S 10]{IbLars}.
\end{example}

\begin{defn}\label{Bokstedtmodel}
Let $(M,j)$ be an $(A,w)$-bimodule. The B\"{o}kstedt model of the real topological Hochschild homology of $A$ with coefficients in $M$ is the geometric realization of the real simplicial spectrum \[\THR_\bullet(A;M):=\THR_\bullet\Omega^{\bullet}_I(A;M;\mathbb{S}),\] which we denote by $\THR(A;M)$. When $A=M$ we write $\THR(A)$ for $\THR(A;A)$, or $\THR(A,w)$ if we want to emphasize the anti-involution.
\end{defn}

\begin{example}\label{middlemodel}
Let $(A,w)$ be an orthogonal ring spectrum with anti-involution and $(M,j)$ an $(A,w)$-bimodule. Consider the $\mathbb{Z}/2$-diagram $\Omega^{\bullet}_I(A;M;\Sh)\colon  \Delta^{op}\wr N^{di}I\to\Sp$ defined by sending  $([k],\underline{i})$ to
\[\Omega^{\bullet}_I(A;M;\Sh)_{\underline{i}}:=\Omega^{i_0+\dots+i_k}(\Sh^{i_0} M\wedge \Sh^{i_1}A\wedge\dots\wedge \Sh^{i_k}A),\]
where $\Sh^{i}X$ denotes the shifted spectrum, given by the sequence of $\mathbb{Z}/2$-spaces $(\Sh^{i}X)_n=X_{i+n}$. This diagram has a $\mathbb{Z}/2$-structure given by postcomposing a loop with the map
\[
\xymatrix@C=40pt{\Sh^{i_0} M\wedge \Sh^{i_1}A\wedge\dots\wedge \Sh^{i_k}A\ar[r]^{j\wedge w^{\wedge {\bf k}}}\ar@{-->}[ddr]
&
\Sh^{i_0} M\wedge \Sh^{i_1}A\wedge\dots\wedge \Sh^{i_k}A\ar[d]^{1\wedge \tau_k}
\\
&\Sh^{i_0} M\wedge \Sh^{i_k}A\wedge\dots\wedge \Sh^{i_1}A\ar[d]^{\tau_{i_0}\wedge\tau_{i_k}\wedge\dots\wedge\tau_{i_1}}
\\
&\Sh^{i_0} M\wedge \Sh^{i_k}A\wedge\dots\wedge \Sh^{i_1}A
}
\]
and precomposing it with
\[\xymatrix@C=60pt{
S^{i_0}\wedge S^{i_1}\wedge\dots\wedge S^{i_k}\ar[r]^{1\wedge \tau_{k}}&
S^{i_0}\wedge S^{i_k}\wedge\dots\wedge S^{i_1}\ar[r]^{\tau_{i_0}\wedge\tau_{i_k}\wedge\dots\wedge\tau_{i_1}}
&
S^{i_0}\wedge S^{i_k}\wedge\dots\wedge S^{i_1}
}.
\]
We will show in Theorem \ref{comparison} that the $\THR$ spectrum associated to $\Omega^{\bullet}_I(A;M;\Sh)$ is equivalent to $\THR(A;M)$.
\end{example}


The B\"{o}kstedt model $\THR(A;M)$ is homotopically better behaved than the dihedral Bar construction $B_{\wedge}^{di}(A;M)$, essentially because it incorporates the derived indexed smash product of $\Z/2$-spectra. We will spend the rest of the section proving the homotopy invariance of this model, as stated in Theorem \ref{invBok} below.
Let $(M,j)$ be an $(A,w)$-module. We say that the $\Z/2$-spectrum $(M,j)$ is levelwise well-pointed if for every finite dimensional $\Z/2$-representation $V$ the pointed $\Z/2$-space $M_V$ obtained by evaluating $M$ at $V$ is well-pointed. We say that $(A,w)$ is levelwise very well-pointed if it is levelwise well-pointed, and if the unit map $S^0\to A_0$ is an $h$-cofibration of $\Z/2$-spaces.

\begin{theorem}\label{invBok}
Let $f\colon (A,w)\to (B,\sigma)$ be a stable $\Z/2$-equivalence of levelwise very well-pointed ring spectra with anti-involution, and $\phi\colon (M,j)\to f^{\ast}(N,l)$ a stable $\Z/2$-equivalence of levelwise well-pointed $(A,w)$-modules. Then the induced map
\[(f,\phi)_\ast\colon \THR(A;M)\longrightarrow \THR(B;N)\]
is a stable equivalence of $\mathbb{Z}/2$-spectra.
\end{theorem}

In order to prove this theorem we need to understand the levelwise homotopy type of the real simplicial spectrum $\THR_\bullet(A;M)$. 

\begin{lemma}\label{goodness}
Let $(M,j)$ be a levelwise well-pointed $(A,w)$-bimodule, and suppose that $(A,w)$ is levelwise very well-pointed. Then $\THR_\bullet (A;M)$ is a good real simplicial spectrum. Moreover, if $(M,j)$ and $(A,w)$  are flat, then $\THR_\bullet\Omega^{\bullet}_I(A;M;\Sh)$ is a good real simplicial spectrum.
\end{lemma}

\begin{proof}
We prove that the composition of degeneracies $s_2s_0\colon \THR_1 (A;M)\to \THR_3 (A;M)$ is a levelwise $h$-cofibration of $\mathbb{Z}/2$-spectra, the argument for the other degeneracies is similar. This map can be written, at a $\Z/2$-representation $V$, as the composite
\[
\xymatrix{
\displaystyle\hocolim_{I\times I}\Omega^{i_0+i_2}(S^V\wedge M_{i_0}\wedge A_{i_2})\ar[r]&\displaystyle \hocolim_{I\times I}\Omega^{i_0+i_2}(S^V\wedge M_{i_0}\wedge A_0\wedge A_{i_2}\wedge A_0)
\ar[d]^-{F_\ast}
\\
&\displaystyle \hocolim_{I\times I\times I\times I}\Omega^{i_0+i_1+i_2+i_3}(S^V\wedge M_{i_0}\wedge A_{i_1}\wedge A_{i_2}\wedge A_{i_3})
}
\]
where the second map is the pushforward by the functor $F\colon I\times I\to I\times I\times I\times I$ given by $F(i_0,i_2) = (i_0,0,i_2,0)$, and the first map is induced by smashing with the unit map $S^0\to A_0$. Since $S^0\to A_0$ is an $h$-cofibration, the $A_i$ and $M_i$ are well-pointed $\Z/2$-spaces, and loop spaces preserve $h$-cofibrations
, the map
\[\Omega^{i_0+i_2}(S^V\wedge M_{i_0}\wedge S^0\wedge A_{i_2}\wedge S^0)\longrightarrow\Omega^{i_0+i_2}(S^V\wedge M_{i_0}\wedge A_0\wedge A_{i_2}\wedge A_0)\]
is a pointwise $h$-cofibration of $\mathbb{Z}/2$-diagrams. That is, it is an equivariant $h$-cofibration with respect to the action of the stabilizer group of the object $(i_0,i_2)\in I\times I$. We show more generally that if $F\colon A\to B$ is an equivariant functor which is faithful and injective on objects, $Y\colon B\to \Top_\ast$ and $X\colon A\to \Top_\ast$ are $\Z/2$-diagrams, and $f\colon X\to F^\ast Y$ is a pointwise $h$-cofibration of pointwise well-pointed $\Z/2$-diagrams, then the maps
\[
\bigvee_{\underline{a}\in N_{k}A}X_{a_0}\xrightarrow{\vee f_{a_o}} \bigvee_{\underline{a}\in N_{k}A}Y_{F(a_0)}\stackrel{F_\ast}{\longrightarrow} \bigvee_{\underline{b}\in N_{k}B}Y_{b_0}
\]
are Reedy cofibration of simplicial $\Z/2$-spaces. The composite will then induce an h-cofibration on geometric realizations between the homotopy colimits, concluding the proof. The Reedy conditions for these maps amount to showing that the maps
\[
\bigvee_{\underline{a}\in (N_{k}A)^{d}}Y_{F(a_0)}\vee \bigvee_{\underline{a}\in (N_{k}A)^{nd}}X_{a_0}\xrightarrow{\id\vee \bigvee f_{a_0}} \bigvee_{\underline{a}\in N_{k}A}Y_{F(a_0)}
\ \ \ \ \ ,\ \ \ \ \ \ 
 \bigvee_{\underline{b}\in (N_{k}B)^{nd}\cup N_{k}F(A)}Y_{b_0}{\longrightarrow} \bigvee_{\underline{b}\in N_{k}B}Y_{b_0}
\]
are $h$-cofibrations of $\Z/2$-spaces, where $(-)^{nd}$ and $(-)^{d}$ denote respectively the subsets of non-degenerate and degenerate simplices. It is easy to verify that the indexed wedge $\bigvee_{\underline{a}} f_{a_0}$ of pointwise $h$-cofibrations is an $h$-cofibration of \textit{pointed} $\Z/2$-spaces. Thus the first map is a pointed $h$-cofibration. The second map the inclusion of a wedge summand, and therefore it is also a pointed $h$-cofibration. Since all the $\Z/2$-spaces involved are well-pointed, these are also unpointed $h$-cofibrations.

An analogous argument shows that $\THR_\bullet\Omega^{\bullet}_I(A;M;\Sh)$ is a good real simplicial spectrum. One will need to observe that the shift functor $\Sh^i$ preserves flat $\Z/2$-spectra. This is done in Appendix \ref{shiftflat}.
\end{proof}

\begin{lemma}\label{ItoN}
Let $(M,j)$ be a levelwise well-pointed $(A,w)$-bimodule, and suppose that $(A,w)$ is levelwise very well-pointed.
The functor $\iota\colon\mathbb{N}\to I$ that sends $n$ to $2n+1$, composed with the diagonal $\Delta\colon I\to I^{\times 1+{\bf k}}$, induces a natural stable equivalence of $\mathbb{Z}/2$-spectra
\[
\hocolim_{\mathbb{N}}\Omega^{n\rho(1+{\bf k})}(\mathbb{S}\wedge M_{n\rho}\wedge A^{\wedge {\bf k}}_{n\rho})\stackrel{\simeq}{\longrightarrow}\THR_k (A;M),
\]
where the involution on ${\bf k}=\{1,\dots,k\}$ reverses the order, and $\rho$ is the regular representation of $\mathbb{Z}/2$.

If $(A,w)$ and $(M,j)$ are moreover flat, the functors $\iota$ and $\Delta$ also induce a natural stable equivalence of $\mathbb{Z}/2$-spectra
\[
\hocolim_{\mathbb{N}}\Omega^{n\rho(1+{\bf k})}(\Sh^{n\rho}M\wedge (\Sh^{n\rho}A)^{\wedge {\bf k}})\stackrel{\simeq}{\longrightarrow}\THR_k\Omega^{\bullet}_I(A;M;\Sh).
\]
\end{lemma}

\begin{proof}
The argument is analogous to the one of \ref{genuineloop}. The first map is the composite of the stabilization with the trivial representation and the map
\[
(\Delta\iota)_\ast\colon \hocolim_{\mathbb{N}}\Omega^{(n\rho+1)(1+{\bf k})}(\mathbb{S}\wedge M_{n\rho+1}\wedge A^{\wedge {\bf k}}_{n\rho+1})\stackrel{}{\longrightarrow}\THR_k (A;M).
\]
We show that $(\Delta\iota)_\ast$ is an equivalence of $\Z/2$-spectra in each simplicial level. Non-equivariantly it is equivalent to the map
\[
c_{\ast}^{\times k+1}\colon\hocolim_{\mathbb{N}^{\times 1+{\bf k}}}\Omega^{n(1+{\bf k})}(\mathbb{S}\wedge M_{n}\wedge A^{\wedge {\bf k}}_{n})\stackrel{\simeq}{\longrightarrow}\THH_k (A;M)
\]
induced by the functor $c\colon \mathbb{N}\to I$ of \ref{cofinal}, which is a level-equivalence of spectra because $A$ and $M$ are orthogonal spectra, and hence semi-stable. We show that the evaluation of $(\Delta\iota)_\ast$ at any $\Z/2$-representation $V$ is an equivalence on fixed points. We treat the case where $k=2q+1$ is odd, the even case is analogous. The fixed points category of $I^{\times 1+{\bf 2q+1}}$ is isomorphic to $I^{\Z/2}\times I^{\times q}\times I^{\Z/2}$, and there is a homotopy commutative diagram
\[
\xymatrix@C=70pt{
\hocolim_{\mathbb{N}^{\times q+2}}(\iota^{\times q+2})^\ast X^{\Z/2}\ar[r]^-{(\iota^{\times 1+{\bf 2q+1}})^{\Z/2}_\ast}
&
\THR_{2q+1} (A;M)^{\Z/2}_V
\\
\hocolim_{\mathbb{N}^{\times q+2}}(\iota\times c^{\times q}\times\iota)^\ast X^{\Z/2}\ar[ur]_-{(\iota\times c^{\times q}\times\iota)_\ast}\ar[u]_-{\simeq}^-{\id\times\lambda^{\times q}\times \id}
}
\]
where $X_{\underline{i}}=\Omega^{i_0+\dots+i_{2q+1}}(S^V\wedge M_{i_0}\wedge A_{i_1}\wedge\dots\wedge A_{i_{2q+1}})$. Here $\lambda\colon c\to \iota$ is the natural transformation of \ref{cofinal}, which induces an ind-equivalence on the homotopy colimit systems.

We need to show that $(\iota\times c^{\times q}\times\iota)_\ast$ is an equivalence. We recall from \ref{gammaandpi} that there are cofinal functors $\pi\colon \mathcal{M}R^{\Z/2}\wr (I^o/\omega)\to I^{\Z/2}$ and $\gamma\colon \mathbb{N}\to I^o/\omega$. Similarly, there are cofinal functors $p\colon  \mathcal{M}\wr (I/\omega)\to I$ and $d\colon \mathbb{N}\to I/\omega$ where $I/\omega$ has objects the injections $\alpha\colon\underline{n}\rightarrowtail \mathbb{N}$ and $ \mathcal{M}=Inj(\mathbb{N},\mathbb{N})$ acts by postcomposition. The maps $p$ and $d$ send $\alpha$ to $n$ and $n$ to $\underline{n}\subset\mathbb{N}$, respectively, and they are right homotopy cofinal essentially by \cite[2.2.9]{Sh00} (or by an argument completely analogous to \ref{gammaandpi}).
Therefore we obtain a right homotopy cofinal functor
\[
(\pi\times p^{\times q}\times\pi)\colon G \longrightarrow I^{\Z/2}\times I^{\times q}\times I^{\Z/2}\cong(I^{\times 1+{\bf 2q+1}})^{\Z/2}
\]
where $G=(\mathcal{M}R^{\Z/2}\times\mathcal{M}^{\times q}\times \mathcal{M}R^{\Z/2})\wr ((I^o/\omega)\times (I/\omega)^{\times q}\times (I^o/\omega))$ is the Grothendieck construction with respect to the componentwise action of the product monoid $\underline{\mathcal{M}}=(\mathcal{M}R^{\Z/2}\times\mathcal{M}^{\times q}\times \mathcal{M}R^{\Z/2})$. This gives a commutative diagram
\[
\xymatrix@C=10pt{\displaystyle
\hocolim_{\mathbb{N}^{\times q+2}}(\iota\times c^{\times q}\times\iota)^\ast X^{\Z/2}\ar[r]^-{(\iota\times c^{\times q}\times\iota)_\ast}\ar[d]_{(\gamma\times d^{\times q}\times\gamma)_\ast}^{\simeq}
&
\displaystyle
\hocolim_{I^{\Z/2}\times I^{\times q}\times I^{\Z/2}} X^{\mathbb{Z}/2}
\\
\displaystyle
\hocolim_{(I^o/\omega)\times(I/\omega)^{\times q}\times (I^o/\omega)}(\pi\times p^{\times q}\times\pi)^\ast X^{\mathbb{Z}/2}\ar[r]
&
\displaystyle
\hocolim_{G}(\pi\times p^{\times q}\times\pi)^\ast X^{\mathbb{Z}/2}\ar[u]_{(\pi\times p^{\times q}\times\pi)_\ast}^{\simeq}\ar[r]
&
B\underline{\mathcal{M}}\simeq \ast \rlap{\ .}
}
\]
If we can show that  $\underline{\mathcal{M}}$ acts by equivalences on $\hocolim_{\mathbb{N}^{\times q+2}}(\iota\times c^{\times q}\times\iota)^\ast X^{\Z/2}$, Lemma \cite[page 98]{Quillen} will show that the bottom sequence is a fiber sequence, which concludes the proof.

We recall that $\hocolim_{\mathbb{N}^{\times q+2}}(\iota\times c^{\times q}\times\iota)^\ast X^{\Z/2}$ is the space
\[
\hocolim_{\mathbb{N}^{\times q+2}}(\Omega^{n_0\rho+1+\rho\otimes(n_1+\dots+n_q)+n_{q+1}\rho+1}(S^V\wedge M_{n_0\rho+1}\wedge A_{n_1}^{\wedge\mathbb{Z}/2}\wedge\dots\wedge A_{n_q}^{\wedge\mathbb{Z}/2}\wedge A_{n_{q+1}\rho+1}))^{\mathbb{Z}/2},
\]
and that the monoid $\underline{\mathcal{M}}$ acts componentwise. By a simple induction argument we can reduce our claim to showing that $\mathcal{M}R^{\mathbb{Z}/2}$ and $\mathcal{M}$ act trivially on the homotopy groups of spaces of the form
\[
\hocolim_{\mathbb{N}}(\Omega^{n\rho+1}E_{n\rho+1})^{\mathbb{Z}/2}\ \ \ \ \ \ \ \ \ \ \mbox{and} \ \ \ \ \ \ \ \ \ \ \ \
\hocolim_{\mathbb{N}}(\Omega^{n\rho}E_{n}^{\wedge \mathbb{Z}/2})^{\mathbb{Z}/2},
\]
respectively, where $E$ is an orthogonal $\mathbb{Z}/2$-spectrum. The action on $\mathcal{M}R^{\mathbb{Z}/2}$ on the first space is trivial on homotopy groups by semi-stability of $E$ (see \cite[3.4]{Schwede}), and it acts by loop maps by the argument of Theorem \ref{genuineloop}. We are not quite able to reduce the triviality of the $\mathcal{M}$-action on the second space to a statement about semi-stability, so we adapt the argument of \cite[3.4]{Schwede} to our situation.

The action of an injection $f\colon\mathbb{N}\to \mathbb{N}$ on the homotopy groups of the second homotopy colimit is defined as follows.
Let
\[
\alpha\colon S^{t+n\rho}\longrightarrow E_{n}^{\wedge \mathbb{Z}/2}\]
be a continuous equivariant map. The restriction of $f$ to the subset $\underline{n}=\{1,\dots,n\}$ defines an inclusion $f\colon \underline{n}\rightarrowtail  \underline{m}$
for some integer $m\geq n$, and thus an equivariant isometric embedding
\[
(f\otimes\rho)\colon n\rho\longrightarrow m\rho.
\]
This further determines an isomorphism $S^{m\rho}\cong S^{(m-n)\rho}\wedge S^{n\rho}$.
The action of $f$ on the homotopy class of $\alpha$ is the homotopy class of the map $f_\ast \alpha$ defined as the composite
\[
\xymatrix@R=20pt{
f_\ast\alpha\colon
S^{u+m\rho}
\ar[r]^-{\cong}
&
 S^{(m-n)\rho}\wedge S^{t+n\rho}
\ar[r]^-{\id\wedge \alpha}
&
S^{(m-n)\rho}\wedge E_{n}^{\wedge \mathbb{Z}/2}\cong (S^{(m-n)}\wedge E_{n})^{\wedge \mathbb{Z}/2}\ar[r]^-{\sigma\wedge \sigma}
&E_{m}^{\wedge \mathbb{Z}/2}
 }
\]
where the map $\sigma$ is the structure map of the orthogonal $\mathbb{Z}/2$-spectrum $E$. Thus we need to show that $f_\ast\alpha$ and $\alpha$ are stably homotopic. Let $i_1\colon m\to m+m$ be the standard inclusion, so that
\[
(i_1\otimes\rho)\colon m\rho\longrightarrow (m+m)\rho \cong m\rho \oplus m\rho
\]
is the first summand inclusion, and $(i_1)_\ast$ is suspension by $S^{m\rho}$.
By \cite[3.4]{Schwede} we can choose an automorphism $\gamma$ of $\mathbb{R}^m$ such that $\gamma\circ f\colon\mathbb{R}^n\to\mathbb{R}^m$ is the standard inclusion $j$, and $i_1\circ\gamma$ and $i_1$ are homotopic isometric embeddings. Then ($[-]$ denotes the stable homotopy class)
\[
[f_\ast \alpha]
=[(i_1)_\ast f_\ast \alpha]
=[(i_1\circ f)_\ast \alpha]
=[(i_1\circ\gamma\circ f)_\ast \alpha]
=[(i_1\circ j)_\ast \alpha]
=[\alpha]
\]
where the third equality holds since if $H_t$ is a one parameter family of isometric embeddings from $i_1$ to $i_1\circ\gamma$, then $(H_t)_\ast$ defines an equivariant homotopy. Thus $f$ acts trivially on the homotopy groups based at the canonical basepoint. Moreover it acts by loop maps, since the diagram
\[
\xymatrix@C=50pt@R=15pt{
\displaystyle\hocolim_{n\in\mathbb{N}}\Omega^{n\rho}(E_{n}^{\wedge \mathbb{Z}/2})\ar[r]^-{f\cdot (-)}\ar[dd]_-{\simeq}\ar[dr]_-{f\cdot (-)}
&
\displaystyle\hocolim_{n\in\mathbb{N}}\Omega^{n\rho}(E_{n}^{\wedge \mathbb{Z}/2})
\\
&\displaystyle\hocolim_{n\geq f(1)}\Omega^{n\rho}(E_{n}^{\wedge \mathbb{Z}/2})\ar[d]_-{\simeq}\ar[u]_-{\simeq}
\\
\Omega\displaystyle\hocolim_{n\in\mathbb{N}}\Omega^{\sigma+(n-1)\rho}(E_{n}^{\wedge \mathbb{Z}/2})\ar[r]^-{\Omega f_1\cdot (-)}
&
\Omega\displaystyle\hocolim_{n\geq f(1)}\Omega^{\sigma+(n-1)\rho}(E_{n}^{\wedge \mathbb{Z}/2})\rlap{\ ,}
}
\]
commutes, where $\sigma\subset\rho$ is the sign representation, the left vertical map is induced by the splitting $S^\rho\cong S^1\wedge S^\sigma$ of the first copy of $\rho$ in $n\rho$, and the lower right vertical map by the same splitting for the copy of  $\rho$ in $n\rho$ indexed by $f(1)\in \underline{n}$. The map $f_1\colon \underline{n}-1\to \underline{f\cdot n}-f(1)$ is the restriction of $f$, and $f_1\cdot(-)$ is defined in a manner completely analogous to $f\cdot (-)$.

The proof for the spectrum $\THR_k\Omega^{\bullet}_I(A;M;\Sh)$ is analogous. One uses the argument above to reduce the proof to the triviality of the action of $f\in \mathcal{M}R^{\Z/2}$ on the fixed points space of $\hocolim_{n\in\mathbb{N}}\Omega^{n\rho}((\Sh^{n}E)^{\wedge \Z/2})_V$, for some flat orthogonal $\Z/2$-spectrum $E$. Since $E$ is flat, the canonical map
\[
\hocolim_{n\in\mathbb{N}}\Omega^{n\rho}((\Sh^{n}E)^{\wedge \Z/2})\stackrel{\simeq}{\longrightarrow} \hocolim_{n\in\mathbb{N}}\Omega^{n\rho}\Sh^{n\rho}(E^{\wedge \Z/2})
\]
is an equivalence. Under this equivalence the action of $f$
sends a map $\alpha\colon S^{n\rho}\to\Sh^{n\rho}(E^{\wedge \Z/2})_{V}= (E^{\wedge \Z/2})_{n\rho+V}$ to the composite
\[
\xymatrix@R=20pt{
f_\ast\alpha\colon
S^{m\rho}
\ar[r]^-{\cong}
&
 S^{(m-n)\rho}\wedge S^{n\rho}
\ar[r]^-{\id\wedge \alpha}
&
S^{(m-n)\rho}\wedge (E^{\wedge \Z/2})_{n\rho+V}\ar[r]^-{}
&(E^{\wedge \Z/2})_{m\rho+V}
 }
\]
where the last map is the structure map of the orthogonal spectrum $E^{\wedge \Z/2}$. This is trivial on homotopy groups by a similar argument.
\end{proof}

\begin{proof}[Proof of \ref{invBok}]
The real simplicial spectra $\THR_\bullet(A;M)$ and $\THR_\bullet(B;N)$ are good by \ref{goodness}, and it is thus sufficient to show that the map
\[
\hocolim_{\underline{i}\in I^{\times 1+{\bf 2k+1}}}\Omega^{i_0+\dots+i_{2k+1}}(\mathbb{S}\wedge M_{i_0}\wedge A_{i_1}\wedge\dots\wedge A_{i_{2k+1}})\longrightarrow\hocolim_{\underline{i}\in I^{\times 1+{\bf 2k+1}}}\Omega^{i_0+\dots+i_{2k+1}}(\mathbb{S}\wedge N_{i_0}\wedge B_{i_1}\wedge\dots\wedge B_{i_{2k+1}})
\]
is an equivalence for every integer $k\geq 0$. This map is naturally equivalent to the map
\[
\hocolim_{\mathbb{N}}\Omega^{n\rho \otimes (k\rho+2)}(\mathbb{S}\wedge M_{n\rho}\wedge A^{\wedge {\bf 2k+1}}_{n\rho})\longrightarrow \hocolim_{\mathbb{N}}\Omega^{n\rho \otimes (k\rho+2)}(\mathbb{S}\wedge N_{n\rho}\wedge B^{\wedge {\bf 2k+1}}_{n\rho})
\]
by \ref{ItoN}, where the action on ${\bf 2k+1}$ reverses the order. This is a non-equivariant weak equivalence by \cite[3.1.2]{Sh00}. We verify that it also induces a weak equivalence on geometric fixed points. The natural transformation $\Phi^{\Z/2}\Omega^V\to \Omega^{V^{\Z/2}}\Phi^{\Z/2}$ is a stable equivalence of spectra. Moreover, the fixed points of the representation $n\rho \otimes (k\rho+2)$ is
\[
(n\rho \otimes (k\rho+2))^{\Z/2}\cong (nk\rho \otimes \rho)^{\Z/2}+(2n\rho)^{\Z/2}\cong 2nk+2n=2n(k+1).
\]
Thus on geometric fixed points the map above is equivalent to
\[
\hocolim_{\mathbb{N}}\Omega^{2n(k+2)}(\mathbb{S}\wedge M_{n\rho}^{\Z/2}\wedge A^{\wedge k}_{2n}\wedge A^{\Z/2}_{n\rho})
\longrightarrow
\hocolim_{\mathbb{N}}\Omega^{2n(k+2)}(\mathbb{S}\wedge N_{n\rho}^{\Z/2}\wedge B^{\wedge k}_{2n}\wedge B^{\Z/2}_{n\rho}).
\]
The diagonal map $\mathbb{N}\to \mathbb{N}^{\times k+2}$ induces an equivalence between this map and
\[
\xymatrix{
\displaystyle\hocolim_{n_0,\dots,n_{k+1}\in \mathbb{N}}\Omega^{n_0+2(n_1+\dots+n_k)+n_{k+1}}(\mathbb{S}\wedge M_{n_0\rho}^{\Z/2}\wedge A_{2n_1}\wedge\dots\wedge A_{2n_k} \wedge A^{\Z/2}_{n_{k+1}\rho})\ar[d]\\
\displaystyle
\hocolim_{n_0,\dots,n_{k+1}\in\mathbb{N}}\Omega^{n_0+2(n_1+\dots+n_k)+n_{k+1}}(\mathbb{S}\wedge N_{n_0\rho}^{\Z/2}\wedge B_{2n_1}\wedge\dots\wedge B_{2n_k} \wedge B^{\Z/2}_{n_{k+1}\rho}).
}
\]
A simple inductive argument shows that this map is an equivalence, provided  we can prove that for every well-pointed space $X$ and equivalence of levelwise well-pointed $\Z/2$-spectra $A\to B$, the maps
\[
\hocolim_{n\in\mathbb{N}}\Omega^{2n}(X\wedge  A_{2n})\longrightarrow \hocolim_{n\in\mathbb{N}}\Omega^{2n}(X\wedge  B_{2n})
\]
\[
\hocolim_{n\in\mathbb{N}}\Omega^{n}(X\wedge  A^{\mathbb{Z}/2}_{n\rho})\longrightarrow \hocolim_{n\in\mathbb{N}}\Omega^{n}(X\wedge  B^{\mathbb{Z}/2}_{n\rho})
\]
are equivalences of spaces. The first map is an equivalence because smashing with a well-pointed space preserves stable equivalences of spectra. The second map is the infinite loop space of the map of the geometric fixed points of the map of $\Z/2$-spectra $X\wedge A\to X\wedge B$, and smashing with a well-pointed space preserves stable equivalences of $\Z/2$-spectra.
\end{proof}

%
%

\subsection{The comparison of the B\"{o}kstedt model and the dihedral Bar construction}\label{seccomp}

In this section we will show that the models for real topological Hochschild homology previously defined in \S\ref{secTHR}  give rise to equivalent $\mathbb{Z}/2$-spectra, under suitable flatness conditions.

\begin{theorem}\label{comparison}
Let $(A,w)$ be a flat ring spectrum with an anti-involution and $(M,j)$ a flat $(A,w)$-bimodule (see Definition \ref{defflat}). Then there is a natural zig-zag of stable equivalences of $\Z/2$-spectra
\[\THR(A;M)\simeq B^{di}_\wedge(A;M).\]
\end{theorem}

\begin{rem}
Given any ring spectrum with anti-involution $(A,w)$ which is levelwise well-pointed, Theorem \ref{comparison} together with the homotopy invariance of $\THR$ of Theorem \ref{invBok} give a stable equivalence of $\Z/2$-spectra
\[\THR(A)\simeq\THR(A^\flat)\simeq B^{di}_\wedge(A^\flat),\]
where $A^\flat\stackrel{\simeq}{\to}A$ is a flat ring spectrum with anti-involution replacement of $(A,w)$, from \ref{flatrepl}. 
\end{rem}

\begin{proof}[Proof of \ref{comparison}]
In order to simplify the notation we assume that $(M,j)=(A,w)$, the proof of the general case is formally identical. For an object $\underline{i}=(i_0,\dots,i_k)$ of $I^{\times 1+{\bf k}}$ we denote by $\Omega^{\underline{i}}:=\Omega^{i_0+\dots+i_k}$ the associated loop space.
For any integer $k\geq 0$, we consider the zig-zag
\[
A^{\wedge 1+{\bf k}}\rightarrow(\hocolim_{I}\Omega^{i}\Sh^i A)^{\wedge 1+{\bf k}}\to\hocolim_{I^{\times 1+{\bf k}}}\Omega^{\underline{i}}(\Sh^{i_0} A\wedge\dots\wedge \Sh^{i_k} A)\leftarrow\hocolim_{I^{\times 1+{\bf k}}}\Omega^{\underline{i}}\Sigma^{\infty}\!(A_{i_0}\wedge\dots\wedge A_{i_k})
\]
where the first map is the $(1+{\bf k})$-fold smash power of the map $t\colon A\stackrel{\simeq}{\to}\hocolim_{I}\Omega^{i}\Sh^i A$. We observe that even though $t$ is an equivalence by \ref{genuineloop}, the target of $t$ is in general not flat, and therefore the map $t^{\wedge k+1}$ is not necessarily a weak equivalence.
The second map is the canonical map that commutes the smash product with the homotopy colimits and the loops. The left-pointing map is induced by the map $\Sigma^{\infty}(A_{i_0}\wedge\dots\wedge A_{i_k})\to \Sh^{i_0} A\wedge\dots\wedge \Sh^{i_k} A$ which is adjoint to the identity map
\[A_{i_k}\wedge\dots\wedge A_{i_k}=(\Sh^{i_0} A)_0\wedge\dots\wedge (\Sh^{i_n} A)_0=(\Sh^{i_0} A\wedge\dots\wedge \Sh^{i_k} A)_0.\]

The right-most spectrum is the $k$-simplices of the B\"{o}kstedt model of $\THR(A;A)$ from Example \ref{Bokstedtmodel}. The middle right hand spectrum is the $k$-simplices of the real simplicial spectrum associated to the functor $\Omega^{\bullet}_I(A;A;\Sh)\colon N^{di}I\to\Sp$ of Example \ref{middlemodel}. The leftward pointing map is induced by a morphism of $\mathbb{Z}/2$-diagrams, and it is therefore a map of real simplicial spectra. It is immediate to see that the composite of the rightward pointing maps is also a map of real simplicial spectra. Moreover under our flatness assumptions these real simplicial spectra are good (\ref{properness} and \ref{goodness}), and therefore our theorem follows if we can show that the two maps in the zig-zag above are equivariant weak equivalences for every fixed simplicial degree $k$.

The right pointing map of the zig-zag factors as
\[
\xymatrix{
A^{\wedge 1+{\bf k}}\ar[r]\ar[d]_-{\simeq}&\displaystyle\hocolim_{I^{\times 1+{\bf k}}}\Omega^{\underline{i}}(\Sh^{i_0} A\wedge\dots\wedge \Sh^{i_k} A)
\\
\displaystyle
\hocolim_{n\in\mathbb{N}}\Omega^{n\rho(1+{\bf k})}(A\wedge S^{n\rho})^{\wedge 1+{\bf k}}
\ar[r]^-{\simeq}
&\displaystyle
\hocolim_{n\in\mathbb{N}}\Omega^{n\rho(1+{\bf k})}(\Sh^{n\rho} A)^{\wedge 1+{\bf k}}
\ar[u]^-{\simeq}_{\ref{ItoN}}
}
\]
where the left vertical map is the canonical equivalence of the loop-suspension adjunction, and the lower horizontal map is induced by the equivalence of orthogonal spectra $c\colon A\wedge S^{m\rho}\to \Sh^{m\rho}A$ (see \cite{Schwede}). Since both suspensions and shifts preserve flatness (see Appendix \ref{shiftflat}) the $(1+{\bf k})$-fold smash power $c^{\wedge 1+{\bf k}}$ is also an equivalence, the right vertical map is an equivalence by \ref{ItoN}. Thus the top horizontal map is an equivalence.

The proof that that the leftward pointing map of the zig-zag is an equivalence is more involved.
Our strategy consists of replacing the loop spaces $\Omega^i$ with the equivalent free spectra $F_i$. The advantage of doing this is that $F_i$ commutes strictly with the smash product and that it preserves flatness. This follows from \cite[Section 2.3.3]{Sto}. The trade-off is that the combination of $F_i$ and $\Sh^i$ is not fully functorial in $\mathbb{N}$, and we need to work in the triangulated homotopy category $\Ho(\Sp^{\mathbb{Z}/2})$.
We recall that for any $\Z/2$-representation $V$, the free spectrum functor $F_V$ is the left adjoint of the shift functor $\Sh^V$.
The functors $\Omega^{n\rho}(-)$ and $F_{n\rho}(-) = F_{n\rho} \wedge -$ are homotopical, i.e. they preserve $\mathbb{Z}/2$-equivariant stable equivalences. For $\Omega^{n\rho}(-)$ this follows from \cite[5.4]{Schwede} and for $F_{n\rho}(-)$ from the fact that  $F_{n\rho}$ is flat and smashing with flat equivariant spectra is homotopical (\cite[Section III.V]{Schwedeglobal} and \cite[Proposition 2.10.1]{BrDuSt}). Hence these functors descend to the homotopy category
\[ \Omega^{n\rho}(-),  F_{n\rho}(-) \colon \Ho(\Sp^{\mathbb{Z}/2}) \longrightarrow \Ho(\Sp^{\mathbb{Z}/2}).\]
The shift $\Sh^{n\rho}$ is also homotopical, it descends to an equivalence of categories on the homotopy category $\Sh^{n\rho}\colon \Ho(\Sp^{\mathbb{Z}/2}) \to \Ho(\Sp^{\mathbb{Z}/2})$. These observations can be now used to define a natural transformation $\xi_n \colon F_{n\rho}(-) \to \Omega^{n\rho}(-)$ in the homotopy category $\Ho(\Sp^{\mathbb{Z}/2})$. We point out that this natural transformation is not the canonical one which is adjoint to $F_{n\rho} \wedge S^{n\rho} \wedge A \to A$, but it is defined in the following way. We consider the map in $\Sp^{\mathbb{Z}/2}$ defined as the composite of stable equivalences of $\mathbb{Z}/2$-spectra
\[ F_{\rho} (\Sh^\rho A)  \stackrel{\simeq}{\longrightarrow} A \stackrel{\simeq}{\longrightarrow} \Omega^{\rho} \Sh^{\rho} A.\]
Since $\Sh^{\rho}\colon  \Ho(\Sp^{\mathbb{Z}/2}) \to \Ho(\Sp^{\mathbb{Z}/2})$ is an equivalence of categories, by precomposing this map with the inverse $(\Sh^{\rho})^{-1}$, we get a natural isomorphism $\xi=\xi_1 \colon F_{\rho}(-) \stackrel{\cong}{\to}  \Omega^{\rho}(-)$
in the homotopy category $\Ho(\Sp^{\mathbb{Z}/2})$. By iterating this process we obtain a natural isomorphism
\[\xi_n \colon F_{n\rho}(-) \stackrel{\cong}{\longrightarrow} \Omega^{n\rho}(-)\]
of endofunctors of $\Ho(\Sp^{\mathbb{Z}/2})$. We will use this isomorphism to replace $\Omega^{n\rho}$ by $F_{n\rho}$ in the homotopy colimits of our zig-zag.

Next, we recall a way of computing the sequential homotopy colimit in the triangulated homotopy category. The homotopy colimit of a sequence of maps
\[\xymatrix{X_0 \ar[r]^{\alpha_0} & X_1 \ar[r]^{\alpha_1} & X_2 \ar[r] & \dots}\]
in $\Ho(\Sp^{\mathbb{Z}/2})$ is defined by the mapping cone sequence
\[\xymatrix{\bigvee_{n \in \mathbb{N}} X_n \ar[r]^{1-\alpha} & \bigvee_{n \in \mathbb{N}} X_n \ar[r] & \hocolim^{\Delta}_{\mathbb{N}} X \ar[r] & \Sigma \bigvee_{n\in \mathbb{N}} X_n },\]
where $\alpha$ sends the wedge summand $X_n$ to the summand $X_{n+1}$ via $\alpha_n$. The symbol $\Delta$ in the notation $\hocolim^{\Delta}_{\mathbb{N}} X$ suggests that we are using the triangulated structure to define this object. We note that $\hocolim^{\Delta}_{\mathbb{N}} X$ comes with canonical maps $\iota_n X_n \to \hocolim^{\Delta}_{\mathbb{N}} X$ for every $n \geq 0$, such that the diagrams
\[\xymatrix@C=60pt@R=15pt{X_n \ar[rr]^{\iota_n} \ar[dr]_-{\alpha_n} & & \hocolim^{\Delta}_{\mathbb{N}} X \\ &  X_{n+1} \ar[ur]_-{\iota_{n+1}} } \]
commute. We also note that the construction $\hocolim^{\Delta}_{\mathbb{N}} X$ is a generalization of the classical Bousfield-Kan construction of the homotopy colimit. If the maps $\alpha_n$ are honest maps in $\Sp^{\mathbb{Z}/2}$, there is an isomorphism
\[X_n\longrightarrow{\hocolim_{\mathbb{N}}}^{\Delta} X \stackrel{\cong}{\longrightarrow}  \hocolim_{\mathbb{N}} X,\]
where the composite is the canonical map.

Now we go back to the proof of the desired result. We define a commutative diagram in the homotopy category
\[
\xymatrix@R=20pt{\displaystyle
\displaystyle\hocolim_{I^{\times 1+{\bf k}}}\Omega^{\underline{i}}(\Sh^{i_0} A\wedge\dots\wedge \Sh^{i_k} A)
&
\displaystyle\hocolim_{I^{\times 1+{\bf k}}}\Omega^{\underline{i}}\Sigma^{\infty}(A_{i_0}\wedge\dots\wedge A_{i_k})\ar[l]
\\
\displaystyle\displaystyle\hocolim_{n\in \mathbb{N}}\Omega^{n\rho(1+{\bf k})}(\Sh^{n\rho}A)^{\wedge 1+{\bf k}}\ar[u]_{\cong}^-{\ref{ItoN}}
&
\displaystyle\hocolim_{n\in \mathbb{N}}\Omega^{n\rho(1+{\bf k})} \Sigma^{\infty}A^{\wedge 1+{\bf k}}_{n\rho}\ar[u]_{\cong}^-{\ref{ItoN}}\ar[l]
\\
\displaystyle\displaystyle{\hocolim_{n\in \mathbb{N}}}^\Delta\Omega^{n\rho(1+{\bf k})}(\Sh^{n\rho}A)^{\wedge 1+{\bf k}}\ar[u]_{\cong}
&
\displaystyle{\hocolim_{n\in \mathbb{N}}}^\Delta\Omega^{n\rho(1+{\bf k})} \Sigma^{\infty}A^{\wedge 1+{\bf k}}_{n\rho}\ar[u]_{\cong}\ar[l]
\\
\displaystyle{\hocolim_{n\in \mathbb{N}}}^\Delta F_{n\rho(1+{\bf k})}(\Sh^{n\rho}A)^{\wedge 1+{\bf k}}\ar[u]^{\text{I}}
&
\displaystyle{\hocolim_{n\in \mathbb{N}}}^\Delta F_{n\rho(1+{\bf k})} \Sigma^{\infty}A^{\wedge 1+{\bf k}}_{n\rho}\ar[u]_{\text{II}}\ar[l]
\\
\displaystyle{\hocolim_{n\in \mathbb{N}}}^\Delta (F_{n\rho}\Sh^{n\rho}A)^{\wedge 1+{\bf k}}\ar[u]_{\cong}
&
\displaystyle{\hocolim_{n\in \mathbb{N}}}^\Delta (F_{n\rho}\Sigma^{\infty}A_{n\rho})^{\wedge 1+{\bf k}}.\ar[u]_{\cong}\ar[l]^{\text{III}}\rlap{\ ,}
}
\]
where the top row is the map of our zig-zag, and the maps labelled I, II and III are isomorphisms. The horizontal map in the second row is induced by the canonical map from suspension spectra to shifts, and therefore the upper square commutes. The middle horizontal arrow is then uniquely defined in the homotopy category $\Ho(\Sp^{\mathbb{Z}/2})$.  For convenience let us write $l=1+{\bf k}$. The source of the map $\text{I}$ is defined as the homotopy colimit of the bottom row of the commutative diagram
\[
\xymatrix@C=13pt{
A^{\wedge l}& A^{\wedge l}\ar@{=}[l]\ar[r]& \Omega^{\rho l}(\Sh^{\rho}A)^{\wedge l}& \Omega^{\rho l}(\Sh^{\rho}A)^{\wedge l}\ar@{=}[l]\ar[r]& \Omega^{2\rho l}(\Sh^{2\rho}A)^{\wedge l}&\dots\ar[l]_-\cong\\
A^{\wedge l}\ar@{=}[u]& (F_\rho(S^\rho\wedge A))^{\wedge l}\ar[l]_-\cong\ar[r]^-\cong& (F_\rho\Sh^\rho A)^{\wedge l}\ar[u]^{\cong}&(F_{2\rho}(S^\rho\wedge \Sh^{\rho}A))^{\wedge l}\ar[l]_-\cong\ar[r]^-\cong& (F_{2\rho}\Sh^{2\rho}A)^{\wedge l}\ar[u]^\cong&\dots\ar[l]_-\simeq
}
\]
and the map I is induced by the vertical maps, which are the composite of the isomorphism $\xi_n\colon F_{n\rho}(-) \to \Omega^{n\rho}(-)$ and of the map that commutes the loops and the smash products. It is an isomorphism since the composite
\[(F_{n\rho}X)^{\wedge l}  \stackrel{\simeq}{\leftarrow}(F_{n\rho}X)^{\wedge_{\mathbf{L}} l} \to (\Omega^{n\rho} X)^{\wedge_{\mathbf{L}} l} \to  (\Omega^{n\rho} X)^{\wedge l}  \to \Omega^{n\rho l} (X^{\wedge l}) \]
is an isomorphism in the homotopy category for any flat $X$, where $\wedge_{\mathbf{L}}l$ denotes the derived indexed smash product. Here we are using that shifts preserve flatness (see Proposition \ref{prop:flatness}). Similarly, map $\text{II}$ is induced by the commutative diagram
\[
\xymatrix@C=13pt{
\Sigma^{\infty}A^{\wedge l}_0&\Sigma^{\infty}A^{\wedge l}_0\ar@{=}[l]\ar[r]& \Omega^{\rho l}\Sigma^{\infty}A_{\rho}^{\wedge l}&\Omega^{\rho l}\Sigma^{\infty}A_{\rho}^{\wedge l}\ar@{=}[l]\ar[r]& \Omega^{2\rho l}\Sigma^{\infty}A_{2\rho}^{\wedge l}&\dots\ar[l]_-\cong
\\
\Sigma^{\infty}A^{\wedge l}_0\ar@{=}[u]& (F_{\rho}S^\rho\wedge \Sigma^\infty A_0)^{\wedge l}\ar[l]_-\cong\ar[r]& (F_{\rho}\Sigma^\infty A_\rho)^{\wedge l}\ar[u]^{\cong}& (F_{2\rho}S^\rho\wedge \Sigma^\infty A_\rho)^{\wedge l}\ar[l]_-\cong\ar[r]& (F_{2\rho}\Sigma^\infty A_{2\rho})^{\wedge l}\ar[u]^{\cong}&\dots\ar[l]_-\cong
}
\]
and it is an isomorphism on homotopy colimits. This diagram commutes because in the homotopy category
\[\xymatrix{
 \Sigma^{\infty}X_0 \ar[r] & \Omega^{\rho} \Sh^{\rho}\Sigma^{\infty}X_0  \ar[r]  &  \Omega^\rho \Sigma^{\infty}X_\rho
 \\
 F_\rho (S^\rho \wedge \Sigma^{\infty}X_0) \ar[u] \ar[r] & F_\rho \Sh^{\rho}\Sigma^{\infty}X_0 \ar[r] \ar[u]_{\xi} & F_\rho \Sigma^{\infty}X_\rho \ar[u]_{\xi}
  } \]
commutes by construction of $\xi$.
Finally, the two lower horizontal arrows are uniquely determined by the fact that the vertical maps are equivalences. It remains to show that map III is an isomorphism. We treat the case $l=1$ first.
By construction the diagram
\[\xymatrix@C=50pt{
\displaystyle{\hocolim_{n\in \mathbb{N}}}^\Delta F_{n\rho}\Sigma^{\infty}A_{n\rho} \ar[r]^{\text{III}}  &\displaystyle {\hocolim_{n\in \mathbb{N}}}^\Delta F_{n\rho}\Sh^{n\rho}A \ar[r]^-{\cong}& A
\\
 F_{n\rho}\Sigma^{\infty}A_{n\rho}\ar[r]^{c} \ar[u] & F_{n\rho}\Sh^{n \rho}A  \ar[u] \ar[ur]_{\cong}
 }
 \]
commutes in the homotopy category for all $n \geq 0$, where $c$ is the canonical map from the suspension to the shift, and the diagonal morphism is induced by the canonical equivalence of flat equivariant spectra $F_{n \rho}\Sh^{n \rho} A \to A$.
We recall that on equivariant homotopy groups there is an isomorphism $\pi_{\ast}^{(-)}(\hocolim^{\Delta}_{\mathbb{N}} X) \cong \colim_{\mathbb{N}} \pi_{\ast}^{(-)}X$. Under this isomorphism the composite
\[{\hocolim_{n\in \mathbb{N}}}^\Delta F_{n\rho}\Sigma^{\infty}A_{n\rho}\xrightarrow{\text{III}}
 {\hocolim_{n\in \mathbb{N}}}^\Delta F_{n\rho}\Sh^{n\rho}A\stackrel{\cong}{\longrightarrow}  A \]
is the canonical homotopy presentation of $A$ (see \cite[B.4.3]{HHR}), which is an isomorphism. For $l\geq 1$, the composite
\[({\hocolim_{n\in \mathbb{N}}}^\Delta F_{n\rho}\Sigma^{\infty}A_{n\rho})^{\wedge l}\cong{\hocolim_{n\in \mathbb{N}}}^\Delta (F_{n\rho}\Sigma^{\infty}A_{n\rho})^{\wedge l}\xrightarrow{\text{III}}
 {\hocolim_{n\in \mathbb{N}}}^\Delta( F_{n\rho}\Sh^{n\rho}A)^{\wedge l}\stackrel{\cong}{\to}  A^{\wedge l} \]
is the $l$-fold smash power of the canonical presentation. Since $A$ and $F_{n\rho}\Sigma^{\infty}A_{n\rho}$ are flat, this is an equivalence.
\end{proof}

\subsection{The geometric fixed points of \texorpdfstring{$\THR$}{THR}}\label{secderived}\label{secgeom}

The main tool used in the calculations of \S\ref{seccalc} is a formula for the geometric fixed points of $\THR$. Given a $\Z/2$-spectrum $X$, we let $\Phi^{\Z/2}X$ denote its derived geometric fixed-points spectrum. If $(A,w)$ is a flat ring spectrum with an anti-involution and $(M,j)$ is a flat $(A,w)$-bimodule, we want a model for $\Phi^{\Z/2}M$ which has the structure of an $A$-module. For concreteness, we define
\[
\Phi^{\Z/2}M:=\Phi^{\Z/2}_{\mathcal{M}} (M^c),
\]
where $M^c$ is a cofibrant replacement of $M$ as a right $N^{\mathbb{Z}/2}_e A$-module for the module structure of \S\ref{secdihedral}, and $\Phi^{\Z/2}_{\mathcal{M}}$ denotes the Mandell-May monoidal geometric fixed points \cite[Section V.4]{ManMay} (see also \cite[Appendix B]{HHR}). 
Then the spectrum $\Phi^{\Z/2}M$ is a right $A$-module via the map
\[\Phi^{\Z/2}_{\mathcal{M}} (M^c) \wedge A \cong \Phi^{\Z/2}_{\mathcal{M}}(M^c) \wedge \Phi^{\Z/2}_{\mathcal{M}} (N^{\Z/2}_eA) \longrightarrow  \Phi^{\Z/2}_{\mathcal{M}}(M^c\wedge N^{\Z/2}_eA)\longrightarrow 
\Phi^{\Z/2}_{\mathcal{M}} (M^c),
\]
where the isomorphism is given by the diagonal map $A \to  \Phi^{\Z/2}_{\mathcal{M}}(N^{\mathbb{Z}/2}_e A)$ (see \cite[Theorem 3.2.16]{BrDuSt}, \cite[Proposition 3.4.28]{Sto} and \cite[Section 3.3]{Cary}), and 
the last map is the geometric fixed points of the module structure. Similarly, $\Phi^{\Z/2}_{\mathcal{M}} (A^c)$ is a left $A$-module, where $A^c$ is a cofibrant replacement in the category of left $N^{\mathbb{Z}/2}_e A$-modules.

\begin{rem}\label{Remderivedgeom}
The spectrum $\Phi^{\Z/2}_{\mathcal{M}} (M^c)$ is a model for the derived geometric fixed-points spectrum of $M$. It is therefore fully homotopical and agrees up to equivalence with the geometric fixed-points spectra of \cite{LMS} and \cite{Schwede}. To see this, let $C \stackrel{\sim}{\to} N^{\mathbb{Z}/2}_e A$ be a cofibrant replacement of $N^{\mathbb{Z}/2}_e A$ in the model category of $\Z/2$-equivariant associative algebras \cite[III.7]{ManMay}. Then $C$ is cofibrant as a $\Z/2$-spectrum, and the induced map
\[ \Phi^{\Z/2}_{\mathcal{M}}(C)  \stackrel{\sim}{\longrightarrow}  \Phi^{\Z/2}_{\mathcal{M}}(N^{\mathbb{Z}/2}_e A)\]
is a weak equivalence of associative ring spectra. This uses the fact that $N^{\mathbb{Z}/2}_e A$ is built out of induced regular cells in the sense of \cite[Theorem 3.2.14]{BrDuSt} and \cite[Proposition 3.4.25]{Sto}. It follows from this equivalence and \cite[A.1 Lemma]{BMcyc} that $\Phi^{\Z/2}_{\mathcal{M}} (M^c)$ is equivalent to the derived geometric fixed-points of $M$, by considering $M^c$ as a retract of a cellular $N^{\mathbb{Z}/2}_e A$-module.
Similarly, $\Phi^{\Z/2}_{\mathcal{M}} (A^c)$ computes the derived geometric fixed-points of $A$. 
\end{rem}

\begin{theorem}\label{geofixcalc} Let $(A,w)$ be a flat ring spectrum with an anti-involution and $(M,j)$ a flat $(A,w)$-bimodule. Then there is a natural zig-zag of stable equivalences
\[\Phi^{\Z/2}\THR(A;M) \simeq \Phi^{\Z/2}M \wedge_{A}^{\mathbf{L}} \ \Phi^{\Z/2}A.\]
\end{theorem}

\begin{proof} 
By Theorem \ref{comparison} and Corollary \ref{derived smash} we need to compute the geometric fixed points
\[ \Phi^{\Z/2}(M \wedge^{\mathbf{L}}_{N^{\mathbb{Z}/2}_e A} A).\]
Let $M^c$ and $A^c$ denote cofibrant replacements of $M$ and $A$, as right and left $N^{\mathbb{Z}/2}_e A$-modules, respectively. Then the derived smash product $M \wedge^{\mathbf{L}}_{N^{\mathbb{Z}/2}_e A} A$ is modeled by
\[ M^c \wedge_{N^{\mathbb{Z}/2}_e A} A^c.\]
Now we consider $M^c$ and $A^c$ as right and left $C$-modules respectively, via the cofibrant replacement map $C \stackrel{\sim}{\to} N^{\mathbb{Z}/2}_e A$. Let $\overline{M^c}$ and $\overline{A^c}$ denote the cofibrant replacements of $M^c$ and $A^c$ as $C$-modules. Then Lemma \ref{barderived} implies that there is a natural stable equivalence of $\Z/2$-spectra
\[M^c \wedge_{N^{\mathbb{Z}/2}_e A} A^c \simeq \overline{M^c} \wedge_C \overline{A^c}.\]
Since $\overline{M^c} \wedge_C \overline{A^c}$ is cofibrant as a $\Z/2$-spectrum we have equivalences
\[\Phi^{\Z/2}(M \wedge^{\mathbf{L}}_{N^{\mathbb{Z}/2}_e A} A) \simeq \Phi^{\Z/2}_{\mathcal{M}}(\overline{M^c} \wedge_C \overline{A^c}).\]
For the right-hand term we have an equivalence
\[\Phi^{\Z/2}_{\mathcal{M}}(\overline{M^c} \wedge_C \overline{A^c}) \cong\Phi^{\Z/2}_{\mathcal{M}}(\overline{M^c}) \wedge_{\Phi^{\Z/2}_{\mathcal{M}}(C)} \Phi^{\Z/2}_{\mathcal{M}}(\overline{A^c}) \stackrel{\simeq}{\longrightarrow} \Phi^{\Z/2}_{\mathcal{M}}(M^c) \wedge_A \Phi^{\Z/2}_{\mathcal{M}}(A^c)= \Phi^{\Z/2}M \wedge_A \Phi^{\Z/2}A,\]
where the isomorphism is from the proof of \cite[Proposition B.203]{HHR}. Indeed, the maps $\Phi^{\Z/2}_{\mathcal{M}}(\overline{M^c}) \to\Phi^{\Z/2}_{\mathcal{M}}(M^c)$ and $\Phi^{\Z/2}_{\mathcal{M}}(\overline{A^c}) \to\Phi^{\Z/2}_{\mathcal{M}}(A^c)$ are equivalences by Remark \ref{Remderivedgeom}. Now the map above can be seen to be an equivalence by comparing the bar constructions for the smash products, and by using Lemma \ref{barderived} since all the $\Z/2$-spectra involved are flat.
\end{proof}

\begin{rem} \label{noC} The proof above shows that the spectrum $\Phi^{\Z/2}_{\mathcal{M}}(M^c \wedge_{N^{\mathbb{Z}/2}_e A} A^c)$ has the correct homotopy type. More precisely, there is a commutative diagram
\[\xymatrix{ \Phi^{\Z/2}_{\mathcal{M}}(\overline{M^c}) \wedge_{\Phi^{\Z/2}_{\mathcal{M}}(C)} \Phi^{\Z/2}_{\mathcal{M}}(\overline{A^c}) \ar[r]^-{\simeq} \ar[d]^-{\cong} & \Phi^{\Z/2}_{\mathcal{M}}(M^c) \wedge_A \Phi^{\Z/2}_{\mathcal{M}}(A^c) \ar[d]^-{\cong}\\ \Phi^{\Z/2}_{\mathcal{M}}(\overline{M^c} \wedge_C \overline{A^c}) \ar[r] & \Phi^{\Z/2}_{\mathcal{M}}(M^c \wedge_{N^{\mathbb{Z}/2}_e A} A^c)}\] 
where the vertical maps are isomorphisms by \cite[Proposition B.203]{HHR} and \cite[Theorem 3.2.16]{BrDuSt}. It follows that the bottom horizontal map is an equivalence. Since $\overline{M^c} \wedge_C \overline{A^c}$ is a cofibrant replacement of $M^c \wedge_{N^{\mathbb{Z}/2}_e A} A^c$ as a $\Z/2$-spectrum, this shows that $\Phi^{\Z/2}_{\mathcal{M}}(M^c \wedge_{N^{\mathbb{Z}/2}_e A} A^c)$ computes the derived geometric fixed-points of $M^c \wedge_{N^{\mathbb{Z}/2}_e A} A^c$.
\end{rem}

\begin{cor}\label{geomzero} Let $(A,w)$ be a flat ring spectrum with an anti-involution and $(M,j)$ a flat $(A,w)$-bimodule. Suppose that the underlying $\Z/2$-spectrum of $(A,w)$ is a module over $H\Z [\tfrac{1}{2}]$, where $\Z[\tfrac{1}{2}]$ has the trivial involution. Then the geometric fixed points spectrum $\Phi^{\mathbb{Z}/2}\THR(A;M)$ is contractible. This is for example the case when $A=HR$ is the Eilenberg-MacLane spectrum of a discrete ring with anti-involution, with $1/2\in R$.
\end{cor}

\begin{proof} Since $(A,w)$ is a module over $H\Z[\tfrac{1}{2}]$, the geometric fixed points $\Phi^{\mathbb{Z}/2}A$ are a module over $\Phi^{\mathbb{Z}/2} H\Z[\tfrac{1}{2}]$, which is contractible. It follows that $\Phi^{\mathbb{Z}/2}A$ is also contractible. Now the geometric fixed point formula \ref{geofixcalc} gives a stable equivalence
\[ \Phi^{\Z/2}\THR(A;M) \simeq \Phi^{\Z/2}M \wedge_{A}^{\mathbf{L}} \ \Phi^{\Z/2}A,\]
and the smash factor on the right hand side is contractible.
\end{proof}

\section{\texorpdfstring{$\THR$}{THR} of Wall antistructures and categories with duality}\label{secgen}

A more general input for real $K$-theory than a ring with anti-involution is a Wall antistructure \cite{Wall}. This structure allows one to ``twist'' Hermitian forms by a unit in the underlying ring. We present a variation of our construction of $\THR$ which accepts this more general input.

\subsection{Categories with duality}\label{seccatdual}

Let $\mathscr{C}$ be a category enriched in orthogonal spectra. We recall that a duality on $\mathscr{C}$ is an enriched functor $D\colon \mathscr{C}^{op}\to \mathscr{C}$ together with a natural isomorphism $\eta\colon \id\to D^2$ such that $D(\eta_c)\eta_{Dc}=\id$. We say that the duality is strict if $\eta$ is the identity natural transformation.

\begin{defn}\label{defbimod} Let $(\mathscr{C},D,\eta)$ be a spectral category with duality. A $(\mathscr{C},D,\eta)$-bimodule is a functor $\mathscr{M}\colon \mathscr{C}^{op}\wedge \mathscr{C}\to \Sp$ with an enriched natural transformation $J\colon \mathscr{M}(c,d)\to \mathscr{M}(Dd,Dc)$ such that
\[
\xymatrix@C=50pt{
\mathscr{M}(c,d)\ar[r]^-J\ar[d]_-{\eta\circ(-)}&\mathscr{M}(Dd,Dc)\ar[d]^-J\\
\mathscr{M}(c,D^2d)\ar[r]_-{(-)\circ\eta^{-1}}&\mathscr{M}(D^2c,D^2d)
}
\]
commutes. We observe that if the duality on $\mathscr{C}$ is strict, then $J^2=\id$.
\end{defn}

We remark that the mapping spectra of the form $\mathscr{C}(Dc,c)$, as well as the spectra $\mathscr{M}(c,Dc)$, inherit strict $\mathbb{Z}/2$-actions defined by the maps
\[
\xymatrix@R=10pt{
\mathscr{C}(Dc,c)\ar[r]^{D}&\mathscr{C}(Dc,D^2c)\ar[r]^-{\eta^{-1}\circ(-)}&\mathscr{C}(Dc,c)
\\
\mathscr{M}(c,Dc)\ar[r]^{J}&\mathscr{M}(D^2c,Dc)\ar[r]^-{\eta^\ast}&\mathscr{M}(c,Dc)\rlap{\ .}
}
\]
The construction of $\THR$ we will propose depends on the genuine equivariant homotopy types of these spectra.

\begin{example}\label{excatdual}
\begin{enumerate}
\item A spectral category with duality with one object is a triple $(A,w,\epsilon)$ consisting of an orthogonal ring spectrum $A$, a unit $\epsilon\in A^{\times}_0$, and a morphism of orthogonal ring spectra $w\colon A^{op}\to A$ such that $w^2(a)=\epsilon a\epsilon^{-1}$ and $w(\epsilon)=\epsilon^{-1}$. The Eilenberg MacLane construction of a Wall antistructure as defined in \cite{Wall}, or of a simplicial ring with involution as in \cite{BF}, provides such an object. We will therefore call a spectral category with duality with one object a \textit{spectral antistructure}.

An example of interest of a spectral anti-structure that does not arise as an Eilenberg-MacLane construction is the spherical group-ring. If $G$ is a well-pointed topological group and $\epsilon$ an element in the center of $G$, the triple $(\mathbb{S}\wedge G_+,w,\epsilon)$ is a spectral antistructure, where $w$ is induced by inversion in $G$.
\item Let $R$ be a discrete ring. A wall antistructure $(R,w,\epsilon)$ defines a duality on the category $\mathcal{P}_R$ of finitely generated projective right $R$-modules. It is defined on objects by the abelian group of morphisms of $R$-module maps
\[DP:=\hom_R(P,R),\]
 where $R$ is a right $R$-module by $ar=w(r)\epsilon a$, and $R$ acts on $DP$ by pointwise right multiplication on $R$. There is an isomorphism $\eta\colon P\to D^2P$ that sends $p$ to the map that sends $\lambda\colon P\to R$ to $w(\lambda(p))\epsilon$. This is a linear category with duality in the sense of \cite{Schlichting}, and applying the Eilenberg-MacLane construction on the abelian groups of morphisms gives rise to a spectral category with duality.
\end{enumerate}
\end{example}

\begin{rem}
One may wish to consider a framework where the spectral category $\mathscr{C}$ has a notion of weak equivalences as in \cite{BMlocTHH}, and where the map $\eta$ is required only to be a weak equivalence. The construction of this section require $\eta$ to be invertible, but an extension of our constructions to an appropriate context with weak equivalences is suggested in Remark \ref{catdualwe}.
\end{rem}

There is a variation of the Segal edgewise subdivision of the spectrally enriched cyclic nerve $N^{cy}_{\wedge}(\mathscr{C};\mathscr{M})$ which supports a strict $\mathbb{Z}/2$-action.
Namely, we define a simplicial spectrum with $k$-simplices
\begin{align*}
\bigvee_{(c_0,\dots,c_{2k+1})}\mathscr{M}(c_0,Dc_{2k+1})\wedge \mathscr{C}(c_1,c_0)\wedge\dots\wedge \mathscr{C}(c_k,c_{k-1})\wedge
\\
 \mathscr{C}(Dc_{k+1},c_{k})\wedge  \mathscr{C}(Dc_{k+2},Dc_{k+1})\wedge\dots\wedge  \mathscr{C}(Dc_{2k+1},Dc_{2k})
\end{align*}
and the standard simplicial structure of the Segal edgewise subdivision of $N^{cy}_{\wedge}(\mathscr{C};\mathscr{M})$. Since the wedge summands are reindexed from the usual summands of the cyclic nerve via an equivalence of categories, it easy to see that this simplicial spectrum is equivalent to the Segal edgewise subdivision of $N^{cy}_{\wedge}(\mathscr{C};\mathscr{M})$. By abuse of notation we represent an element of the cyclic nerve as a string of composable arrows
\begin{equation}\label{mf}
(m,\underline{f})=Dc_{2k+1}\stackrel{m}{\longleftarrow}c_0\stackrel{f_1}{\longleftarrow}c_1\stackrel{}{\longleftarrow}\dots \longleftarrow c_k\stackrel{f_{k+1}}{\longleftarrow}Dc_{k+1}\stackrel{f_{k+2}}{\longleftarrow}Dc_{k+2}\longleftarrow\dots\stackrel{f_{2k+1}}{\longleftarrow}Dc_{2k+1}.
\end{equation}
The simplicial spectrum above has a simplicial involution, which is defined by sending the $(c_0,\dots,c_{2k+1})$-summand to the $(c_{2k+1},\dots,c_{0})$-summand via the map which sends $(m,\underline{f})$ to
\[
\xymatrix@C=40pt{
Dc_{0}&c_{2k+1}\ar[l]_-{J(m)\eta}&c_{2k}\ar[l]_-{\eta^{-1}D(f_{2k+1})\eta}&\dots\ar[l]&c_{k+1}\ar[l]_-{\eta^{-1}D(f_{k+2})\eta}\\
Dc_{0}\ar[r]^{Df_1}
&
 Dc_1\ar[r]^{Df_2}
&\dots\ar[r]
&Dc_{k-1}\ar[r]_-{Df_k}
&Dc_{k}\ar[u]_-{\eta^{-1}D(f_{k+1})}&.
}
\]

\begin{defn} The resulting $\Z/2$-spectrum is called the dihedral Bar construction of  $(\mathscr{C},D,\eta)$ with coefficients in $(M,J)$. It is denoted $B^{di}_\wedge(\mathscr{C};\mathscr{M},\eta)$, or by $B^{di}_\wedge(\mathscr{C},\eta)$ when $(\mathscr{M},J)=(\Hom_{\mathscr{C}}(-,-),D)$. We will write  $B^{di}_\wedge(A,w,\epsilon)$ for the dihedral Bar construction of a spectral antistructure $(A,w,\epsilon)$.
\end{defn}

Similarly, the variation of the diagram $I^{\times 2k+2}\to \Sp$ that defines the B\"{o}kstedt model
\[
\Omega^{i_0+\dots+i_{2k+1}}\bigvee_{c_0,\dots,c_{2k+1}\in Ob\mathscr{C}}\mathbb{S}\wedge\mathscr{M}(c_0,Dc_{2k+1})_{i_0}\wedge \mathscr{C}(c_1,c_0)_{i_1}\wedge\dots\wedge \mathscr{C}(Dc_{2k+1},Dc_{2k})_{i_{2k+1}}
\]
from \cite{ringfctrs} admits an analogous $\mathbb{Z}/2$-structure. It is defined by replacing the map $M_{i_0}\wedge A_{i_1}\wedge\dots\wedge A_{i_{2k+1}}\to M_{i_0}\wedge A_{i_{2k+1}}\wedge\dots\wedge A_{i_{1}}$ of Example \ref{Bokstedtmodel1} with the map
\[
\xymatrix@R=10pt{
\mathscr{M}(c_0,Dc_{2k+1})_{i_0}\wedge \mathscr{C}(c_1,c_0)_{i_1}\wedge\dots\wedge \mathscr{C}(Dc_{2k+1},Dc_{2k})_{i_{2k+1}}\ar[d]\\
 \mathscr{M}(c_{2k+1},Dc_{0})_{i_{2k+1}}\wedge \mathscr{C}(c_{2k},c_{2k+1})_{i_{2k}}\wedge\dots\wedge \mathscr{C}(Dc_{0},Dc_{1})_{i_{0}}
}
\]
that sends $m\wedge f_1\wedge\dots\wedge f_{2k+1}$ to
\[
J(m)\eta\wedge \eta^{-1}D(f_{2k+1})\eta\wedge\dots\wedge \eta^{-1}D(f_{k+2})\eta\wedge \eta^{-1}D(f_{k+1})\wedge D(f_{k})\wedge\dots\wedge D(f_1).
\]
\begin{defn}
The resulting $\mathbb{Z}/2$-spectrum is denoted by $\THR(\mathscr{C},\eta;\mathscr{M})$, or $\THR(\mathscr{C},\eta)$ if $(\mathscr{M},J)=(\Hom_{\mathscr{C}}(-,-),D)$. We will write  $\THR(A,w,\epsilon)$ in the case of a spectral antistructure $(A,w,\epsilon)$.
\end{defn}

\begin{defn}
We say that a spectral category with duality  $(\mathscr{C},D,\eta)$ is flat if the mapping spectra are flat, and if the mapping spectra of the form $\mathscr{C}(c,Dc)$ are flat as orthogonal $\mathbb{Z}/2$-spectra for every object $c$ in $\mathscr{C}$, with respect to the involutions defined after Definition \ref{defbimod}.
Similarly, a $(\mathscr{C},D,\eta)$-bimodule $(\mathscr{M},J)$ is flat if it takes values in flat orthogonal spectra and if the  orthogonal $\mathbb{Z}/2$-spectra $\mathscr{M}(c,Dc)$ are flat.
\end{defn}

The proof of Theorem \ref{comparison} can be adapted to the categorical framework, giving the following.

\begin{theorem}\label{categoricalcomp}
Let  $(\mathscr{M},J)$ be a flat bimodule over a flat spectral category with duality $(\mathscr{C},D,\eta)$. Then there is a zig-zag of stable equivalences of $\mathbb{Z}/2$-spectra
\[
\THR(\mathscr{C},\eta;\mathscr{M})\simeq B^{di}_{\wedge}(\mathscr{C},\eta;\mathscr{M}).
\]
\end{theorem}\qed

\begin{cor}
Under the assumptions of \ref{categoricalcomp}, the spectrum $\Phi^{\mathbb{Z}/2}\THR(\mathscr{C},\eta;\mathscr{M})$ is equivalent to the geometric realization of the simplicial spectrum
\[
\bigvee_{c_0,\dots,c_{k}\in Ob\mathscr{C}}\Phi^{\mathbb{Z}/2}(\mathscr{M}(c_0,Dc_0))\wedge \mathscr{C}(c_1,c_0)\wedge\dots\wedge \mathscr{C}(c_k,c_{k-1})\wedge \Phi^{\mathbb{Z}/2}(\mathscr{C}(Dc_k,c_{k})).
\]
The first and last face maps are induced by maps $\Phi^{\mathbb{Z}/2}(\mathscr{M}(c,Dc))\wedge \mathscr{C}(d,c)\to \Phi^{\mathbb{Z}/2}(\mathscr{M}(d,Dd))$ and $\mathscr{C}(c,d)\wedge \Phi^{\mathbb{Z}/2}(\mathscr{C}(Dc,c))\to \Phi^{\mathbb{Z}/2}(\mathscr{C}(Dd,d))$ defined in a manner analogous to \ref{geofixcalc}.
\end{cor}

\begin{proof}
By Theorem \ref{categoricalcomp} the geometric fixed-point spectrum of $\THR(\mathscr{C};\mathscr{M})$ is equivalent to the geometric fixed points of the realization of the Segal edgewise subdivision of $N^{di}_{\wedge}(\mathscr{C};\mathscr{M})$. Since the geometric fixed points functor commutes with realizations this is the geometric realization of a simplicial spectrum with $k$-simplices
\[
\Phi^{\mathbb{Z}/2}\big(\bigvee_{c_0,\dots,c_{2k+1}\in Ob\mathscr{C}}\mathscr{M}(c_0,c_{2k+1})\wedge \mathscr{C}(c_1,c_0)\wedge\dots\wedge \mathscr{C}(c_{2k+1},c_{2k})\big).
\]
The geometric fixed points commute with indexed coproducts, in the sense that $\Phi^G(\bigvee_{i\in I}X_i)\simeq \bigvee_{i\in I^G}\Phi^G(X_i)$ for every finite $G$-set $I$ and $I$-indexed family of spectra $\{X_i\}$. Thus the spectrum above is equivalent to
\[
\bigvee\Phi^{\mathbb{Z}/2}\big(\mathscr{M}(c_0,Dc_{0})\wedge \mathscr{C}(c_1,c_0)\wedge\dots\wedge  \mathscr{C}(c_{k},c_{k-1})\wedge \mathscr{C}(Dc_k,c_{k})\wedge  \mathscr{C}(Dc_{k-1},Dc_{k})\wedge\dots\wedge \mathscr{C}(Dc_{0},Dc_{1})\big),
\]
where the wedge runs through the collections of objects $c_0,\dots,c_{k}\in Ob\mathscr{C}$.
The action on the smash product is indexed over the involution on $\{1,\dots,2k+1\}$ which reverses the order, which has a unique fixed-point $k+1$. Thus the spectrum above is equivalent to
\[
\bigvee_{c_0,\dots,c_{k}\in Ob\mathscr{C}}\Phi^{\mathbb{Z}/2}(\mathscr{M}(c_0,Dc_{0}))\wedge \mathscr{C}(c_1,c_0)\wedge\dots\wedge  \mathscr{C}(c_{k},c_{k-1})\wedge \Phi^{\mathbb{Z}/2}(\mathscr{C}(Dc_k,c_{k})).\]
\end{proof}

\subsection{Functoriality of \texorpdfstring{$\THR$}{THR}}

We will explain the functoriality of the $\THR$ construction for categories with duality, and determine which natural transformations induce equivariant homotopies on $\THR$.

\begin{defn}\!\!\cite[3.2]{Schlichting}
A morphisms of spectral categories with non-strict duality is a pair $(F,\xi)\colon(\mathscr{C},D,\eta)\to (\mathscr{C}',D',\eta')$ of a spectrally enriched functor $F\colon\mathscr{C}\to \mathscr{C}'$ and a natural isomorphism $\xi\colon FD\to D'F$ such that
\[
\xymatrix{
F(c)\ar[r]^{F(\eta_c)}\ar[d]_-{\eta_{F(c)}'}&FD^2(c)\ar[d]^{\xi_{Dc}}\\
(D')^2F(c)\ar[r]_-{D'\xi_c}&D'FD(c)
}
\]
commutes for every object $c$ of $\mathscr{C}$.
\end{defn}

A morphism $(F,\xi)\colon(\mathscr{C},D,\eta)\to (\mathscr{C}',D',\eta')$ induces a map of $\mathbb{Z}/2$-spectra on $\THR$ and on the dihedral nerves. It sends a string of arrows $(m,\underline{f})$ as in (\ref{mf}) to
\[
D'\!Fc_{2k+1}\!\!\xleftarrow{\xi F(m)}\!Fc_0\!\!\xleftarrow{F(f_1)}\dots \xleftarrow{F(f_k)}\! Fc_{k}\!\!\xleftarrow{F(f_{k+1})\xi^{-1}}\! D'\!Fc_{k+1}\!\!\xleftarrow{\xi F(f_{k+2})\xi^{-1}}\dots\xleftarrow{\xi F(f_{2k+1})\xi^{-1}}\!D'\!Fc_{2k+1}.
\]

\begin{defn}
Let $(F,\xi),(G,\chi)\colon(\mathscr{C},D,\eta)\to (\mathscr{C}',D',\eta')$ be two morphisms of spectral categories with non-strict duality. Given a natural transformation $U\colon F\to G$ we let $\overline{U}\colon G\to F$ be the natural transformation defined as the composite
\[
\xymatrix{
\overline{U}\colon G(c)\ar[r]^-{G(\eta_c)}
&
G(D^2c)\ar[r]^{\chi_{Dc}}
& D'GD(c)\ar[r]^-{D'U_{Dc}}
&
D'FD(c)
\ar[r]^-{\xi_{Dc}^{-1}}
&
FD^2(c)\ar[r]^-{F(\eta_{c}^{-1})}
&F(c)
}.
\]
We say that a natural isomorphism $U\colon F\to G$ is equivariant if $\overline{U}=U^{-1}$.
\end{defn}

\begin{prop}\label{homotopiesTHR}
Let $(F,\xi),(G,\chi)\colon(\mathscr{C},D,\eta)\to (\mathscr{C}',D',\eta')$ be two morphisms of spectral categories with non-strict duality, and $U\colon F\to G$ an equivariant natural isomorphism. Then the induced morphisms
\[F_\ast, G_\ast\colon \THR(\mathscr{C},\eta)\longrightarrow \THR(\mathscr{C}',\eta')\]
and
\[F_\ast, G_\ast\colon B^{di}_\wedge(\mathscr{C},\eta)\longrightarrow B^{di}_\wedge(\mathscr{C}',\eta')\]
are equivariantly homotopic.
\end{prop}

\begin{proof}
We prove the proposition for the dihedral nerve, the argument for $\THR$ is analogous. We show that the Segal edgewise subdivisions of $F_\ast$ and $G_\ast$ are simplicially equivariantly homotopic.
We define a simplicial homotopy in degree $k$ by sending $((m,\underline{f}),\sigma\in\Delta_{k}^1)$ to $F_\ast (m,\underline{f})$ if $\sigma=0$, to $G_\ast (m,\underline{f})$ if $\sigma=1$, and to
the string of morphisms
\[\xymatrix{
D'Fc_{2k+1}
&Fc_0\ar[l]_-{\xi F(m)}&Fc_1\ar[l]_-{F(f_1)}&\ar[l]_-{F(f_2)}\dots&Fc_{n_\sigma-1} \ar[l]_-{F(f_{n_\sigma-1})}&
\\
Gc_{k}\ar[r]^{G(f_k)}&Gc_{k-1}\ar[r]&\dots\ar[r]\ar[r]&Gc_{n_\sigma+1}\ar[r]^-{G(f_{n_\sigma+1})}&Gc_{n_\sigma}\ar[u]_-{\overline{U}G(f_{n_{\sigma}})}
\\
D'Gc_{k+1}\ar[u]^-{G(f_{k+1})\chi^{-1}}&&D'Gc_{k+2}\ar[ll]^-{\chi G(f_{k+2})\chi^{-1}}&\dots\ar[l]&D'Gc_{2k+1-n_\sigma}\ar[l]
\\
D'Fc_{2k+1}\ar[r]^-{\chi F(f_{2k+1})\chi^{-1}}&\dots\ar[r]&D'Fc_{2k+3-n_\sigma}\ar[rr]^-{\chi F(f_{2k+3-n_\sigma})\chi^{-1}}&&D'Fc_{2k+2-n_\sigma}\ar[u]_-{\chi G(f_{2k+2-n_\sigma})\chi^{-1}D'\overline{U}}
}\]
otherwise, where $0< n_{\sigma}< k+1$ is the cardinality of the preimage of $0$ of $\sigma\colon [k]\to [1]$.
\end{proof}

\begin{rem}
It is possible to have homotopies on $\THR$ induced by natural transformations which are not isomorphisms, if one is willing to work with non-unital functors. Namely, a functor $(F,\xi)\colon(\mathscr{C},D,\eta)\to (\mathscr{C}',D',\eta')$ which does not preserve the identities still induces a morphism of semisimplicial objects on the dihedral nerve and $\THR_\bullet$, and therefore a map of $\Z/2$-spectra on thick realizations. If the categories are flat this is equivalent to the thin realization. Given a natural transformation $U\colon F\to G$, one requires that
\[
\xymatrix{
F(c)\ar[d]_-{F(f)}\ar[r]^-{U_c}
&
G(c)\ar[d]^-{\overline{U}_c}
\\
F(d)&F(c)\ar[l]_-{F(f)}
}
\ \ \ \ \ \ \ \ \ \ \ \ \ \ \ \
\xymatrix{
G(c)\ar[d]_-{G(f)}\ar[r]^-{G(f)}
&
G(d)\ar[d]^-{\overline{U}_d}
\\
G(d)&F(d)\ar[l]_-{U_d}
}
\]
commute for every morphism $f\colon c\to d$ in $\mathscr{C}$. Then a formula similar to the proof of \ref{homotopiesTHR} defines a semisimplicial homotopy between the maps induced by $F$ and $G$. We will find ourselves in a similar situation in the proof of the cofinality Theorem \ref{cofinality}.
\end{rem}

\subsection{Strictification of the duality}

A spectral category with non-strict duality $(\mathscr{C},D,\eta)$ can be replaced with an equivalent spectral category with strict duality $(\mathcal{D}\mathscr{C}, D)$. We finish the section with a comparison of the corresponding real topological Hochschild homologies.

The spectral category $\mathcal{D}\mathscr{C}$ has objects the triples $(c,d,\phi)$ where $c$ and $d$ are objects of $\mathscr{C}$ and $\phi\colon d\to Dc$ is an isomorphism in the underlying morphism set $\mathscr{C}(d,Dc)_0$. The spectrum of morphisms from $(c,d,\phi)$ to $(c',d',\phi')$ is defined as the pullback
\[
\xymatrix{
\mathcal{D}\mathscr{C}((c,d,\phi),(c',d',\phi'))\ar[r]\ar[d]
&\mathscr{C}(c,c')\ar[d]^-{D(-)\circ\phi'}\\
\mathscr{C}(d',d)\ar[r]^-{\cong}_-{\phi\circ (-)}&
\mathscr{C}(d',Dc).
}
\]
The strict duality is the enriched functor that sends an object $(c,d,\phi)$ to $(d,c,D\phi\circ\eta_c)$, and a morphism $(f,g)$ to $(g,f)$. The projection onto the first coordinate defines an enriched equivalence $\mathcal{D}\mathscr{C}\to\mathscr{C}$. In particular the corresponding topological Hochschild homologies and cyclic nerves are equivalent.

\begin{cor}\label{strictifyTHR}
Let $(\mathscr{C},D,\eta)$ be a spectral category with non-strict duality. There are natural stable equivalences of $\mathbb{Z}/2$-spectra
\[
\THR(\mathscr{C},\eta)\simeq \THR(\mathcal{D}\mathscr{C},\id)\ \ \ \ \ \ \ \ \ \ \mbox{and} \ \ \ \ \ \ \ \ \ \ \ \ B^{di}_\wedge(\mathscr{C},\eta)\simeq B^{di}_\wedge(\mathcal{D}\mathscr{C},\id).
\]
\end{cor}

\begin{proof} We prove the claim for the dihedral Bar construction.
The functor $p\colon\mathcal{D}\mathscr{C}\to\mathscr{C}$ can be lifted to a morphism of categories with duality, by the natural isomorphism
\[\xi=\phi\colon pD(c,d,\phi)=d\to Dc.\]
 The functor $p$ has an inverse $s\colon\mathscr{C}\to \mathcal{D}\mathscr{C}$ that sends $c$ to $(c,Dc,\id_{Dc})$ and a morphism $f$ to $(f,Df)$. Also $s$ admits a natural isomorphism \[\chi=(\id_{Dc},\eta)\colon sDc=(Dc,D^2c,\id_{D^2c})\stackrel{}{\to}Dsc=(Dc,c,\eta)\]
 which defines a morphism of categories with duality.
Clearly $ps$ is the identity, and there is a natural isomorphism
\[U=(\id_c,\phi)\colon sp(c,d,\phi)=(c,Dc,\id_{Dc})\longrightarrow (c,d,\phi).\]
By \ref{homotopiesTHR} it suffices to verify that $U$ is equivariant. The natural transformation $\overline{U}\colon \id\to sp$ is the composite
\[
\xymatrix{
\overline{U}\colon (c,d,\phi)\ar[rrr]^-{DU_{D}=(D\phi\circ \eta,\id_d)}
&&&
(Dd,d,\eta_d)\ar[rrr]^{((D\phi\circ\eta)^{-1},\phi^{-1})}&&&(c,Dc,\id_{Dc})
}
\]
and therefore $\overline{U}=U^{-1}$.
\end{proof}

\begin{rem}
Let $(R,w,\epsilon)$ be a Wall antistructure and $\mathscr{C}=H\mathcal{P}_R$ the Eilenberg-MacLane construction of the linear category $\mathcal{P}_R$ of finitely generated projective $R$-modules, with the duality discussed in Example \ref{excatdual}. The spectrum $\THR(\mathcal{D}H\mathcal{P}_R,\id)$ is the natural recipient of a trace map from the real $K$-theory spectrum $\KR(R,w,\epsilon)$ (\cite{thesis}, \cite{IbLars}).
On the other hand by Morita invariance the $\mathbb{Z}/2$-spectra $\THR(H\mathcal{P}_R,\eta)$ and $\THR(HR,w,\epsilon)$ are equivalent (see Theorem \ref{Morita}). Thus combined with \ref{strictifyTHR} this produces a natural map
\[\tr\colon \KR(R,w,\epsilon)\longrightarrow \THR(HR,w,\epsilon)\]
 in the homotopy category of $\Z/2$-spectra.
\end{rem}

\begin{rem}\label{catdualwe}
The construction $\mathcal{D}\mathscr{C}$ makes sense also when the map $\eta\colon \id\to D^2$ is required only to be a weak equivalence instead of an isomorphism. The cyclic nerve $N^{cy}_\wedge\mathscr{C}$ and $\THH(\mathscr{C})$ themselves do not support a strict involution in this case, since our definition of the involution requires inverting $\eta$. In this case one should take $N^{di}_\wedge(\mathcal{D}\mathscr{C},\id)$, or the equivalent $\THR(\mathcal{D}\mathscr{C},\id)$, as models for real topological Hochschild homology.
\end{rem}

\section{Fundamental properties of \texorpdfstring{$\THR$}{THR}}\label{secprop}

\subsection{Multiplicative structures}\label{secmult}

Let $(A,w)$ and $(B,\sigma)$ be two ring spectra with anti-involution. The smash product $A\wedge B$ is canonically a ring spectrum, and the map $w\wedge \sigma$ defines an anti-involution on $A\wedge B$. We observe that there is a natural isomorphism
\[
B^{di}(A\wedge B,w\wedge \sigma)\cong B^{di}(A,w)\wedge B^{di}(B,\sigma),
\]
induced by the levelwise shuffling of the smash factors
\[
(A\wedge B)^{\wedge 1+{\bf k}}\cong A^{\wedge 1+{\bf k}}\wedge B^{\wedge 1+{\bf k}}
\]
and by commuting the smash product with the geometric realizations.
This isomorphism makes $B^{di}$ into a symmetric monoidal functor from the category of orthogonal ring spectra with anti-involution to the category of orthogonal $\mathbb{Z}/2$-spectra.

If the underlying ring spectrum $A$ is commutative, the map $w$ is simply a multiplicative involution on $A$, and $(A,w)$ is a commutative orthogonal $\mathbb{Z}/2$-ring spectrum.

\begin{prop}\label{multBdi}
Let $(A,w)$ be a commutative orthogonal $\mathbb{Z}/2$-ring spectrum. Then $B^{di}(A,w)$ is a commutative augmented $A$-algebra in the category of orthogonal $\mathbb{Z}/2$-spectra.
\end{prop}

\begin{proof} A commutative orthogonal $\mathbb{Z}/2$-ring spectrum $(A,w)$ defines a commutative monoid object in the category of ring spectra with anti-involution. Since $B^{di}$ is a symmetric monoidal functor it follows that $B^{di}(A,w)$ is a commutative orthogonal $\mathbb{Z}/2$-ring spectrum.

The algebra structure is induced by the inclusion of the zero-simplices $A \to B^{di}(A,w)$ which is a monoidal natural transformation. The augmentation map $B^{di}(A,w)\to A$ is the geometric realization of the map of real simplicial spectra defined levelwise by the iterated multiplication $\mu\colon A^{\wedge 1+{\bf k}}\longrightarrow A$.
\end{proof}

There is a similar monoid structure on the B\"{o}kstedt model $\THR(A,w)$ defined in \cite[1.7.1]{Wittvect}, which we recall below. We prove in \ref{comparisonmult} that the comparison of Theorem \ref{comparison} is in fact an equivalence of associative ring spectra (see \ref{Eoo} below  for a remark about the commutative structures). The natural transformation $\THR(A,w)\wedge \THR(B,\sigma)\to \THR(A\wedge B,w\wedge \sigma)$ is obtained by geometrically realizing the maps of $\mathbb{Z}/2$-spectra
\begin{equation}\label{THRmult}
\xymatrix@R=15pt{\displaystyle
(\hocolim_{I^{\times 1+{\bf k}}}\Omega^{\underline{i}}(\mathbb{S} \wedge A_{i_0}\wedge\dots\wedge A_{i_k}))\wedge (\hocolim_{I^{\times 1+{\bf k}}}\Omega^{\underline{j}}(\mathbb{S} \wedge B_{j_0}\wedge\dots\wedge B_{j_k}))\ar[d]^{\cong}
\\
\displaystyle
\hocolim_{I^{\times 1+{\bf k}}\times I^{1+{\bf k}}}(\Omega^{\underline{i}}(\mathbb{S} \wedge A_{i_0}\wedge\dots\wedge A_{i_k}))\wedge (\Omega^{\underline{j}}(\mathbb{S} \wedge B_{j_0}\wedge\dots\wedge B_{j_k}))\ar[d]
\\
\displaystyle \hocolim_{I^{\times 1+{\bf k}}\times I^{1+{\bf k}}}\Omega^{\underline{i}+\underline{j}}((\mathbb{S}  \wedge A_{i_0}\wedge\dots\wedge A_{i_k})\wedge (\mathbb{S} \wedge B_{j_0}\wedge\dots\wedge B_{j_k}))\ar[d]^{\cong}
\\
\displaystyle \hocolim_{I^{\times 1+{\bf k}}\times I^{1+{\bf k}}}\Omega^{\underline{i}+\underline{j}}(\mathbb{S} \wedge A_{i_0}\wedge B_{j_0}\wedge\dots\wedge A_{i_k}\wedge B_{j_k})\ar[d]
\\
\displaystyle \hocolim_{I^{\times 1+{\bf k}}\times I^{1+{\bf k}}}\Omega^{\underline{i}+\underline{j}}(\mathbb{S} \wedge (A\wedge B)_{i_0+j_0}\wedge\dots\wedge (A\wedge B)_{i_k+j_k})\ar[d]^{+_\ast}
\\
\displaystyle \hocolim_{I^{1+{\bf k}}}\Omega^{\underline{l}}(\mathbb{S} \wedge (A\wedge B)_{l_0}\wedge\dots\wedge (A\wedge B)_{l_k})
}
\end{equation}
where the isomorphisms are the canonical ones. The first non-isomorphism commutes smash products and loops and the second is the canonical bimorphism of the smash product of orthogonal spectra. The map labelled $+_\ast$ is the pushforward on homotopy colimits induced by the addition functor $+\colon I\times I\to I$. It is immediate to verify that this natural transformation is equivariant for the $\mathbb{Z}/2$-action on $\THR$ (cf. the proof of \ref{antisymmon}). Thus this natural transformation makes $\THR$ into a monoidal functor from orthogonal ring spectra with anti-involution to orthogonal $\mathbb{Z}/2$-spectra.
The same argument as in \ref{multBdi} proves the following. 

\begin{cor}\label{multTHR}
Let $(A,w)$ be a commutative orthogonal $\mathbb{Z}/2$-ring spectrum. Then $\THR(A,w)$ is an associative $\mathbb{Z}/2$-equivariant ring spectrum.
\end{cor}\qed

\begin{rem}\label{Eoo}
The transformation $\THR(A,w)\wedge \THR(B,\sigma)\to \THR(A\wedge B,w\wedge \sigma)$ is lax monoidal, but it is not symmetric. The problem can be traced back to the fact that the monoidal structure $+\colon I\times I\to I$ is not strictly symmetric.  Nevertheless, $\THH(A)$ has an $E_\infty$-structure when $A$ is commutative (see \cite[1.7.1]{Wittvect}). A comparison of the $E_\infty$-structures on $\THH(A)$ and of an infinity-categorical version of the cyclic Bar construction $B^{cy}A$ is carried out in \cite{NikSchol}.

In our equivariant context, the dihedral Bar construction $B^{di}(A,w)$ is strictly commutative, hence it admits an action of the $\mathbb{Z}/2$-equivariant $E_{\infty}$-operad. We believe that  $\THR(A,w)$ also admits such an action, and we leave the comparison of these two structures as an open question. In the following result we prove that the comparison of \ref{comparison} holds as associative monoids. We see this result as allowing a transition from the B\"{o}kstedt model to the more manageable dihedral Bar construction, and we will only make use of the commutative structure on the latter.
\end{rem}

\begin{theorem}\label{comparisonmult}
Let $(A,w)$ be a commutative flat orthogonal $\mathbb{Z}/2$-ring spectrum. Then the equivalence
\[
\THR(A,w)\simeq B^{di}(A,w)
\]
of Theorem \ref{comparison} is a stable equivalence of associative orthogonal $\mathbb{Z}/2$-ring spectra.
\end{theorem}

\begin{proof}
We need to show that the middle $\mathbb{Z}/2$-spectrum of the zig-zag
\[
A^{\wedge 1+{\bf k}}\longrightarrow\hocolim_{I^{\times 1+{\bf k}}}\Omega^{\underline{i}}(\Sh^{i_0} A\wedge\dots\wedge \Sh^{i_k} A)\stackrel{}{\longleftarrow}\hocolim_{I^{\times 1+{\bf k}}}\Omega^{\underline{i}}(\mathbb{S} \wedge A_{i_0}\wedge\dots\wedge A_{i_k})
\]
from \ref{comparison} admits an associative multiplication, and that the maps are multiplicative. The multiplication is defined by a sequence of maps analogous to the one on $\THR(A,w)$ from diagram (\ref{THRmult}), except that the fourth map is replaced with the one induced by the maps
\[
(\Sh^{i}A)\wedge (\Sh^{j}B)\longrightarrow \Sh^{i+j}(A\wedge B).
\]
These are the maps which correspond to the bimorphism
\[
(\Sh^{i}A)_n\wedge (\Sh^{j}B)_m=A_{i+n}\wedge B_{j+m}\longrightarrow (A\wedge B)_{i+n+j+m}\cong (A\wedge B)_{i+j+n+m}=(\Sh^{i+j}(A\wedge B))_{n+m},
\]
where the isomorphism is induced by the block permutation that shuffles $j$ and $n$.

The multiplicativity of the left-pointing map can be easily reduced to the commutativity of the square
\[
\xymatrix{
(\Sh^{i}A)\wedge (\Sh^{j}B)\ar[d]&\mathbb{S} \wedge A_i\wedge B_j\ar[l]\ar[d]\\
\Sh^{i+j}(A\wedge B)&\mathbb{S} \wedge (A\wedge B)_{i+j}\rlap{\ .}\ar[l]
}
\]
The horizontal maps are adjoint to the maps $A_i\wedge B_j= (\Sh^i A)_0\wedge (\Sh^j A)_0= (\Sh^i A\wedge \Sh^j B)_0$, and therefore this square commutes because the square
\[
\xymatrix{
(\Sh^{i}A\wedge \Sh^{j}B)_0\ar[d]& (\Sh^i A)_0\wedge (\Sh^j A)_0\ar@{=}[l]&A_i\wedge B_j\ar@{=}[l]\ar[d]\\
(\Sh^{i+j}(A\wedge B))_0&& (A\wedge B)_{i+j}\rlap{\ .}\ar@{=}[ll]
}
\]
commutes.

The right-pointing map is the canonical map induced by the inclusion of the object $\underline{i}=(0,\dots ,0)$ into $I^{\times 1+{\bf k}}$. To show that this map is multiplicative we need to show that the diagram
\[
\xymatrix@R=15pt{
A^{\wedge 1+{\bf k}}\wedge B^{\wedge 1+{\bf k}}\ar[r]\ar[d]_-{\cong}
&
\displaystyle \hocolim_{I^{\times 1+{\bf k}}\times I^{1+{\bf k}}}\Omega^{\underline{i}+\underline{j}}(\Sh^{i_0}A\wedge\dots\wedge \Sh^{i_k}A\wedge \Sh^{j_0}B\wedge\dots\wedge \Sh^{j_k}B)\ar[d]^{\cong}
\\
(A\wedge B)^{\wedge 1+{\bf k}}\ar[r]\ar[]!<-1ex,0ex>;[ddr]!<-23ex,0ex>
&
\displaystyle \hocolim_{I^{\times 1+{\bf k}}\times I^{1+{\bf k}}}\Omega^{\underline{i}+\underline{j}}(\Sh^{i_0}A\wedge \Sh^{j_0}B\wedge\dots\wedge \Sh^{i_k}A\wedge \Sh^{j_k}B)\ar[d]
\\
&
\displaystyle \hocolim_{I^{\times 1+{\bf k}}\times I^{1+{\bf k}}}\Omega^{\underline{i}+\underline{j}}(\Sh^{i_0+j_0}(A\wedge B)\wedge\dots\wedge \Sh^{i_k+j_k}(A\wedge B))\ar[d]^{+_\ast}
\\
&
\displaystyle \hocolim_{I^{1+{\bf k}}}\Omega^{\underline{l}}(\Sh^{l_0} (A\wedge B)\wedge\dots\wedge \Sh^{l_k}(A\wedge B))
}
\]
commutes.
This is clear since all the non-vertical maps are the inclusion of the objects $\underline{i}=\underline{j}=(0,\dots ,0)$ and the two lowest vertical maps are the identity on these objects.
\end{proof}

Let $(A,w)$ be a commutative orthogonal $\mathbb{Z}/2$-ring spectrum. Then $A$ is an algebra over the norm $N^{\mathbb{Z}/2}_eA$ in orthogonal $\Z/2$-spectra (see \S\ref{secdihedral}), and it can be replaced by a cofibrant associative $N^{\mathbb{Z}/2}_eA$-algebra $A^c$. Thus the derived smash product  
\[
A\wedge^{\mathbf{L}}_{N^{\mathbb{Z}/2}_eA}A:=A^c\wedge_{N^{\mathbb{Z}/2}_eA}A^c
\]
has the structure of an associative $\Z/2$-equivariant ring spectrum. Similarly, the geometric fixed-points spectrum $\Phi^{\Z/2}_{\mathcal{M}}A^c$ is a $\Phi^{\Z/2}_{\mathcal{M}}N^{\Z/2}_eA$-algebra, and if $A$ is flat $\Phi^{\Z/2}_{\mathcal{M}}N^{\Z/2}_eA\cong A$. Thus the derived smash product 
\[
\Phi^{\Z/2}A\wedge^{\mathbf{L}}_{A}\Phi^{\Z/2}A:=\Phi^{\Z/2}_{\mathcal{M}}A^c\wedge_A\Phi^{\Z/2}_{\mathcal{M}}A^c
\]
is an associative ring spectrum.

\begin{cor}\label{cor:fixed-pt-alg-str}
Let $(A,w)$ be a commutative orthogonal $\mathbb{Z}/2$-ring spectrum, whose underlying $\Z/2$-spectrum is flat. There is a stable  equivalence of associative  $\Z/2$-equivariant ring spectra
\[
\THR(A)\simeq A\wedge^{\mathbf{L}}_{N^{\mathbb{Z}/2}_eA}A,
\]
and a stable equivalence of associative ring spectra between the corresponding geometric fixed points
\[
\Phi^{\Z/2}\THR(A)\simeq \Phi^{\Z/2}A\wedge^{\mathbf{L}}_{A}\Phi^{\Z/2}A.
\]
\end{cor}
\begin{proof} Theorem \ref{comparisonmult} identifies $\THR(A)$ and $B^{di}A$ multiplicatively, and the same argument of Corollary \ref{derived smash} shows that $B^{di}A$ computes the derived smash product as associative rings. Next, the proof of Theorem \ref{geofixcalc} and Remark \ref{noC} show that the derived geometric fixed points $\Phi^{\Z/2}\THR(A)$ is equivalent as an algebra to $\Phi_{\mathcal{M}}^{\Z/2}(A^c\wedge_{N^{\mathbb{Z}/2}_eA}A^c)$ by making sure that all the cofibrant replacements take place in the category of associative algebras. But the latter is isomorphic to the derived smash $\Phi^{\Z/2}_{\mathcal{M}}A^c\wedge_A\Phi^{\Z/2}_{\mathcal{M}}A^c$ by \cite[Proposition B.203]{HHR}
\end{proof}


%
%

\subsection{Cofinality}\label{seccof}

We recall that a full subcategory $B\subset C$ is cofinal if every object in $C$ is a retract of an object in $B$. Given a spectral category $\mathscr{C}$, we say that a full subcategory $\mathscr{B}\subset \mathscr{C}$ is cofinal if the inclusion of the underlying categories $B\subset C$, obtained by taking the $0$-th space of the mapping spectra, is cofinal.

\begin{theorem}\label{cofinality}
Let $(\mathscr{C},D,\eta)$ be a flat spectral category with duality, and $\mathscr{B}\subset \mathscr{C}$ a cofinal subcategory which is invariant under the duality. Then the inclusion $\iota\colon \mathscr{B}\to \mathscr{C}$ induces a stable equivalence of $\Z/2$-spectra
\[
\iota \colon \THR(\mathscr{B},\eta)\stackrel{\simeq}{\longrightarrow}\THR(\mathscr{C},\eta).
\]
\end{theorem}

\begin{proof}
For every object $c$ in $C$ let us chose an object $r(c)$ in $B$ and maps $i_c\colon c\to r(c)$, $p_c\colon r(c)\to c$ such that $p_c\circ i_c=\id_c$. We make this choice in such a way that if $c$ is an object of $\mathscr{B}$ then $r(c)=c$ and $i_c=p_c=\id_c$. We define a map $r\colon \sd_eN^{di}_\wedge\mathscr{C}\to \sd_eN^{di}_\wedge\mathscr{B}$ as follows. By abuse of notation we represent an element of $\sd_eN^{di}_\wedge\mathscr{C}$ as a string of arrows
\[\underline{f}=(Dc_{2k+1}\stackrel{f_0}{\longleftarrow}c_0\stackrel{f_1}{\longleftarrow}c_1\stackrel{}{\longleftarrow}\dots\stackrel{f_k}{\longleftarrow}c_{k}\stackrel{f_{k+1}}{\longleftarrow}Dc_{k+1}\stackrel{f_{k+2}}{\longleftarrow}\dots\stackrel{f_{2k+1}}{\longleftarrow}Dc_{2k+1}).
\]
For simplicity let us denote $p_j:=p_{c_j}$ and $i_j:=i_{c_j}$. Then $r$ sends $\underline{f}$ to
\[
\xymatrix@C=35pt{Drc_{2k+1}&&\ar[ll]_-{D(p_{2k+1})f_0p_0}
rc_0&\ar[l]_-{i_{0}f_1p_1}
rc_1
&
\dots\ar[l]_-{i_{1}f_2p_2}
&
rc_{k}\ar[l]_-{i_{k-1}f_kp_k}
\\
\ar[rr]_-{D(p_{2k}) f_{2k+1}D(i_{2k+1})}Drc_{2k+1}&&Drc_{2k}\ar[r]&\dots\ar[rr]_-{D(p_{k+1}) f_{k+2}D(i_{k+2})}
&&
Drc_{k+1}\ar[u]_-{i_k f_{k+1}D(i_{k+1})}&.
}
\]
The compositions by the maps in the underlying category are defined by the maps of spectra
\[
\mathscr{C}(d,c)_0\wedge \mathscr{C}(c',d)\longrightarrow \mathscr{C}(c',c)
\]
obtained by restricting the composition maps, and similarly for the left compositions.
It is easy to verify that this map commutes with the anti-involution regardless of the choices of $r(c)$, $i_c$ and $p_c$, and that it commutes with the face maps. This map does not preserve degeneracies, but since $(\mathscr{C},D,\eta)$ is flat we can just as well work with thick realizations, by \ref{properness}. Since we chose $r(c)=c$ when $c$ lies in $\mathscr{B}$ the map $r$ is a retraction for $\iota$.

We recall that a homotopy of semisimplicial objects is a collection of maps $h_{k+1,l}\colon X_{k}\to Y_{k+1}$, for $0\leq l\leq k$ that satisfies the compatibility conditions with the face maps of \cite[5.1]{MaySimp}. In the presence of degeneracies this is the same datum of a simplicial map $X\times \Delta^1\to Y$, but for semisimplicial objects this is generally more structure. Such a collection of maps induces a homotopy from the thick realization of $d_{0}h_{k+1,0}\colon X_k\to Y_k$ to the thick realization of $d_{k+1}h_{k+1,k}\colon X_k\to Y_k$ (see e.g. \cite{ERW}).

Our goal is to construct a homotopy between the identity functor of $\mathscr{C}$ and the composite $\iota\circ r$. We define a semisimplicial homotopy on the subdivided dihedral nerve
\[h_{k+1,l}\colon (N^{di}_\wedge \mathscr{C})_{2k+1}\longrightarrow (N^{di}_\wedge \mathscr{C})_{2k+3}\]
by a formula similar to the homotopy of \ref{homotopiesTHR}. We  send $\underline{f}$ to
\[\xymatrix@C=23pt{
&Drc_{2k+1}
&&\ar[ll]_-{D(p_{2k+1})f_0p_0}
rc_0
&rc_1\ar[l]_-{i_{0}f_1p_1}
&\ar[l]_-{i_{1}f_2p_2}\dots
&rc_{l} \ar[l]_-{i_{l-1}f_lp_l}&
\\
&&c_{k}\ar[r]^{f_k}
&c_{k-1}\ar[r]^{f_{k-1}}
&\dots\ar[r]^{f_{l+2}}
&c_{l+1}\ar[r]^-{f_{l+1}}
&c_l\ar[u]_-{i_l}
\\
&&Dc_{k+1}\ar[u]^-{f_{k+1}}
&Dc_{k+2}\ar[l]_-{f_{k+2}}
&\dots\ar[l]_-{f_{k+3}}
&&Dc_{2k+1-l}\ar[ll]_-{f_{2k+1-l}}
\\
Drc_{2k+1}\ar[rr]_-{D(p_{2k})f_{2k+1}D(i_{2k+1})}
&&\dots\ar[r]
&Drc_{2k+2-l}\ar[rrr]_-{D(p_{2k+1-l})f_{2k+2-l}D(i_{2k+2-l})}
&&
&Drc_{2k+1-l}\ar[u]_-{D(i_{2k+1-l})} \rlap{\ .}
}\]
A direct calculation shows that the maps $h_{k+1,l}$ are equivariant and that they define a simplicial homotopy. The map $d_0h_{k+1,0}$ sends $\underline{f}$ to
\[
d_0(Drc_{2k+1}\xleftarrow{D(p_{2k+1})f_0p_0}rc_0\xleftarrow{i_0}c_0\stackrel{f_1}{\longleftarrow}c_1\leftarrow\dots\xleftarrow{f_{2k+1}}Dc_{2k+1}\xleftarrow{D(i_{2k+1})}Drc_{2k+1} )=\underline{f}.
\]
The map $h_{k+1,k}$ sends $\underline{f}$ to
\[
\xymatrix{
Drc_{2k+1}
&&
\ar[ll]_-{D(p_{2k+1})f_0p_0} rc_0
&\ar[l]_-{i_0f_1p_1}rc_1
&\ar[l]\dots
&
\ar[l]_-{i_{k-1}f_{k}p_{k}}rc_{k}
&\ar[l]_-{i_k}c_k
\\
&Drc_{2k+1}
\ar[rr]_-{D(p_{2k})f_{2k+1}D(i_{2k+1})}
&
&
\ar[rr]_-{D(p_{k+1})f_{k+2}D(i_{k+2})}
\dots
&&
\ar[r]_-{D(i_{k+1})} Drc_{k+1}
&
\ar[u]_-{f_{k+1}}Dc_{k+1}\rlap{\ ,}
}
\]
and $d_{k+1}h_{k+1,k}$ is therefore equal to $\iota\circ r$.
\end{proof}

Let $(R,w,\epsilon)$ be a  Wall antistructure. We recall from Example \ref{excatdual} that the category $\mathcal{P}_R$ of finitely generated projective $R$-modules inherits a non-strict duality. We let $\mathcal{F}_R$ be the full subcategory of $\mathcal{P}_R$ of free $R$-modules. This is clearly a cofinal subcategory, and therefore Theorem \ref{cofinality} gives the following.

\begin{cor}\label{cofinalitymodules}
The inclusion $\mathcal{F}_R\subset\mathcal{P}_R$ induces a stable equivalence of $\Z/2$-spectra
\[
\THR(H\mathcal{F}_R,\eta)\stackrel{\simeq}{\longrightarrow}\THR(H\mathcal{P}_R,\eta).\]
\end{cor}\qed

\subsection{Morita Invariance}\label{secmorita}

We prove that the dihedral Bar construction satisfies a real version of Morita invariance, and we deduce that the dihedral Bar construction of a discrete ring is equivalent to the the dihedral Bar construction of its category of finitely generated projective modules. The analogous statements for the B\"{o}kstedt model of $\THR$ have been proven in the first author's thesis \cite{thesis}.

Let $(A,w)$ be an orthogonal ring spectrum with anti-involution, and $n\geq 1$ an integer. We consider the non-unital ring spectrum $M^{\vee}_nA:=\bigvee_{n\times n}A$ as a model for the matrix ring of $A$. We recall that its multiplication is defined by the map
\[
(\bigvee_{n\times n}A)\wedge (\bigvee_{n\times n}A)\cong \bigvee_{n\times n\times n\times n}A\wedge A\longrightarrow \bigvee_{n\times n}A
\]
which sends the $(i,j,k,l)$-summand to the $(i,l)$-summand via the multiplication map of $A$ if $j=k$, and to the basepoint otherwise. This ring spectrum has an anti-involution
\[
w^T\colon\bigvee_{n\times n}A\stackrel{\bigvee w}{\longrightarrow}\bigvee_{n\times n}A\stackrel{\tau}{\longrightarrow}\bigvee_{n\times n}A
\]
where $\tau$ is the automorphism of $n\times n$ that swaps the product factors.
\begin{rem}\label{discretematrices}
Let $(R,w)$ be a discrete ring with anti-involution. The matrix ring $M_nR=\bigoplus_{n\times n}R$ has an anti-involution defined by applying $w$ to the entries and by transposing the matrix. The inclusion of indexed wedges into indexed products defines an equivalence of ring spectra with anti-involution
\[
M^{\vee}_nHR\stackrel{\simeq}{\longrightarrow}\prod_{n\times n}HR\cong HM_nR.
\]
\end{rem}

Now suppose that $(A,w,\epsilon)$ is a spectral antistructure. Since $M^{\vee}_nA$ does not contain the diagonal matrices, $\epsilon$ does not define a unit in $(M^{\vee}_nA)_0$ and $M^{\vee}_nA$ cannot be considered as a spectral antistructure in the sense of Example \ref{excatdual}. However, $\epsilon$ acts on $M^{\vee}_nA$ by entrywise left and right multiplication, and the definition of the involution on $N^{di}_\wedge$ of \S\ref{seccatdual} makes sense for $M_{n}^{\vee}A$.
Moreover since $M^{\vee}_nA$ is not unital the dihedral nerve $N^{di}_\wedge M^{\vee}_nA$ does not have degeneracies. We define the dihedral Bar construction  $B^{di}_\wedge M^{\vee}_nA$ to be the thick geometric realization of the semisimplicial spectrum defined by the Segal edgewise subdivision of $N^{di}_\wedge M^{\vee}_nA$. The same considerations apply to B\"{o}kstedt's model.

\begin{theorem}\label{Morita} Let $(A,w,\epsilon)$ be a flat spectral antistructure.
Then the inclusion $A\to M^{\vee}_nA$ of the $(1,1)$-wedge summand induces a homotopy equivalence of $\mathbb{Z}/2$-spectra
\[
\THR(A,w,\epsilon)\stackrel{\simeq}{\longrightarrow} \THR(M^{\vee}_nA,w^T,\epsilon)
\]
for every $n\geq 1$.
\end{theorem}

\begin{rem}
A statement analogous to \ref{Morita} holds for $\THR$ with coefficients in an $(A,w,\epsilon)$-bimodule $(M,j)$, as well as for categories with duality (see \cite{thesis}).
\end{rem}

\begin{rem} Even though $M^{\vee}_nA$ is not unital, the geometric fixed points of $B^{di}_\wedge(M^{\vee}_nA,w^T,\epsilon)$ are equivalent to the realization of a simplicial spectrum.
We recall that there is an isomorphism of semisimplicial orthogonal spectra between the geometric fixed points of $B^{di}_\wedge(M^{\vee}_nA,w^T,\epsilon)$ and the thick realization of the two-sided Bar construction
\[
\Phi^{\Z/2}\sd_e N^{di}_\wedge M^{\vee}_nA\cong N_\wedge(\Phi^{\Z/2}M^{\vee}_nA,M^{\vee}_nA,\Phi^{\Z/2}M^{\vee}_nA).
\]
Since geometric fixed points commute with indexed coproducts there is an isomorphism
\[\Phi^{\Z/2}M^{\vee}_nA=\Phi^{\Z/2}\bigvee_{n\times n}A\cong \bigvee_n\Phi^{\Z/2}A.\]
Under this isomorphism the left action of $M^{\vee}_nA$ on $\Phi^{\Z/2}M^{\vee}_nA$ of \ref{geofixcalc} corresponds to the action
\[
M^{\vee}_nA\wedge \bigvee_n\Phi^{\Z/2}A\cong\bigvee_{n\times n\times n}A\wedge \Phi^{\Z/2}A \longrightarrow \bigvee_n\Phi^{\Z/2}A
\]
which sends the summands $(i,j,j)$ to the $i$-summand via the action map $A\wedge \Phi^{\Z/2}A \to \Phi^{\Z/2}A$ of \ref{geofixcalc} , and it is trivial on the components $(i,j,k)$ with $k\neq j$. The right action admits a similar description. This gives an isomorphism of semisimplicial spectra
\[
\Phi^{\Z/2}\sd_eN^{di}_\wedge M^{\vee}_nA\cong N_\wedge(\bigvee_n\Phi^{\Z/2}A,M^{\vee}_nA,\bigvee_n\Phi^{\Z/2}A).
\]
Now we observe that the left action of $M^{\vee}_nA$ on $\bigvee_n\Phi^{\Z/2}A$ above extends to an action of the unital matrix ring $M_nA=\prod_{n}\bigvee_nA$, by the composite
\[
(\prod_{n}\bigvee_nA)\wedge (\bigvee_n\Phi^{\Z/2}A)\longrightarrow\bigvee_n\prod_n\bigvee_{n}(A\wedge \Phi^{\Z/2}A) \longrightarrow \bigvee_n\Phi^{\Z/2}A
\]
of the canonical map and the map that sends the $(k,j_1,\dots,j_n)$-summand to the $j_k$-summand via the action map $A\wedge \Phi^{\Z/2}A \to \Phi^{\Z/2}A$. A similar extension exists for the right action. Thus the equivalence $M^{\vee}_nA\stackrel{\simeq}{\to}M_nA$ induces a levelwise equivalence of spectra
\[
\Phi^{\Z/2}B^{di}_\wedge M^{\vee}_nA\cong B_\wedge(\bigvee_n\Phi^{\Z/2}A,M^{\vee}_nA,\bigvee_n\Phi^{\Z/2}A)\stackrel{\simeq}{\longrightarrow}B_\wedge(\bigvee_n\Phi^{\Z/2}A,M_nA,\bigvee_n\Phi^{\Z/2}A).
\]
where the target admits degeneracies.
\end{rem}

\begin{proof}[Proof of \ref{Morita}]
We drop the anti-involution $w$ and the unit $\epsilon$ from the notation.
By \ref{properness} the thick and the standard realizations of the Segal edgewise subdivision of $N^{di}_\wedge A$ are equivalent. It is therefore sufficient to show that the map
\[
\iota\colon \sd_e N^{di}_\wedge A\stackrel{}{\longrightarrow}\sd_e N^{di}_\wedge M^{\vee}_nA
\]
induces an equivalence on thick geometric realizations.
The $k$-simplices of the dihedral nerve $N^{di}_\wedge M^{\vee}_nA$ are isomorphic to
\[
(N_{\wedge}^{di}M^{\vee}_nA)_k=((n\times n)_+\wedge A)^{\wedge k+1}\cong (n\times n)^{\times k+1}_+\wedge A^{\wedge k+1}.
\]
We define a ``trace'' map $\tr\colon \sd_e N^{di}_\wedge M^{\vee}_nA\to \sd_e N^{di}_\wedge A$
by sending the summand indexed by $(\underline{i},\underline{j})=(i_{2k+2},j_0),(i_1,j_1),\dots ,(i_{2k+1},j_{2k+1})$ to
\[
\tr(\underline{i},\underline{j},x)=\left\{
\begin{array}{ll}
x&\mbox{if $j_0=i_1, j_1=i_2,\dots,j_{2k}=i_{2k+1}, j_{2k+1}=i_{2k+2}$},
\\
\ast&\mbox{otherwise}.
\end{array}
\right.
\]
This is analogous to the trace map for Hochschild homology, defined for example in \cite[1.2.1]{Loday}. A similar map is defined in \cite[1.6.18]{ringfctrs} for $\THH$ but it is unfortunately not semisimplicial.
The map $\tr$ is clearly an equivariant retraction for $\iota$.

We define a semisimplicial homotopy $h_{k+1,l}\colon (N^{di}_\wedge M^{\vee}_nA)_{2k+1}\to (N^{di}_\wedge M^{\vee}_nA)_{2k+3}$ on the subdivided dihedral nerve.
We define $C_0=(n\times n)^{\times 2k+2}$, and for every $1\leq l\leq k$ we let $C_l$ be the subset of $C_0$ of the sequences $(i_{2k+2},j_0),(i_{1},j_1),\dots,(i_{2k+1},j_{2k+1})\in n\times n$ that satisfy
\[
\begin{array}{llll}
j_0=i_1& j_1=i_2&\dots& j_{l-1}=i_l\\
i_{2k+2}=j_{2k+1}& i_{2k+1}=j_{2k}&\dots&  i_{2k+3-l}=j_{2k+2-l}.
\end{array}
\]
We define $h_{k+1,l}$ by sending $(\underline{i},\underline{j},x)=((i_{2k+2},j_0),(i_{1},j_1),\dots ,(i_{2k+1},j_{2k+1}),x)$ to
\[((1,1),\overbrace{(1,1),\dots, (1,1)}^{l},(1,j_{l}),(i_{l+1},j_{l+1}),
\dots,(i_{2k+1-l},j_{2k+1-l}),(i_{2k+2-l},1),\overbrace{(1,1),\dots, (1,1)}^{l},s_lx)
\]
if $(\underline{i},\underline{j})\in C_l$, and to the basepoint otherwise.
Here $s_l\colon A^{\wedge 2k+1}\to A^{\wedge 2k+3}$ is the $l$-degeneracy map of the subdivision of $N^{di}_\wedge A$. It is straightforward to verify that these maps define a semisimplicial homotopy, and that $d_0h_{k+1,0}=\id$. The map  $d_{k+1}h_{k+1,k}$ sends $(\underline{i},\underline{j},x)$ to
\[
d_{k+1}d_{k+2}(\!\overbrace{(1,1),\dots, (1,1)}^{k+1},(1,j_{k}),(i_{k+1},j_{k+1}),(i_{k+2},1),\overbrace{(1,1),\dots, (1,1)}^{k},s_kx\!)\!=\!(\!\overbrace{(1,1),\dots, (1,1)}^{2k+2},x\!)
\]
if $(\underline{i},\underline{j})\in C_k$ and $j_{k}=i_{k+1}, j_{k+1}=i_{k+2}$, and to the basepoint otherwise. This is precisely the map $\iota\circ\tr$
\end{proof}

Now let $(R,w,\epsilon)$ be a Wall antistructure, considered as an $Ab$-enriched category with duality with one object. Let $(\mathcal{P}_R,D,\eta)$ be the category of finitely generated projective right $R$-modules with the duality that sends an $R$-module $P$ to the $R$-module
\[D(P):=\hom_R(P,R)\]
of Example \ref{excatdual}. There is a morphism of categories with duality $\iota\colon R\to \mathcal{P}_R$ whose underlying functor sends the unique object to $R$ with right multiplication, and a morphism $r\in R$ to $r\cdot(-)\colon R\to R$. The compatibility between the dualities is given by the canonical isomorphism $R\cong \hom_R(R,R)$ that sends $1$ to $w$.

\begin{cor}\label{Moritamodules}
The functor $\iota\colon R\to \mathcal{P}_R$ induces a stable equivalence of $\Z/2$-spectra
\[\THR(R,w,\epsilon)\stackrel{\simeq}{\longrightarrow}\THR(H\mathcal{P}_R,D,\eta).\]
\end{cor}

\begin{proof}
Let $\mathcal{F}_R$ be the category of free modules, and $\mathcal{F}^{\leq n}_R$ its full subcategory of modules of rank less than or equal to $n$. The duality of $\mathcal{P}_R$ restricts to both $\mathcal{F}_R$ and $\mathcal{F}^{\leq n}_R$ and the map $\iota$ factors as
\[
\xymatrix{
B^{di}_\wedge(R,w,\epsilon)\ar[d]_{\iota}\ar[r]^-{\simeq}
&
\hocolim_{n}B^{di}_\wedge(M^{\vee}_n(HR),w,\epsilon)\ar[r]^-{\simeq}
&
\hocolim_{n}B^{di}_\wedge(HM_nR,w,\epsilon)\ar[d]
\\
B^{di}_\wedge(H\mathcal{P}_R,D,\eta)&
B^{di}_\wedge(H\mathcal{F}_R,D,\eta)\ar[l]
&
\hocolim_{n}B^{di}_\wedge(H\mathcal{F}^{\leq n}_R,D,\eta)
\ar[l]
}
\]
where the homotopy colimits are taken with respect to the maps $M^{\vee}_nHR\to M^{\vee}_{n+1}HR$, $M_nR\to M_{n+1}R$ and $\mathcal{F}^{\leq n}_R\to \mathcal{F}^{\leq n+1}_R$ that increase the size of a matrix by adding a row and a column of zeros. These operations are not unital, and therefore the spectra of the diagrams are all obtained by taking thick realizations. The top row of the diagram consists of equivalences by \ref{Morita} and \ref{discretematrices}.

The inclusion $M_nR\to  \mathcal{F}^{\leq n}_R$ is clearly cofinal, and therefore the vertical map is an equivalence by \ref{cofinality}. Similarly, the last map is an equivalence by \ref{cofinalitymodules}. The remaining map is the thick realization of the equivalence
\[
\hocolim_{n}\bigvee_{0\leq a_0,\dots,a_k\leq n}(HM_{a_0,a_k}R)\wedge\dots\wedge (HM_{a_k,a_{k-1}}R)\stackrel{\simeq}{\longrightarrow}\!\bigvee_{0\leq a_0,\dots,a_k}\!(HM_{a_0,a_k}R)\wedge\dots\wedge (HM_{a_k,a_{k-1}}R)
\]
where $M_{a,b}R:=\bigoplus_{a\times b}R$ is the abelian group of $a\times b$-matrices.
\end{proof}

\section{Calculations}\label{seccalc}

\subsection{The Mackey functor of components of \texorpdfstring{$\THR$}{THR}}\label{secpi0}

Let $(A,w)$ be a flat ring spectrum with anti-involution, whose underlying $\Z/2$-spectrum is connective.
The aim of this section is to compute the Mackey functor of components
\[
\xymatrix@C=40pt{
\pi_0 \THR(A) \ar@(dl,dr)_-{w}\ar@<.5ex>[r]^-{\tran}&\pi^{\mathbb{Z}/2}_0 \THR(A) \ar@<.5ex>[l]^-{\res}
}
\]
based on the description of $\THR(A)$ as the derived smash product $A \wedge^{\mathbf {L}}_{N^{\mathbb{Z}/2}_e A} A$.

We start by describing the multiplicative structure on $\underline{\pi}_0 A$ induced by the multiplication of $A$. The abelian group $\pi_0A$ has a ring structure, and $\pi_0(w)=w$ is an anti-involution.
There is also a multiplicative action
\[\pi_0A\otimes \pi^{\mathbb{Z}/2}_0A\xrightarrow{N\otimes\id}\pi^{\mathbb{Z}/2}_0(N^{\Z/2}_eA)\otimes \pi^{\mathbb{Z}/2}_0A \stackrel{\wedge}{\longrightarrow} \pi^{\mathbb{Z}/2}_0(N^{\Z/2}_eA\wedge A)\longrightarrow
 \pi^{\mathbb{Z}/2}_0 A,\]
where $N$ is the external norm and the last map is the left $N^{\mathbb{Z}/2}_e A$-module structure on $A$ of \S\ref{secdihedral}. We will denote the value of this map at an element $(a, x)$ by $a\cdot x$. Although $\pi^{\mathbb{Z}/2}_0 A$ is a priori not a ring, it still has a preferred unit element $1 \in \pi^{\mathbb{Z}/2}_0 A$ which restricts to $1 \in \pi_0A$. This element is given by the homotopy class of the $\Z/2$-equivariant unit map $\mathbb{S} \to A$.
Combining the action above with the unit $1 \in \pi^{\mathbb{Z}/2}_0 A$, we get a multiplicative transfer
\[N\colon\pi_0A \longrightarrow \pi^{\mathbb{Z}/2}_0 A\]
by sending $a$ to its multiplicative action on the unit $a \cdot 1$. When $(A,w)$ is a commutative $\Z/2$-equivariant ring spectrum this multiplicative transfer coincides with the multiplicative norm defined in \cite{Brun}.

\begin{theorem}\label{pi0relations} Let $(A,w)$ be a flat ring spectrum with anti-involution whose underlying orthogonal $\Z/2$-spectrum is connective. Then the Mackey functor $\underline{\pi}_0\THR(A)$ is naturally isomorphic to the Mackey functor
\[
\xymatrix@C=40pt{
\pi_0A/[\pi_0A, \pi_0A] \ar@<6ex>@(dl,dr)_-{w}\ar@<.5ex>[r]^-{\tran}& (\pi^{\mathbb{Z}/2}_0 A \otimes \pi^{\mathbb{Z}/2}_0 A)/T \ar@<.5ex>[l]^-{\res}
},
\]
where $[\pi_0A, \pi_0A]$ is the commutator subgroup, and $T$ is the subgroup generated by the following elements:
\begin{itemize}
\item[i)] $x \otimes a \cdot y-\omega (a) \cdot x \otimes y$, for $x,y \in \pi^{\mathbb{Z}/2}_0 A$ and $a \in \pi_0A$;
\item[ii)] $x \otimes \tran(a \res(y) w(b))-\tran(w(b)\res(x)a) \otimes y$, for $x,y \in \pi^{\mathbb{Z}/2}_0 A$ and $a, b \in \pi_0A$.
\end{itemize}
The restriction and the transfer are given respectively by
\[\res(x \otimes y)=\res(x)\res(y)\ \ \ \ \ \ \ \ \ \ \  \mbox{and} \ \ \ \ \ \ \ \ \ \ \ \ \  \tran(a)= \tran(a) \otimes 1,\]
where $1 \in \pi^{\mathbb{Z}/2}_0 A$ is the above mentioned unit.
\end{theorem}

\begin{cor} \label{compi0} Let $(A,w)$ be a connective $\Z/2$-equivariant commutative ring spectrum with flat underlying $\Z/2$-spectrum. Then $[\pi_0A, \pi_0A]=0$, and the relations in Theorem \ref{pi0relations} simplify as follows:
\begin{itemize}
\item[i)] $x \otimes N(a)y-xN(a) \otimes y$, where $N\colon\pi_0A \to \pi_0^{\Z/2}A$ is the multiplicative norm;
\item[ii)] $x \otimes \tran(a)y- x \tran(a) \otimes y$.
\end{itemize}
Moreover the additive relations $i)$ and $ii)$ generate an ideal, and $(\pi^{\mathbb{Z}/2}_0 A \otimes \pi^{\mathbb{Z}/2}_0 A)/T$ is a commutative ring. The Tambara structure on $\underline{\pi}_0\THR(A)$ from Proposition \ref{multBdi} is given by  the multiplicative transfer $a\mapsto N(a) \otimes 1$.
\end{cor}

The rest of the subsection is devoted to the proof of these two results.
We will make use of the box product $\Box$ of Mackey functors (see e.g., \cite{LewisGreen} and \cite{Bouc}), and the fact that the monoids for $\Box$ are the Green functors. We recall that the structure of a $\Z/2$-Green functor on a Mackey functor
\[M=\big(
\xymatrix@C=40pt{
M(\mathbb{Z}/2)\ar@(dl,dr)_-{\tau}\ar@<.5ex>[r]^-{\tran}&M(\ast)\ar@<.5ex>[l]^-{\res}
}\big)
\]
is the datum of ring structures on $M(\ast)$ and $M(\Z/2)$ such that the restriction $\res \colon M(\ast) \to M(\Z/2)$ and involution $\tau \colon M(\Z/2) \to M(\Z/2)$ are ring homomorphisms, and where the Frobenius reciprocity relations
\[\tran (a \res(x))=\tran(a)x\ \ \ \ \  \mbox{and} \ \ \ \ \  \tran(\res(x)a)=x\tran(a)\]
hold for all $a\in M(\Z/2)$ and $x\in M(\ast)$.
It follows from \cite{LM06} that the graded homotopy Mackey functor $\underline{\pi}_\ast$ is a lax symmetric monoidal functor with respect to the smash and the box product. In particular $\underline{\pi}_0$ sends equivariant ring spectra to Green functors and equivariant module spectra to Mackey modules over Green functors.

Let $(A,w)$ be a flat ring spectrum with anti-involution whose underlying $\Z/2$-spectrum is  connective. We know from \cite{LM06} that there is a strongly convergent spectral sequence of Mackey functors
\[ \Tor_{p,q}^{{\underline{\pi}}_\ast(N^{\mathbb{Z}/2}_e A)}({\underline{\pi}}_\ast A, {\underline{\pi}}_\ast A) \Rightarrow {\underline{\pi}}_{p+q}(A \wedge^{\mathbf {L}}_{N^{\mathbb{Z}/2}_e A} A), \]
where $\Tor$ is the left derived functor of the $\Box$-product of Mackey functors. It follows immediately from this spectral sequence that we get an isomorphism of Mackey functors
\[\underline{\pi}_0\THR(A) \cong \underline{\pi}_0(A \wedge^{\mathbf {L}}_{N^{\mathbb{Z}/2}_e A} A) \cong \underline{\pi}_0(A) \Box_{{\underline{\pi}}_0(N^{\mathbb{Z}/2}_e A)} \underline{\pi}_0(A). \]

In order to make the latter expression more computable we have to understand the Green functor ${\underline{\pi}}_0(N^{\mathbb{Z}/2}_e A)$ and the left and right ${\underline{\pi}}_0(N^{\mathbb{Z}/2}_e A)$-module structures on $\underline{\pi}_0A$.
The Mackey functor of components of the norm construction has already been computed in several places in the literature in terms of the norm of Mackey functors, see for example \cite{Hoy, Maz, Ul2}. For the purpose of our computation it is convenient to give an explicit description of ${\underline{\pi}}_0(N^{\mathbb{Z}/2}_e A)$ in terms of generators and relations, which is similar to a ring of $2$-truncated Witt vectors. This is essentially contained in \cite{Maz}.

\begin{defn} \label{non-comm Witt} Let $S$ be a ring. We let $(S \otimes S)_{\Z/2}$ denote the coinvariants with respect to the flip automorphism $\tau(a \otimes b)=b \otimes a$ of $S\otimes S$.
The $2$-truncated non-commutative ring of Witt vectors $W^{\otimes}_2(S)$ is the set $S \times (S \otimes S)_{\Z/2}$ with the ring structure defined by the operations
\begin{itemize}
\item[i)] $(a,c)+(a',c')= (a+a', c+c'-a \otimes a')$
\item[ii)] $(a,c)(a',c')=(aa', (a\otimes a)c'+c(a' \otimes a') + cc'+ c\tau(c'))$,
\end{itemize}
for every $a,a'\in S$ and $c,c'\in (S \otimes S)_{\Z/2}$. We note the expression ii) is symmetric, since $\tau(c)c'=c\tau(c')$ in $(S \otimes S)_{\Z/2}$.
\end{defn}

It is easy to check that these formulas give an associative ring with zero $(0,0)$ and unit $(1,0)$. We also note that the expressions
\[\mbox{$(a\otimes a)c'$,\ \ \ $c(a' \otimes a')$\ \ \ and\ \ \ $cc'+ c\tau(c')$}\]
are well defined in $(S \otimes S)_{\Z/2}$ (here we use the ring structure on $S\otimes S$) although $(S \otimes S)_{\Z/2}$ is not a ring itself.

\begin{rem}
If the ring $R$ is a commutative and solid ring, i.e. the multiplication $R \otimes_{\Z} R \to R$ is an isomorphism, then $W^{\otimes}_2(R)$ is isomorphic to the ring of $2$-truncated Witt vectors $W_2(R)$. For example, we see that $W^{\otimes}_2(\mathbb{F}_2) \cong \Z/4$ and $W^{\otimes}_2(\mathbb{F}_p) \cong \mathbb{F}_p \times \mathbb{F}_p$ for every odd prime $p$.
There are also lifts of the ghost coordinates and of the Verschiebung to the non-commutative case. These are given by the ring homomorphisms
\[W^{\otimes}_2(S) \xrightarrow{(w_0, w_1)}S \times (S \otimes S)^{\Z/2} \]
defined by $w_0(a,c)=a$ and $w_1(a,c)=a \otimes a+ c+\tau(c)$, for every $a\in S$ and $c\in (S \otimes S)_{\Z/2}$, and by
\[V\colon S\otimes S\xrightarrow{(0,\id)} S\times S\otimes S\longrightarrow W^{\otimes}_2(S),\]
respectively, where the second map is the product of the identity and the projection map $S\otimes S\to (S\otimes S)_{\Z/2}$.
\end{rem}

The ring $W^{\otimes}_2(S)$ is in fact part of a $\Z/2$-Green functor $\mathbb{W}^{\otimes}_2(S)$, which is defined by
\[\mathbb{W}^{\otimes}_2(S)=\big(
\xymatrix@C=40pt{
S\otimes S\ar@(dl,dr)_-{\tau}\ar@<.5ex>[r]^-{V}&W^{\otimes}_2(S)\ar@<.5ex>[l]^-{w_1}
}\big).
\]
Checking the Frobenius reciprocity laws is straightforward. The Mackey functor $\underline{\pi_0}A$ is naturally a left and right module over the Green functor $\mathbb{W}^{\otimes}_2(\pi_0A)$, where $\pi_0A$ here is just considered as a ring (without an anti-involution). The action maps of the left module structure are
\[\xymatrix@R=7pt{(\pi_0A \otimes \pi_0A)\otimes \pi_0A\ar[r]^-{\mu^{l}_1}& \pi_0A\\
a \otimes a' \otimes b \ar@{|->}[r] & abw(a')
}
\ \ \ \ \ \ \ \ \
\xymatrix@R=7pt{W^{\otimes}_2(\pi_0 A) \otimes \pi^{\mathbb{Z}/2}_0 A\ar[r]^-{\mu^{l}_2}& \pi^{\mathbb{Z}/2}_0 A\\
(a,c) \otimes x\ar@{|->}[r]& a \cdot x+ \tran(\mu^{l}_1(c\otimes\res(x)))}
\]
where $a \in \pi_0 A$, $c \in (\pi_0 A \otimes \pi_0 A)_{\Z/2}$. We show that this is a well-defined module structure at the end of the section by using that $\underline{\pi}_0A$ has a natural ``Hermitian structure'' (see \ref{pi0Herm}). We also remark that since the transfer is constant on orbits, the expression $\tran(\mu^{l}_1(c\otimes\res(x)))$ is independent on the representative of $c$.
The right $\mathbb{W}^{\otimes}_2(\pi_0A)$-module structure on $\underline{\pi}_0A$ is defined similarly by the formulas
\[\mu^{r}_1(b\otimes a \otimes a')=w(a') b a\ \ \ \ \ \ \mbox{and}\ \ \ \ \ \mu^r_2(x\otimes (a,c))= w(a) \cdot x+\tran(\mu^{r}_1(\res(x)\otimes c)).\]

\begin{prop}\label{pi_0THR} Let $(A,w)$ be a flat ring spectrum with anti-involution whose underlying orthogonal $\Z/2$-spectrum is connective. Then the Mackey functor $\underline{\pi}_0\THR(A)$ is naturally isomorphic to the box product
\[\underline{\pi}_0(A) \Box_{\mathbb{W}^{\otimes}_2(\pi_0 A)} \underline{\pi}_0(A).\]
\end{prop}

The first step in the proof of Proposition \ref{pi_0THR} is to compute $\pi_0^{\Z/2}(N^{\mathbb{Z}/2}_e A)$.

\begin{lemma} \label{Witt vectors} Let $A$ be a flat connective orthogonal ring spectrum. Then
\begin{itemize}
\item[i)] There is a short exact sequence $0 \to(\pi_0(A) \otimes \pi_0(A))_{\Z/2} \xrightarrow{t} \pi_0^{\Z/2}(N^{\mathbb{Z}/2}_e A) \xrightarrow{\Phi^{\Z/2}} \pi_0(A) \to 0$.
\item[ii)] The diagram
\[\xymatrix{\pi_0(A) \otimes \pi_0(A) \ar[r]^-{\proj} \ar[dr]_-\tran & (\pi_0(A) \otimes \pi_0(A))_{\Z/2} \ar[d]^-t \ar[r]^-{N_\oplus} & (\pi_0(A) \otimes \pi_0(A))^{\Z/2} \\ & \pi_0^{\Z/2}(N^{\mathbb{Z}/2}_e A), \ar[ur]_-\res & }  \]
commutes, where $N_\oplus$ is the additive norm sending $a \otimes b$ to $a \otimes b + b \otimes a$, $\tran$ is the transfer of the Mackey structure of ${\underline{\pi}}_0(N^{\mathbb{Z}/2}_e A)$ and $\res$ becomes the Mackey restriction after composing with the inclusion $(\pi_0(A) \otimes \pi_0(A))^{\Z/2} \hookrightarrow \pi_0(A) \otimes \pi_0(A)$.
\item[iii)] The external norm $N\colon \pi_0(A) \to \pi_0^{\Z/2}(N^{\mathbb{Z}/2}_e A)$ is multiplicative and satisfies the identities
\[\Phi^{\Z/2} \circ N=1\ \ \ \ \ \ \ \ \mbox{and}\ \ \ \ \ \ \ \ N(a+b)=N(a)+N(b)+t(a \otimes b).\]
\end{itemize}
\end{lemma}

\begin{proof} The multiplicative norm $N\colon \pi_0(A) \to \pi_0^{\Z/2}(N^{\mathbb{Z}/2}_e A)$ (see e.g. \cite[\S 9,\S10]{Schwede}) is multiplicative and $\Phi^{\Z/2} \circ N=1$ (first part of (iii)). Because $A$ is connective, the isotropy separation sequence for $\Z/2$ gives an exact sequence
\[\xymatrix{\pi_1^{\Z/2}(N^{\mathbb{Z}/2}_e A) \ar[r]^-{\Phi^{\Z/2}} &  \pi_1(A) \ar[r] & (\pi_0(A) \otimes \pi_0(A))_{\Z/2} \ar[r]^-{t} & \pi_0^{\Z/2}(N^{\mathbb{Z}/2}_e A) \ar[r]^-{\Phi^{\Z/2}} & \pi_0(A) ,} \]
where the last map is surjective. The first map is surjective because $\Phi^{\Z/2} \circ N^{\mathbb{Z}/2}_e \cong 1$ and since geometric fixed points send Euler classes to $1$. Indeed, given a stable map $f \colon \mathbb{S}^1 \to A$, we can apply the derived norm \cite{HHR} functor $N^{\mathbb{Z}/2}_e$ to $f$ and get a map in the homotopy category of $\Z/2$-spectra
\[\xymatrix{S^{\rho} \simeq N^{\mathbb{Z}/2}_e(\mathbb{S}^1) \ar[r]^-{N^{\mathbb{Z}/2}_e(f)} & N^{\mathbb{Z}/2}_e A. }\]
Now there is a $\Z/2$-map $e \colon S^1 \to S^\rho$ which is the suspended Euler class of the sign representation. Then $\Phi^{\Z/2}(N^{\mathbb{Z}/2}_e(f) \circ e)=f$.

Part $ii)$ follows from the construction of the isotropy separation sequence and the double coset formula. The last formula in $iii)$ follows from the external version of \cite[Proposition 10.9 (vi)]{Schwede}. Note that since we only work externally, commutativity of $A$ is not needed.
\end{proof}

\begin{cor} \label{2 inverted} Let $A$ be a flat connective ring spectrum. Suppose that $\frac{1}{2} \in \pi_0(A)$. Then the ring maps $\Phi^{\Z/2}\colon \pi_0^{\Z/2}(N^{\mathbb{Z}/2}_e A) \to  \pi_0(A)$ and $\res \colon  \pi_0^{\Z/2}(N^{\mathbb{Z}/2}_e A) \to  (\pi_0(A) \otimes \pi_0(A))^{\Z/2}$ induce a ring isomorphism
\[\pi_0^{\Z/2}(N^{\mathbb{Z}/2}_e A) \cong  \pi_0(A) \times  (\pi_0(A) \otimes \pi_0(A))^{\Z/2}. \]
\end{cor}

\begin{proof} This is an immediate consequence of part ii) of Lemma \ref{Witt vectors}. We only need to observe that since $\frac{1}{2} \in \pi_0(A)$, the additive norm map $N_\oplus$ is an isomorphism.
 \end{proof}

If $2$ is not invertible in $\pi_0A$, then $\pi_0^{\Z/2}(N^{\mathbb{Z}/2}_e A)$ does not necessarily split. 

\begin{prop}\label{pi_0norm}  Let $A$ be a flat connective orthogonal ring spectrum. Then the Green functor ${\underline{\pi}}_0(N^{\mathbb{Z}/2}_e A)$ is naturally isomorphic to the Green functor $\mathbb{W}^{\otimes}_2(\pi_0A)$.
\end{prop}

\begin{proof}
By Lemma \ref{Witt vectors} any element $x$ of $\pi_0^{\Z/2}(N^{\mathbb{Z}/2}_e A)$ can be uniquely written as
\[x=N(a)+t(c),\]
where $a=\Phi^{\Z/2}x \in \pi_0A$ and  $t(c)=x-N\Phi^{\mathbb{Z}/2}x$. This determines a bijection $\pi_0^{\Z/2}(N^{\mathbb{Z}/2}_e A)\cong W_2(\pi_0A)$ which is additive, since
\[N(a)+t(c)+N(a')+t(c')=N(a+a')+t(c+c'-a \otimes a').\]
To compare the multiplications we use Frobenius reciprocity, and see that
\begin{align*} (N(a)+t(c))(N(a')+t(c'))=N(aa')+t(c \res(\tran(c'))+c\res(N(a'))+\res(N(a))c') \\ =N(aa') + t(cc'+c\tau(c')+c(a'\otimes a') + (a \otimes a)c').\end{align*}
This shows that the bijection above is a ring isomorphism $\pi_0^{\Z/2}(N^{\mathbb{Z}/2}_e A)\cong W_2(\pi_0A)$.

Now we compare the Mackey structures. The underlying ring of the Green functor ${\underline{\pi}}_0(N^{\mathbb{Z}/2}_e A)$ is $\pi_0(A)\otimes \pi_0(A)$, with the involution $\tau$. The pair of ring isomorphisms
\[
W_2(\pi_0A)\stackrel{\cong}{\longrightarrow}\pi_0^{\Z/2}(N^{\mathbb{Z}/2}_e A)\ \ \ \ \ \ \ \ \ \ \ \ \pi_0(A)\otimes \pi_0(A)\stackrel{\id}{\longrightarrow}\pi_0(A)\otimes \pi_0(A)
\]
is clearly compatible with the involution and with the transfers. The compatibility with the restrictions follows from the formula
\[\res(N(a)+t(c))=a \otimes a+ c+\tau(c) \]
for $c\in (\pi_0(A)\otimes \pi_0(A))_{\mathbb{Z}/2}$ and $a\in\pi_0A$.
\end{proof}

Now let $(A,w)$ be a ring spectrum with an anti-involution such that $A$ is flat and connective as a $\Z/2$-spectrum. It remains to show that the $\mathbb{W}^{\otimes}_2(\pi_0 A)$-module structures on $\underline{\pi}_0A$ are well-defined, and that they agree with the ${\underline{\pi}}_0(N^{\mathbb{Z}/2}_e A)$-module structures under the isomorphism of Proposition \ref{pi_0norm}. It is convenient to isolate the structure on $\underline{\pi}_0A$ that is used to define the actions of $\mathbb{W}^{\otimes}_2(\pi_0 A)$.

\begin{defn}[\cite{DO}]\label{defHermMackey} A \emph{
Hermitian Mackey} functor is a $\Z/2$-Mackey functor $M$ with a multiplicative monoid structure on $M(\Z/2)$ which makes it into a ring, and a left action of this monoid on the abelian group $M(\ast)$, denoted by $a \cdot x$, satisfying the following properties:
\begin{itemize}
\item[i)] $w(ab)=w(b)w(a)$ for $a, b \in M(\Z/2)$ and $w(1)$=1, where $w$ is the involution on $M(\Z/2)$. In other words, $M(\Z/2)$ is a ring with the anti-involution $w$;
\item[ii)] $\res(a \cdot x)=a\res(x)w(a)$ for $a \in M(\Z/2)$ and $x \in M(\ast)$;
\item[iii)] $\tran(abw(a))=a \cdot \tran(b)$ for $a, b \in M(\Z/2)$;
\item[iv)] $(a+b) \cdot x= a \cdot x + b \cdot x+ \tran(a \res(x) w(b))$ for $a, b \in M(\Z/2)$ and $x \in M(\ast)$.
 \end{itemize}
\end{defn}

\begin{prop}\label{pi0Herm} Let $(A,w)$ be a flat ring spectrum with anti-involution whose underlying orthogonal $\Z/2$-spectrum is  connective. Then $\underline{\pi}_0(A)$ has a natural Hermitian Mackey functor structure.
\end{prop}

\begin{proof} As we saw above, the $\Z/2$-equivariant left module structure $(N^{\mathbb{Z}/2}_e A) \wedge A \to A$ gives the multiplicative action
\[\pi_0(A) \otimes \pi_0^{\Z/2}(A)\longrightarrow  \pi_0^{\Z/2}(A).\]
defined by $a \cdot x:=N(a)x$, where $a \in \pi_0(A)$ and $x \in \pi_0^{\Z/2}(A)$.
The first axiom is clear since $w$ is an anti-involution on $A$. The second and third axioms
\[\res(N(a) x)=a\res(x)w(a)\ \ \ \ \mbox{and}\ \ \ \ N(a) \tran(b)=\tran(abw(a)) \]
hold because $\underline{\pi}_0(A)$ is a $\underline{\pi_0}(N^{\mathbb{Z}/2}_e A)$-module. Finally, the fourth axiom follows from Lemma \ref{Witt vectors} $iii)$.
\end{proof}

A Hermitian Mackey functor $M$ is a left module over the Green functor $\mathbb{W}^{\otimes}_2(M(\Z/2))$, where $M(\Z/2)$ is just considered as a ring (without an anti-involution). The action maps are
\[\xymatrix@C=14pt@R=7pt{M(\Z/2) \otimes M(\Z/2) \otimes M(\Z/2)\ar[r]^-{\mu^{l}_1}& M(\Z/2)\\
a \otimes a' \otimes b \ar@{|->}[r] & abw(a')
}
\ \ \ \ \
\xymatrix@C=0pt@R=7pt{W^{\otimes}_2(M(\Z/2)) \otimes M(\ast)\ar[r]^-{\mu^{l}_2}&M(\ast)\\
\hspace{-1cm}(a,c) \otimes x\ar@{|->}[r(.25)]& \hspace{-1cm} a \cdot x+ \tran(\mu^{l}_1(c\otimes \res(x))).}
\]

\noindent Axioms $ii)$ and $iii)$ ensure that the action is compatible with the transfers and the restrictions of the Mackey structures, and axiom $iv)$ implies that the action is compatible with the addition in $\mathbb{W}^{\otimes}_2(M(\Z/2))$. We also see that when $M=\underline{\pi}_0A$ this is precisely the $\mathbb{W}^{\otimes}_2(\underline{\pi}_0A)$-module structure defined at the beginning of the section. Using the anti-involution $w$ one can also define a right $\mathbb{W}^{\otimes}_2(M(\Z/2))$-module structure on $M$. This generalizes the right module structure described above for $M=\underline{\pi}_0A$. Hence for any Hermitian Mackey functor $M$ we can define the canonical box product $M \Box_{\mathbb{W}^{\otimes}_2(M(\Z/2))} M$.

\begin{proof}[Proof of Proposition \ref{pi_0THR}]

Frobenius reciprocity implies that
\[(N(a)+t(c))x= N(a)x+\tran(c\res(x)),\]
which is equivalent to the statement that the left $\mathbb{W}^{\otimes}_2(\underline{\pi}_0A)$-module and $\underline{\pi_0}(N^{\mathbb{Z}/2}_e A)$-module  structures on $\underline{\pi_0}A$ are compatible under the isomorphism $\mathbb{W}^{\otimes}_2(\underline{\pi}_0A)\stackrel{\cong}{\to}\underline{\pi_0}(N^{\mathbb{Z}/2}_e A)$ of Proposition \ref{pi_0norm}.

An analogous argument holds for the right actions as well. This can be deduced from the relation $N(a)x=xN(w(a))$, which follows from the following observation: There is a $\Z/2$-equivariant map (in fact an anti-homomorphism) $\Omega \colon  N^{\mathbb{Z}/2}_e R \to N^{\mathbb{Z}/2}_e R$ sending $a\wedge b$ to $\omega(b) \wedge \omega(a)$ such that the diagram
\[\xymatrix{ N^{\mathbb{Z}/2}_e R \wedge R \ar[r] \ar[d]^{\Omega \wedge 1} & R \\ N^{\mathbb{Z}/2}_e R \wedge R \ar[r]^{twist} & R \wedge N^{\mathbb{Z}/2}_e R \ar[u] } \]
commutes.
\end{proof}
We are now ready to complete the proof of the main result of this subsection.
\begin{proof}[Proof of Theorem \ref{pi0relations}] The essential extra structure which allows us to further simplify the box product
\[\underline{\pi}_0(A) \Box_{\mathbb{W}^{\otimes}_2(\pi_0 A)} \underline{\pi}_0(A)\]
is the unit $1 \in \pi_0^{\Z/2}A$, with the property that $\res(1)=1$ in $\pi_0A$. The value of the Box product at the $\Z/2$-set $\Z/2/\Z/2=\ast$ (see e.g. \cite[Proposition 1.5.1]{Bouc}) is
\[ [(\pi_0 A \otimes \pi_0 A) \oplus  (\pi^{\mathbb{Z}/2}_0 A \otimes \pi^{\mathbb{Z}/2}_0 A)]/\mathcal{R},\]
where the subgroup $\mathcal{R}$ is generated by the relations:
\item[i)] $\omega(a) \otimes \omega(b) -a \otimes b$, for $a, b \in \pi_0A$;
\item[ii)] $x \otimes \tran(a)-\res(x) \otimes a$ and $\tran(a) \otimes x-a \otimes \res(x)$, for $a \in \pi_0A$ and $x \in \pi_0^{\Z/2}A$;
\item[iii)] $x \otimes a \cdot y-\omega (a) \cdot x \otimes y$, for $x,y \in \pi^{\mathbb{Z}/2}_0 A$ and $a \in \pi_0A$;
\item[iv)] $x \otimes \tran(a \res(y) w(b))-\tran(w(b)\res(x)a) \otimes y$, for $x,y \in \pi^{\mathbb{Z}/2}_0 A$ and $a, b \in \pi_0A$.
\item[v)] $a \otimes b a'b'-b'ab \otimes a'$, for $a,a',b, b' \in \pi_0A$.

These relations can be substantially simplified using the unit $1 \in \pi^{\mathbb{Z}/2}_0 A$. Indeed, we see that any element $a \otimes b \in \pi_0 A \otimes \pi_0 A$ is getting identified to some element in $\pi^{\mathbb{Z}/2}_0 A \otimes \pi^{\mathbb{Z}/2}_0 A$ via the chain of equivalences
\[a \otimes b \sim a b \otimes 1 = a b \otimes \res(1) \sim \tran(a b) \otimes 1. \]
Hence
$(\underline{\pi}_0(A) \Box_{\mathbb{W}^{\otimes}_2(\pi_0 A)} \underline{\pi}_0(A)) (\ast)$
is in fact a quotient of $\pi^{\mathbb{Z}/2}_0 A \otimes \pi^{\mathbb{Z}/2}_0 A$. We claim that after this identification all the relations follow from iii) and iv). This follows from the equivalences
\[\res(x) \otimes a \sim \tr(\res(x)a) \otimes 1 =\tran(w(1)\res(x)a) \otimes 1 \sim x \otimes \tran(a\res(1)w(1))=x \otimes \tran(a), \]
\[ \tran(a b) \otimes 1 = \tran(w(w(a)) \res(1) b )  \otimes 1 \sim 1 \otimes \tran(b \res(1) w(w(a)))= 1 \otimes \tran(b a).\]
This completes the additive identification of the Box product at $\Z/2/\Z/2=\ast$. The value at $\Z/2$ is simply the quotient of $\pi_0 A \otimes \pi_0 A$ by the relation v). This gives the usual formula for the zero-th Hochschild homology group
\[ \pi_0A/[\pi_0A, \pi_0A]. \]
Finally the structure maps are readily identified using the formulas in \cite[Proposition 1.5.1]{Bouc}.
\end{proof}

We note that the previous argument applies to any Hermitian Mackey functor with a unit, i.e. with an element $1 \in M(\ast)$ such that $\res(1)=1$ in $M(\Z/2)$. That is, the relations in the canonical box product $M \Box_{\mathbb{W}^{\otimes}_2(M(\Z/2))} M$ can be simplified as in the proof of Theorem \ref{pi0relations} in the presence of this extra unit.

\begin{proof}[Proof of Corollary \ref{compi0}] The simplification of the relations follows from the definition of the norm and from Frobenius reciprocity. The claim about the ring structure of $(\pi^{\mathbb{Z}/2}_0 A \otimes \pi^{\mathbb{Z}/2}_0 A)/T$ is clear. The formula for the norm follows from \cite[Proposition 9.1]{Strickland}. \end{proof}

\subsection{Group-rings}\label{grouprings}

Let $M$ be a topological monoid with an anti-involution $\iota\colon M^{op}\to M$. The case we will be most interested in is when $M$ is a topological group and $\iota$ is the inversion map. Given a ring spectrum with anti-involution $(A,w)$ we can form the monoid-ring $A[M]:=A\wedge M_+$. This is a ring spectrum with the multiplication
\[
A[M]\wedge A[M]\cong A\wedge A\wedge (M\times M)_+\longrightarrow A\wedge M_+
\]
where the last map is the smash product of the multiplications of $A$ and $M$. The ring spectrum $A[M]$ acquires an anti-involution
\[
A[M]^{op}\cong A^{op}\wedge M^{op}_+\xrightarrow{w\wedge \iota} A[M].
\]
\begin{example}
Let $(R,w)$ be a discrete ring with anti-involution, and $(M,\iota)$ a discrete monoid with anti-involution. The Eilenberg-MacLane functor $H$ from ${\Z/2}$-abelian groups to orthogonal $\Z/2$-spectra commutes up to homotopy with indexed coproducts, so we get a stable equivalence of ring spectra with anti-involution 
\[
(HR)[M]=\bigvee_{M}HR\stackrel{\simeq}{\longrightarrow}H (\bigoplus_M R ) = H(R[M])
\]
where $R[M]=\oplus_MR$ is the monoid ring, with the anti-involution induced by sending $r\cdot m$ to $w(r) \cdot\iota (m)$.
\end{example}
The following result was originally proved in \cite{Amalie} when $A=\mathbb{S}$, by working directly on the B\"{o}kstedt model.

\begin{prop}\label{assembly}
Let $(A,w)$ be a flat ring spectrum with anti-involution and $(M,\iota)$ a topological monoid with anti-involution which is well-pointed at the identity as a $\Z/2$-space. Then the assembly map
\[\THR(A)\wedge (B_{\times}^{di}M)_+\stackrel{\simeq}{\longrightarrow}\THR(A[M])\]
is a stable equivalence. In particular we recover the equivalence $\THR(\mathbb{S}[M])\simeq \Sigma^\infty(B_{\times}^{di}M)_+$
from \cite{Amalie}.
\end{prop}

\begin{rem}\label{remark:freeloops}
When the monoid $M$ is group-like, the dihedral Bar construction $B_{\times}^{di}M$ is a model for the free loop space $\map(S^{\sigma},B^{\sigma}M)$, where $S^{\sigma}$ is the sign-representation sphere and $B^{\sigma}M$ is the realization of the real simplicial space $NM$ with levelwise involution
\[
M^{\times k}\stackrel{\iota^{\times k}}{\longrightarrow} M^{\times k}\stackrel{\tau_k}{\longrightarrow}M^{\times k}.
\]
Here $\tau_k$ is the permutation of ${\bf k}=\{1,\dots,k\}$ that reverses the order. This is proved as in \cite{Goodwilliecyclic} by comparing the sequences
\[
\xymatrix{
M\ar[d]\ar[r]&B_{\times}^{di}M\ar[d]\ar[r]&B^{\sigma}M\ar@{=}[d]\\
\Omega^\sigma B^{\sigma}M\ar[r]&\map(S^{\sigma},B^{\sigma}M)\ar[r]_-{\ev}&B^{\sigma}M
}
\]
where the middle vertical map is induced by the $S^{1}$-action on $B_{\times}^{di}M$. The lower sequence is a fiber sequence. When $M$ is group-like the upper one is also a fiber sequence and the map on fibers is a $\Z/2$-equivalence by \cite[\S 6.2]{thesis} and \cite{Stiennon}.
\end{rem}

\begin{proof}[Proof of \ref{assembly}]
The assembly map of the statement is defined as the geometric realization of the map of real simplicial spectra defined in simplicial degree $k$ by the assembly map of the loop space and the shuffle of the smash summands
\[
\xymatrix{
\displaystyle(\hocolim_{I^{\times 1+{\bf k}}}\Omega^{\underline{i}}(\Sigma^{\infty}A_{i_0}\wedge\dots\wedge A_{i_k}))\wedge (M_+)^{\wedge 1+{\bf k}}\ar[r]\ar@{-->}[dr]
&
\displaystyle\hocolim_{I^{\times 1+{\bf k}}}\Omega^{\underline{i}}(\Sigma^{\infty}A_{i_0}\wedge\dots\wedge A_{i_k}\wedge (M_+)^{\wedge 1+{\bf k}})\ar[d]
\\
&\displaystyle\hocolim_{I^{\times 1+{\bf k}}}\Omega^{\underline{i}}(\Sigma^{\infty}A_{i_0}\wedge M_+\wedge\dots\wedge A_{i_k}\wedge M_+)\rlap{\ .}
}
\]
A similar assembly map can be constructed for the middle spectrum in the zig-zag of the comparison Theorem \ref{comparison}, and is easy to verify that these maps are compatible with the assembly of the cyclic nerve, defined by the shuffle isomorphism
\[
(N^{di}_\wedge A)\wedge (N^{di}_\times M)_+\cong (N^{di}_\wedge A)\wedge N^{di}_\wedge(M_+)\stackrel{\cong}{\longrightarrow} N^{di}_\wedge (A\wedge M_+)=N^{di}_\wedge A[M].\]
Since the assembly for the dihedral nerve is an isomorphism, the claim follows.
When $A=\mathbb{S}$ is the sphere spectrum this gives an equivalence of real simplicial spectra
\[
\Sigma^\infty  (N^{di}_\times M)_+\cong  (N^{di}_\wedge \mathbb{S})\wedge (N^{di}_\times M)_+\cong N^{di}_\wedge \mathbb{S}[M]\simeq \THR(\mathbb{S}[M]).
\]
\end{proof}

From now on we will denote $B_{\times}^{di}M$ by $B^{di}M$ keeping in mind that the dihedral bar construction is taken with respect to the cartesian product.

\;

\;

We will now analyze the Mackey functor $\underline{\pi}_0 \THR(A[M])$, the fundamental case is when $A = \mathbb{S}$. Assume from now on that $M$ is cofibrant as a $\Z/2$-space. The space of fixed points of $M$ under the $\mathbb{Z}/2$-action $m \mapsto \iota (m) = \bar{m}$ has a left $M$-action given by $m \cdot n = mn\bar{m}$ and a right action given by $n \cdot m =  \bar{m}nm$.
There is a natural homeomorphism $(B^{di}M)^{\mathbb{Z}/2} \cong B(M^{\Z/2},M,M^{\Z/2})$, obtained by subdividing the dihedral bar construction on $M$ and taking $\mathbb{Z}/2$-fixed points levelwise before geometric realization. We write $(\pi_0 M)_{conj}$ for $\pi_0 B^{di}M$. When $\pi_0 M$ is a group, this is just the set of conjugacy classes of elements of $\pi_0 M$. Proposition \ref{assembly} and the tom-Dieck splitting give a sequence of stable equivalences
\[ \THR(\mathbb{S}[M])^{\mathbb{Z}/2} \simeq (\Sigma^\infty B^{di}M_+)^{\mathbb{Z}/2} \simeq \Sigma^\infty ((B^{di}M)_{h\mathbb{Z}/2})_+ \vee \Sigma^\infty B(M^{\Z/2},M,M^{\Z/2})_+.\]
It follows that the Mackey functor $\underline{\pi}_0\THR(\mathbb{S}[M])$ can be presented as
\[\underline{\pi}_0(\Sigma^\infty B^{di}M_+)\cong \big(\!\!\xymatrix@C=40pt{
\mathbb{Z}[(\pi_0M)_{conj}]\ar@(dl,dr)_-{\pi_0(\iota)}\ar@<.5ex>[r]^-{\tran}&\Z[(\pi_0M)_{conj}]_{\mathbb{Z}/2} \oplus\Z [\pi_0 (M^{\Z/2}) \times_{\pi_0 M} \pi_0{(M^{\Z/2})}] \ar@<.5ex>[l]^-{\res}
}\!\!\big),\]
where $\Z[-]$ denotes the free abelian group functor. The involution is well defined on conjugacy classes and gives the action $\pi_0(\iota)$ of $\Z/2$ on $\Z[(\pi_0 M)_{conj}]$. The transfer is just the projection map to the coinvariants, and the restriction map is given by the additive norm on $\Z[(\pi_0M)_{conj}]_{\Z/2}$ and by $[m,m'] \mapsto [m m']$ on the basis elements $[m,m']$ of the free abelian group $\Z [\pi_0 M^{\Z/2} \times_{\pi_0 M} \pi_0 M^{\Z/2}]$.

\begin{cor}\label{cor:monring}
Let $(A,w)$ be a flat ring spectrum with anti-involution whose underlying orthogonal $\Z/2$-spectrum is connective and let $(M,\iota)$ be a topological monoid with anti-involution which is cofibrant as a $\Z/2$-space. Then there is an isomorphism of Mackey functors
\[
\underline{\pi}_0\THR(A[M])\cong \underline{\pi}_0\THR(A)\Box \underline{\pi}_0(\Sigma^\infty B^{di}M_+),
\]
where $\underline{\pi}_0\THR(A)$ is calculated in Theorem \ref{pi0relations} and $\underline{\pi}_0(\Sigma^\infty B^{di}M_+)$ is calculated above.
\end{cor}

\begin{proof} This follows immediately from Proposition \ref{assembly} and the box-product spectral sequence for a smash product of $\Z/2$-spectra.
\end{proof}

\begin{cor}\label{cor:grpring}
Let $G$ be a discrete group with the inversion involution. Then there is a natural isomorphism
\[
\underline{\pi}_0\THR(\Z[G])\cong \big(\xymatrix@C=40pt{
\mathbb{Z}[G_{conj}]\ar@(dl,dr)_-{(-)^{-1}}\ar@<.5ex>[r]^-{\tran}& (\Z[G_{conj}]_{\Z/2} \oplus \Z[G^{\Z/2} \times_G G^{\Z/2}])/D \ar@<.5ex>[l]^-{\res}
}\big),\]
where $D$ is the subgroup generated by the elements $2[g,g'] - [gg']$, for all $[g,g'] \in G^{\Z/2} \times_G G^{\Z/2}$. The Mackey structure maps are the same as in Corollary \ref{cor:monring}.
\end{cor}

\begin{proof} The identification of the groups follows from Corollary \ref{cor:monring} and the explicit formula for the box-product \cite[1.5.1]{Bouc}. To see that the restriction map vanishes on $D$, note that for $[g,g'] \in G^{\Z/2} \times_G G^{\Z/2} $ we have $res([g \cdot g']) = [gg'] + [(gg')^{-1}] = 2[gg']$ and $res([g,g']) = [g \cdot g']$.

\end{proof}

We end this section with two instructive examples:

\begin{example}\label{ZZ} Let $G =\Z= \langle t \rangle$ be the cyclic group of infinite order with the inversion involution. We write $\mathbb{S}[t,t^{-1}]:=\mathbb{S}[\Z]$ for the corresponding spherical group-ring, where the involution sends $t$ to $t^{-1}$. By Remark \ref{remark:freeloops} there is a weak equivalence $B^{di}\Z \simeq \map(S^{\sigma},B^{\sigma}\Z)$. There are weak equivalences $B^\sigma \Z \simeq S^1$, where $S^1$ is the usual circle with the trivial $\Z/2$-action, and $\map(S^\sigma, S^1) \simeq \Z \times S^1$, where $\Z$ has the inversion action. We get a weak equivalence of $\Z/2$-spectra $\THR(\mathbb{S}[t,t^{-1}]) \simeq \Sigma^\infty (\Z \times S^1)_+$.
There is an isomorphism of Mackey functors
\[\underline{\pi}_0(\THR(\mathbb{S}[t,t^{-1}])) \cong \big(\xymatrix@C=40pt{
\mathbb{Z}[t,t^{-1}]\ar@(dl,dr)_-{\Z[t^n \mapsto t^{-n}]}\ar@<.5ex>[r]^-{\tran}&\Z[t,t^{-1}]_{\mathbb{Z}/2} \oplus \Z \ar@<.5ex>[l]^-{\res}
}\big),\]
where the transfer maps trivially into the second summand. With integral coefficients, we get the formula
\[ \underline{\pi}_0(\THR(\Z[t,t^{-1}])) \cong \big(\xymatrix@C=40pt{
\mathbb{Z}[t,t^{-1}]\ar@(dl,dr)_-{\Z[t^n \mapsto t^{-n}]}\ar@<.5ex>[r]^-{\tran}& t\Z[t] \oplus \Z \ar@<.5ex>[l]^-{\res}
}\big),\]
where the transfer sends $t^n$ and $t^{-n}$ to $t^n$ in the first summand for $n >0$, and it sends $1$ to $2$ in the second summand.

\end{example}

\begin{example}\label{ZZ2}
Let $G = C_2$, the cyclic group of order 2. Note that every involution on this group is trivial. There is a weak equivalence
\[B^{di}C_2 \simeq  C_2 \times B^\sigma C_2, \]
since $C_2$ is abelian and has the trivial involution. By Corollary \ref{cor:monring} we get a stable equivalence of $\Z/2$-spectra $\THR(\mathbb{S}[C_2]) \simeq \Sigma^\infty (C_2 \times B^\sigma C_2)_+$.
With sphere coefficients we get
\[\underline{\pi}_0(\THR(\mathbb{S}[C_2])) \cong \big(\xymatrix@C=40pt{
\mathbb{Z}[C_2]\ar@(dl,dr)_-{ id}\ar@<.5ex>[r]^-{\tran}&\Z[C_2] \oplus \Z[C_2 \times C_2] \ar@<.5ex>[l]^-{\res}
}\big), \]
and with integer coefficients
\[\underline{\pi}_0(\THR(\mathbb{Z}[C_2])) \cong \big(\xymatrix@C=40pt{
\mathbb{Z}[C_2]\ar@(dl,dr)_-{ id}\ar@<.5ex>[r]^-{\tran}&\Z[C_2] \oplus (\Z/2)^2 \ar@<.5ex>[l]^-{\res}
}\big), \]
where the transfer is the multiplication by 2 map into the first summand and the 0 map to the second summand, and $\res$ is projection on the first summand.
\end{example}

\subsection{The homotopy type of $\THR(\mathbb{F}_p)$}\label{secFp}

Let $p$ be a prime number. Since $H \F_p$ is a commutative ring spectrum, $\THH(\F_p)$ is a ring spectrum. B{\"o}kstedt \cite{Bok2} and Breen \cite{breen} showed that there is an isomorphism of graded rings $\pi_\ast \THH(\F_p) \cong \F_p[x]$, where $x$ is a generator in $\pi_2 \THH(\F_p) \cong \F_p$. The purpose of this section is to carry out an analogous calculation for $\THR(\F_p):=\THR(H\F_p)$, where $H\F_p$ is the Eilenberg MacLane ring spectrum of $\F_p$ with the trivial anti-involution.
Our calculations will rely on the knowledge of the results of B{\"o}kstedt and Breen.

We recall that the homotopy groups of a $\mathbb{Z}/2$-equivariant ring spectrum $A$ naturally form a $\Z \times \Z$-graded ring, whose $(n,k)$-graded piece is the group of equivariant stable homotopy classes of maps
\[
\pi_{n,k}A:=[S^{n,k},A]^{\Z/2},
\]
where $S^{n,k}=S^{n-k}\wedge S^{k\sigma}$ and $\sigma$ denotes the sign representation of $\Z/2$. We write $\Sigma^{n,k}$ for the corresponding suspension functor. We let
\[
T_{H\F_p}(S^{2,1}) = \bigvee^\infty_{n=0} \Sigma^{2n,n} H\F_p
\]
denote the free associative $H\F_p$-algebra on $S^{2,1}$.

\begin{theorem}\label{theorem:thrfp} There is a stable equivalence of $\Z/2$-equivariant ring spectra
\[T_{H\F_p}(S^{2,1})\stackrel{\simeq}{\longrightarrow} \THR(\F_p).\]
In particular, there is an isomorphism of bigraded rings ${H\F_p}_{\ast,\ast}[\tilde{x}]\cong \pi_{\ast,\ast}\THR(\F_p)$,
where $\tilde{x}$ has bidegree $(2,1)$.
\end{theorem}

Before giving a proof of this theorem we deduce the associative ring structures on the fixed point spectrum of $\THR(\F_p)$.

\begin{cor} \label{poddTHRFp} For $p$ odd there is an isomorphism of graded rings \[\pi_\ast (\THR(\F_p)^{\Z/2}) \cong \F_p [y], \]
where $y$ is a generator in degree $4$ which maps to $x^2$ in $\F_p[x] \cong \pi_\ast \THH(\F_p)$ under the restriction map.
\end{cor}
\begin{proof}Since $p$ is odd there is an isomorphism $(H\F_p)_{\ast,\ast} \cong \F_p[u^{\pm 1}]$ with $u$ in bidegree $(0,2)$ (see \cite[page 338]{HuKriz}). Therefore Theorem \ref{theorem:thrfp} gives an isomorphism $\pi_{\ast,0} \THR(\F_p) \cong \F_p[y]$, where $y = u^{-1}\cdot \tilde{x}^2$.
\end{proof}

\begin{cor}There is an isomorphism of graded rings
\[\pi_\ast (\THR(\F_2)^{\Z/2}) \cong \F_2[\bar{x},y],\]
where $y$ is of degree $1$ and $\bar{x}$ is of degree $2$. Under the restriction map
the element $\bar{x}$ maps to the multiplicative generator $x$ of $\F_2[x] \cong \pi_\ast \THH(\F_2)$ and $y$ maps to $0$.
\end{cor}
\begin{proof}The bi-negative part of $(H\F_2)_{\ast,\ast}$ is a polynomial ring on the Euler class $a$, which has bidegree $(-1,-1)$, and a generator $u$ of bidegree $(0,-1)$ (see \cite[Proposition 6.2]{HuKriz}). From this and Theorem \ref{theorem:thrfp} we see that
\[\pi_{\ast,0} \THR(\F_2) \cong \F_2[\bar{x},y],\]
where $\bar{x} = \tilde{x} u$ and $y  = \tilde{x} a$.

From the proof of \ref{theorem:thrfp} we will see that the map $\pi_{2,0} \THR(\F_2) \to \pi_2 \THH(\F_2)$ sends $\bar{x} = \tilde{x} u$ to $x$. The element $y$ lands in the group $\pi_1\THH(\F_p) $ which is trivial.
\end{proof}

The strategy of the proof of Theorem \ref{theorem:thrfp} is to construct a map $T_{H\F_p}(S^{2,1})\to \THR(\F_p)$ exploiting the fact that  $T_{H\F_p}(S^{2,1})$ is a free $H\F_p$-algebra, and then show that this map induces an equivalence on geometric fixed points by means of Theorem \ref{geofixcalc}. The following key lemma will allow us to construct the map.

\begin{lemma}\label{keylemma}\begin{enumerate}
\item \label{lemma:Fp-i} $\pi_{1,1} \THR(\F_p) = 0$.
\item \label{lemma:Fp-ii} The restriction map $\pi_{2,1}\THR(\mathbb{F}_p) \to \pi_2 \THH(\mathbb{F}_p)$
is surjective.
\end{enumerate}
\end{lemma}

\begin{proof} To see that part $\ref{lemma:Fp-i}$ holds, consider the exact sequence of homotopy groups
\[\pi_1\THH(\F_p) \to \pi_{1,1}\THR(\F_p) \to \pi_{0,0}\THR(\F_p) \to \pi_0\THH(\F_p).\]
The left hand group is trivial and the right hand map is an isomorphism, by Theorem \ref{pi0relations} and Corollary \ref{compi0}. Hence the claim follows.
To prove part $\ref{lemma:Fp-ii}$ we note that the inclusion of the $0$-skeleton $H\F_p \to \THR(\F_p)$, along with the circle action on the target induces a map of $\Z/2$-spectra $S^{1,1}_+ \wedge H\F_p \to \THR(\F_p)$. The left hand side splits stably as $S^{1,1}_+ \wedge H\F_p \simeq H\F_p \vee \Sigma^{1,1}H\F_p$ and we are interested in the map out of the suspended part.
Moding out $p$ gives a map
\[ (\Sigma^{1,1}H\F_p)/p \to (\THR(\F_p))/p,\]
where the source again splits as $(\Sigma^{1,1}H\F_p)/p \simeq \Sigma^{1,1}H\F_p \vee \Sigma^{2,1}H\F_p$. We consider the resulting map $\alpha \colon \Sigma^{2,1} H\F_p \to \THR(\F_p)/p$. On homotopy groups there is an induced diagram
\[\xymatrix{\pi_{2,1} \Sigma^{2,1} H\F_p \ar[d] \ar[r]^-{\alpha_\ast} & \pi_{2,1} \THR(\F_p)/p \ar[d] \\
\pi_{2} \Sigma^{2} H\F_p  \ar[r]^-{\alpha_\ast} & \pi_{2} \THH(\F_p)/p
 \rlap{\ .}}\]
The left hand vertical map is an isomorphism since it is equivalent to the restriction of the constant Mackey functor $\mathbb{F}_p$. The lower horizontal map is an isomorphism by B{\"o}kstedt's calculation \cite{Bok2} (see also \cite[Section 4.2]{Wittvect}), and hence the right hand vertical map is surjective. The natural map $\THR(\F_p) \to \THR(\F_p)/p $ induces a diagram of homotopy groups
 \[\xymatrix{\pi_{2,1} \THR(\F_p) \ar[d] \ar[r] & \pi_{2,1} \THR(\F_p)/p \ar[d] \\
\pi_{2} \THH(\F_p)  \ar[r] & \pi_{2} \THH(\F_p)/p
 \rlap{\ ,}}\]
where the lower horizontal map is an isomorphism. We have just seen that the right hand vertical map is surjective. The top horizontal map is surjective, since its cokernel injects into $\pi_{2,1}\Sigma^{1,0}\THR(\F_p) \cong \pi_{1,1} \THR(\F_p)$, which is zero by part $\ref{lemma:Fp-i}$. The result follows.
\end{proof}

By the previous lemma we can choose a lift $\tilde{x} \in \pi_{2,1}\THR(\F_p)$ of the generator $x \in \pi_2\THH(\F_p)$. Since $\THR(\F_p)$ is an $H\F_p$-algebra by \ref{multTHR}, the map $\tilde{x} \colon S^{2,1} \to \THR(\F_p)$ induces an algebra map from the free associative $H\F_p$-algebra
\[T_{\tilde{x}} \colon T_{H\F_p}(S^{2,1}) \to \THR(\F_p).\]
In order to show that this map induces an equivalence on geometric fixed points, we start by computing the homotopy groups of the geometric fixed points of $\THR(\F_p)$ abstractly.

\begin{prop}\label{geomTHRFp} For $p$ an odd prime, the spectrum $\Phi^{\Z/2} \THR(\F_p)$ is contractible. For $p=2$ there is an isomorphism of graded rings
\[ \pi_\ast \Phi^{\Z/2} \THR(\F_2) \cong \F_2[w_1,w_2], \]
where $|w_1| = |w_2| = 1$. 
\end{prop}
\begin{proof}
By Theorem \ref{geofixcalc} and Corollary \ref{cor:fixed-pt-alg-str} there is a natural equivalence of associative ring spectra
\[\Phi^{\Z/2} \THR(\F_p) \simeq \Phi^{\Z/2} H\F_p \wedge_{H\F_p}^{\mathbf{L}} \Phi^{\Z/2} H\F_p\]
for every prime, where $H\F_p$ is a flat model for the Eilenberg-MacLane spectrum of $\F_p$ as a commutative $\Z/2$-orthogonal ring spectrum (see \ref{flatreplcomm}).
 If $p$ is odd  $\Phi^{\Z/2} H\F_p$ is contractible and thus so is $\Phi^{\Z/2} \THR(\F_p)$. For $p=2$
the K{\"u}nneth formula gives an isomorphism \[\pi_\ast (\Phi^{\Z/2} H\F_2 \wedge_{H\F_2}^{\mathbf{L}} \Phi^{\Z/2} H\F_2) \cong (\pi_\ast \Phi^{\Z/2} H\F_2) \otimes_{\F_2} (\pi_\ast \Phi^{\Z/2} H\F_2)\] and it follows from \cite{HuKriz} (see also \cite{Wilson}) that $\Phi^{\Z/2} H\F_2 \cong \F_2[w]$, where $|w|=1$. The result follows.
\end{proof}

\begin{proof}[Proof of \ref{theorem:thrfp}] The map induces a stable equivalence of underlying non-equivariant spectra by B{\"o}kstedt's calculations. It therefore suffices to show that it also induces a stable equivalence on $\Z/2$-geometric fixed points.

When $p$ is odd this holds because the geometric fixed points of both source and target are contractible. For $p = 2$ by identifying $\Phi^{\Z/2}(T_{H\F_2}(S^{2,1})) \simeq T_{\Phi^{\Z/2}(H\F_2)}(S^{1})$, we get an isomorphism
\[\pi_\ast \Phi^{\Z/2}(T_{H\F_2}(S^{2,1})) \cong \pi_\ast(\Phi^{\Z/2}{H\F_2})[v] \cong \F_2[v,w].\]


To see that the map $\pi_\ast \Phi^{\Z/2}(T_{H\F_2}(S^{2,1})) \to \pi_\ast \Phi^{\Z/2} \THR(\F_2)$ is an isomorphism, it suffices to show that it is an isomorphism in degree $1$, since the source and the target are both polynomial rings over $\F_2$ on two generators in that degree. By considering the long exact sequences of homotopy groups associated to the isotropy separation sequence we see that it suffices to prove that the map $\pi_{1,0} T_{H\F_2}(S^{2,1}) \to \pi_{1,0} \THR(\F_2)$ is an isomorphism. Since the canonical map $\Sigma^{2,1}H\F_2 \to T_{H\F_2}(S^{2,1})$ induces an isomorphism on $\pi_{1,0}$ we are left with showing that the map $\tilde{x} \colon \Sigma^{2,1}H\F_2 \to \THR(\F_2)$ induces an isomorphism on $\pi_{1,0}$. There is a commutative diagram with exact rows
\[\xymatrix@C=15pt@R=15pt{\pi_{2} \Sigma^{2}H\F_2 \ar[d]^-{\cong} \ar[r]^-{\tran=0} & \pi_{2,1} \Sigma^{2,1}H\F_2 \ar@{>->}[d] \ar[r]^{\cong} & \pi_{1,0} \Sigma^{2,1}H\F_2 \ar[d] \ar[r] & \pi_1 H\F_2=0 \ar@{=}[d] \\
\pi_{2} \THH(\F_2)  \ar[r]^-0 & \pi_{2,1} \THR(\F_2)  \ar[r]_-{\cong}^-f & \pi_{1,0} \THR(\F_2)\ar[r]&  \pi_1 \THH(\F_2) = 0. } \]
The upper middle horizontal map is an isomorphism since the transfer map $\tran$ is equal to $2$ and hence to $0$, and the map $f$ is surjective, since $\pi_1 \THH(\F_2) = 0$. The left hand vertical map is an isomorphism, hence the lower left hand horizontal map is zero. It follows that $f$ is an isomorphism. From the commutative square
\[
\xymatrix@R=15pt@C=40pt{
\pi_{2,1} \Sigma^{2,1}H\F_2\ar[r]^{\cong}\ar[d]^{\tilde{x}}&\pi_{2} \Sigma^{2}H\F_2\ar[d]^{x}_{\cong}\\
\pi_{2,1} \THR(\F_2)\ar[r]&\pi_{2} \THH(\F_2)
}
\]
we see that the map $\tilde{x}$ is injective in $\pi_{2,1}$. Therefore the map $\pi_{1,0} \Sigma^{2,1}H\F_2 \to \pi_{1,0} \THR(\F_2)$ is an injective map with the source isomorphic to $\F_2$. Thus it remains to show that $\pi_{1,0} \THR(\F_2)$ is isomorphic to $\F_2$.

We argue by considering the isotropy separation sequence of $\THR(\F_2)$ and the equivariant map $H\F_2 \to \THR(\F_2)$. In low degrees this is the sequence
\[
\xymatrix@C=10pt@R=5pt{
\pi_1\THR(\F_2)_{h\Z/2}\ar[r]
&
\pi_{1,0}\THR(\F_2)\ar[r]
&
\pi_{1}\Phi^{\Z/2}\THR(\F_2)\ar[r]
&
\pi_0\THR(\F_2)_{h\Z/2}\ar[r]
&
\pi_{0,0}\THR(\F_2)
\\
\F_2\ar@{=}[u]\ar[r]
&
?\ar@{=}[u]\ar[r]
&
\F_2\oplus\F_2\ar@{=}[u]\ar[r]
&
\F_2\ar@{=}[u]\ar[r]_-0
&
\F_2\ar@{=}[u]
}
\]
The description of the first and fourth term in this sequence follows from the homotopy orbit spectral sequence. The third group is identified using Proposition \ref{geomTHRFp} and the last term is computed using Theorem \ref{pi0relations}. We claim that the first map in this sequence is zero. Indeed, the map $H\F_2 \to \THR(\F_2)$ induces a commutative diagram
\[\xymatrix{ \pi_1((H\F_2)_{h\Z/2}) \ar[r] \ar[d]^\cong & \pi_{1,0} H \F_2=0 \ar[d] \\ \pi_1\THR(\F_2)_{h\Z/2} \ar[r]
& \pi_{1,0}\THR(\F_2). }\]
The left hand vertical map is an isomorphism by the homotopy orbit spectral sequence. Thus the lower horizontal map is zero, and we get a short exact sequence
\[
\xymatrix{
0\ar[r]
&
\pi_{1,0}\THR(\F_2)\ar@{>->}[r]
&
\F_2\oplus\F_2\ar@{->>}[r]
&
\F_2\ar[r]
&0.
}
\]
This shows that $\pi_{1,0}\THR(\F_2)$ is isomorphic to $\F_2$, which concludes the proof.
\end{proof}

\subsection{Towards $\THR(\mathbb{Z})$}\label{secTHRZ}

This section describes parts of the homotopy type and multiplicative structure of $\THR(\Z)$. We begin by computing its $\Z/2$-geometric fixed points and go on to describe the $\Z/2$-equivariant homotopy type of $\THR(\Z)[\tfrac{1}{2}]$.

\begin{theorem}\label{thm:phiTHR(Z)} There is an isomorphism of graded rings
\[ \pi_\ast \Phi^{\Z/2} \THR(\Z)\cong \F_2[b_1,b_2,e]/e^2,\]
where $|b_1|=|b_2|= 2$ and $|e|=1$.
\end{theorem}

\begin{proof}
The formula from Theorem \ref{geofixcalc} gives an equivalence of orthogonal ring spectra
\[\Phi^{\Z/2} \THR(\Z) \simeq \Phi^{\Z/2}H\Z \wedge_{H\Z}^{\mathbf{L}} \Phi^{\Z/2} H \Z.\]
We compute the homotopy ring $\pi_\ast(\Phi^{\Z/2}H\Z \wedge_{H\Z}^{\mathbf{L}} \Phi^{\Z/2} H \Z)$ by means of the multiplicative K\" unneth spectral sequence (see for example \cite{Tilson})
\[E^2_{p,q}=\Torr_{p,q}^{\mathbb{Z}}(\pi_\ast(\Phi^{\Z/2}H\Z), \pi_\ast(\Phi^{\Z/2}H\Z)) \Rightarrow \pi_{p+q}(\Phi^{\Z/2}H\Z \wedge_{H\Z}^{\mathbf{L}} \Phi^{\Z/2} H \Z). \]
This spectral sequence collapses at the $E^2$-term since the ring $\Z$ has global dimension $1$, and the $E^2$-page has only two potentially non-trivial columns $p=0$ and $p=1$. By the computations of \cite{HuKriz} the graded ring $\pi_\ast(\Phi^{\Z/2}H\Z)$ is a polynomial algebra over $\F_2$ on one generator $b$ of degree $2$. We compute the bigraded ring
\[E^2_{p,q}\cong\Torr_{p,q}^{\mathbb{Z}}(\F_2[b], \F_2[b])\]
by choosing the bigraded multiplicative resolution
\[(\Z[\overline{b},\varepsilon]/(\varepsilon^2),d(\varepsilon)=2) \longrightarrow \F_2[b],\]
where the bidegrees of the generators are $|\overline{b}|=(0,2)$ and $|\varepsilon|=(1,0)$, and where $\overline{b}$ maps to $b$ and $\varepsilon$ to zero. We then see that the $E^2$-term is isomorphic to
\[E^2_{p,q}\cong\F_2[\overline{b_1}, \overline{b_2}, \varepsilon]/(\varepsilon^2)\]
as a bigraded algebra,
where $\vert \overline{b_1} \vert = \vert \overline{b_2} \vert=(0,2)$ and $\vert \varepsilon \vert =(1,0)$. The only non-trivial terms are thus concentrated at the entries $(0, 2k)$ and $(1,2k)$, for $k \geq 0$, and there are no additive extensions. Let $b_i$ denote the element in $\pi_2(\Phi^{\Z/2}H\Z \wedge_{H\Z}^{\mathbf{L}} \Phi^{\Z/2} H \Z)$ which is detected by $\overline{b_i}$ and $e$ denote the element in $\pi_1(\Phi^{\Z/2}H\Z \wedge_{H\Z}^{\mathbf{L}} \Phi^{\Z/2} H \Z)$ which is detected by $\varepsilon$. The multiplicativity of the spectral sequence implies that the polynomial algebra $\F_2[b_1, b_2]$ is a subalgebra of $\pi_\ast(\Phi^{\Z/2}H\Z \wedge_{H\Z}^{\mathbf{L}} \Phi^{\Z/2} H \Z)$ (this also follows from the K\"unneth Theorem). The multiplicativity of the filtration tells us that the composition 
\[ \F_2[b_1, b_2] \subset\pi_\ast(\Phi^{\Z/2}H\Z \wedge_{H\Z}^{\mathbf{L}} \Phi^{\Z/2} H \Z) \xrightarrow{e \cdot (-)} \pi_\ast(\Phi^{\Z/2}H\Z \wedge_{H\Z}^{\mathbf{L}} \Phi^{\Z/2} H \Z)\]
is injective and that
\[\pi_\ast(\Phi^{\Z/2}H\Z \wedge_{H\Z}^{\mathbf{L}} \Phi^{\Z/2} H \Z) = \F_2[b_1, b_2] \oplus e\F_2[b_1, b_2].\]
It remains to prove that $e^2=0$. This is the only multiplicative extension problem of the K\"unneth spectral sequence, and it needs an extra input. 

We solve the multiplicative extension problem by exploiting our knowledge of $\Phi^{\Z/2} \THR(\F_2)$. The mod $2$ reduction map $\red \colon H\Z \to H\F_2$ induces a multiplicative map of orthogonal ring spectra
\[(\Phi^{\Z/2} \red)\wedge_{\red} (\Phi^{\Z/2} \red)\colon \Phi^{\Z/2}H\Z \wedge_{H\Z}^{\mathbf{L}} \Phi^{\Z/2} H \Z \longrightarrow \Phi^{\Z/2}H\F_2 \wedge_{H\F_2}^{\mathbf{L}} \Phi^{\Z/2} H \F_2. \]
On homotopy groups the morphism $\Phi^{\Z/2} \red \colon \Phi^{\Z/2}H\Z \to \Phi^{\Z/2}H\F_2$ induces the homomorphism of graded rings $\F_2[b] \to \F_2[w]$
which sends $b$ to $w^2$. It follows that the map
\[\pi_\ast(\Phi^{\Z/2}H\Z \wedge_{H\Z}^{\mathbf{L}} \Phi^{\Z/2} H \Z) \longrightarrow \pi_\ast(\Phi^{\Z/2}H\F_2 \wedge_{H\F_2}^{\mathbf{L}} \Phi^{\Z/2} H \F_2)\cong \F_2[w_1,w_2]\]
sends $b_1$ to $w_1^2$ and $b_2$ to $w_2^2$ (see Proposition \ref{geomTHRFp} for the computation of the geometric fixed points of $\THR(\F_2)$). If we can show that the map
\[\pi_1(\Phi^{\Z/2}H\Z \wedge_{H\Z}^{\mathbf{L}} \Phi^{\Z/2} H \Z) \longrightarrow \pi_1(\Phi^{\Z/2}H\F_2 \wedge_{H\F_2}^{\mathbf{L}} \Phi^{\Z/2} H \F_2)\]
is zero, we would have that $e^2=0$. Indeed, if $e^2$ is nonzero it must be either $b_1$, $b_2$ or $b_1+b_2$. But all these elements map to nontrivial elements in $\pi_\ast(\Phi^{\Z/2}H\F_2 \wedge_{H\F_2}^{\mathbf{L}} \Phi^{\Z/2} H \F_2)$ which will contradict to the fact that $e$ maps to zero.

In order to show the vanishing statement for the induced homomorphism on $\pi_1$, we need to analyze the map 
\[\Phi^{\Z/2} \red \colon \Phi^{\Z/2}H\Z \longrightarrow \Phi^{\Z/2}H\F_2.\]
The multiplication maps define maps of $\Z/2$-equivariant commutative orthogonal ring spectra $f_{\Z} \colon N^{\Z/2}_e H \Z \to H\Z$ and $f_{\F_2} \colon N^{\Z/2}_e H \F_2 \to H\F_2$ which we call the ``Frobenius maps'', following \cite{NikSchol}.
These Frobenius maps induce maps of algebras on geometric fixed points
\[\Phi^{\Z/2}f_{\Z} \colon H \Z \to \Phi^{\Z/2} H\Z \ \ \ \ \ \ \mbox{and}\ \ \ \ \ \ \ \Phi^{\Z/2}f_{\F_2} \colon H \F_2 \to \Phi^{\Z/2} H\F_2,\]
which induce $H\Z$ and $H\F_2$-module structures on $\Phi^{\Z/2} H\Z$ and $\Phi^{\Z/2} H\F_2$, respectively. These are the module structures of Theorem \ref{geofixcalc}, and we refer to them as the Frobenius module structures. Since any $H \F_2$-module splits uniquely up to homotopy, the Frobenius module structure of $\Phi^{\Z/2} H\F_2$ determines a unique  isomorphism $\theta \colon \Phi^{\Z/2} H\F_2 \cong\bigvee_{n \geq 0} \Sigma^{n} H\F_2$ in the homotopy category of $H\F_2$-modules. It follows from the universal coefficient theorem that any $H \Z$-module with homotopy groups concentrated in even degrees splits uniquely, up to automorphisms of the summands. Thus the Frobenius module structure on $\Phi^{\Z/2} H\Z$ determines a unique isomorphism $\zeta \colon \Phi^{\Z/2} H\Z \cong \bigvee_{n \geq 0} \Sigma^{2n} H\F_2$ in the homotopy category of $H\Z$-modules.
Under these decompositions, the map
\[\pi_1(\Phi^{\Z/2}H\Z \wedge_{H\Z}^{\mathbf{L}} \Phi^{\Z/2} H \Z) \longrightarrow \pi_1(\Phi^{\Z/2}H\F_2 \wedge_{H\F_2}^{\mathbf{L}} \Phi^{\Z/2} H \F_2) \]
corresponds to the map
\[\pi_1(\bigvee_{m,n \geq 0} \Sigma^{2m+2n} H\F_2 \wedge_{H\Z}^{\mathbf{L}} H\F_2) \to \pi_1(\bigvee_{k,l \geq 0} \Sigma^{k+l} H\F_2 \wedge_{H\Z}^{\mathbf{L}} H\F_2) \to \pi_1(\bigvee_{k,l \geq 0} \Sigma^{k+l} H\F_2 \wedge_{H\F_2}^{\mathbf{L}} H\F_2),\]
 where the first map is induced by $\theta (\Phi^{\Z/2}\red) \zeta^{-1} \wedge  \theta (\Phi^{\Z/2}\red) \zeta^{-1}$ and the second map is the change of base. The only summands which can affect $\pi_1$ are $m=n=0$ and $k,l \in \{0,1\}$. In other words, we need to show that the map
\[\pi_1(H \F_2 \wedge_{H \Z}^{\mathbf{L}} H \F_2) \to \pi_1(H \F_2 \wedge_{H \F_2}^{\mathbf{L}} H \F_2) \oplus \pi_1(\Sigma H \F_2 \wedge_{H \F_2}^{\mathbf{L}} H \F_2) \oplus \pi_1 (H \F_2 \wedge_{H \F_2}^{\mathbf{L}} \Sigma H \F_2) \]
is zero. The first component of this map is zero because $\pi_1(H \F_2 \wedge_{H \F_2}^{\mathbf{L}} H \F_2)=\pi_1(H \F_2)=0$. The other two components are zero by Lemma \ref{techlemmabockst} below.
\end{proof}

\begin{lemma} \label{techlemmabockst} The composite
\[H \F_2 \xrightarrow{\incl_{n=0}} \bigvee_{n \geq 0} \Sigma^{2n} H \F_2 \xrightarrow{\theta \Phi^{\Z/2}(\red) \zeta^{-1}}  \bigvee_{n \geq 0} \Sigma^{n} H \F_2\xrightarrow{\proj_{n=1}}  \Sigma H \F_2 \]
is equal to zero.
 \end{lemma}

\begin{proof}
We begin by choosing exact triangles respectivelry in the homotopy category of $N^{\Z/2}_e H \Z$ and $N^{\Z/2}_e H \F_2$-modules
\[F_{\Z} \to N^{\Z/2}_e H \Z \stackrel{f_{\Z}}{\longrightarrow} H \Z\to \Sigma F_{\Z}\ \ \ \ \ \ \mbox{and}\ \ \ \ \ \ \ \ F_{\F_2} \to N^{\Z/2}_e H \F_2 \stackrel{f_{\F_2}}{\longrightarrow} H \F_2\to \Sigma F_{\F_2}.\] By pulling back the second triangle with respect to the mod $2$ reduction map we can regard it as a triangle in the homotopy category of $N^{\Z/2}_e H \Z$-modules. We get a morphism of triangles in the homotopy category of $N^{\Z/2}_e H \Z$-modules
\[\xymatrix@C=50pt{F_{\Z} \ar[r] \ar[d]^-r &  N^{\Z/2}_e H \Z \ar[r]^-{f_{\Z}} \ar[d]^-{N^{\Z/2}_e(\red)} & H \Z \ar[r] \ar[d]^-{\red} & \Sigma F_{\Z} \ar[d]^-{\Sigma r} \\  F_{\F_2} \ar[r] &  N^{\Z/2}_e H \F_2 \ar[r]^-{f_{\F_2}} & H \F_2 \ar[r] & \Sigma F_{\F_2}\rlap{\ ,} }\]
where $r$ is the map induced on homotopy fibers. From here we obtain a commutative diagram of $H \Z$-modules on geometric fixed-points
\[\xymatrix@C=50pt{\Phi^{\Z/2}F_{\Z} \ar[r] \ar[d]^{\Phi^{\Z/2} r} &  H \Z \ar[r]^-{\Phi^{\Z/2} f_{\Z}} \ar[d]^{\red} &  \Phi^{\Z/2} H \Z \ar[r] \ar[d]^{\Phi^{\Z/2} (\red)} & \Sigma \Phi^{\Z/2} F_{\Z} \ar[d]^{\Sigma \Phi^{\Z/2}  r} \\ \Phi^{\Z/2} F_{\F_2} \ar[r] &  H \F_2 \ar[r]^-{\Phi^{\Z/2} f_{\F_2}} & \Phi^{\Z/2} H \F_2 \ar[r] & \Sigma \Phi^{\Z/2} F_{\F_2},} \]
where the rows are exact triangles. We want to identify this diagram under the splittings provided by $\theta$ and $\zeta$.
Since the bottom row is pulled back from a triangle in the homotopy category of $H \F_2$-modules, it splits as
\[
\xymatrix@C=40pt{\Phi^{\Z/2} F_{\F_2} \ar[r] \ar[d]^{\cong}_{\theta_{F}} &  H \F_2 \ar[d]_{\id} \ar[r]^-{\Phi^{\Z/2} f_{\F_2}} & \Phi^{\Z/2} H \F_2 \ar[r] \ar[d]^{\cong}_{\theta} & \Sigma \Phi^{\Z/2} F_{\F_2} \ar[d]^{\cong}_{\Sigma \theta_{F}}
\\
\bigvee_{n \geq 1} \Sigma^{n-1} H \F_2  \ar[r]^-0 & H \F_2 \ar[r]^-{\incl_{n=0}} & \bigvee_{n \geq 0} \Sigma^{n} H \F_2 \ar[r]^-{\proj_{n\geq 1}} & \bigvee_{n \geq 1}  \Sigma^{n} H \F_2 }
\]
where the homotopy type of the fiber $\Phi^{\Z/2} F_{\F_2}$ is determined by the fact that $\Phi^{\Z/2} f_{\F_2}$ corresponds to a summand inclusion. 
Similarly, the top row splits in the homotopy category of $H \Z$-modules
\[\xymatrix@C=30pt{\Phi^{\Z/2}F_{\Z} \ar[r] \ar[d]^-{\cong}_{\zeta_{F}} &  H \Z \ar[r]^-{\Phi^{\Z/2} f_{\Z}} \ar[d]_{\id} &  \Phi^{\Z/2} H \Z \ar[r] \ar[d]^{\cong}_{\zeta} & \Sigma \Phi^{\Z/2} F_{\Z} \ar[d]^{\cong}_{\Sigma \zeta_{F}} 
\\
H \Z \vee \bigvee_{n \geq 1} \Sigma^{2n-1} H \F_2  \ar[r]^-{2 \vee 0} & H \Z \ar[r]^-{\red \vee 0} & H \F_2 \vee \bigvee_{n \geq 1} \Sigma^{2n} H \F_2 \ar[r]^-{\beta \vee \id} & \Sigma H \Z  \vee \bigvee_{n \geq 1} \Sigma^{2n} H \F_2\rlap{ ,}}\]
where $\beta$ is the Bockstein and the isomorphisms here are easily seen to be unique by the universal coefficient theorem. Now combining the latter three diagrams we get a morphism of exact triangles of $H \Z$-modules
\[\xymatrix{H \Z \vee \bigvee_{n \geq 1} \Sigma^{2n-1} H \F_2  \ar[r]^-{2 \vee 0} \ar[d]_{\theta_{F} (\Phi^{\Z/2}r) \zeta_{F}^{-1}} & H \Z \ar[d]^{\red} \ar[r]^-{\red \vee 0} & H \F_2 \vee \bigvee_{n \geq 1} \Sigma^{2n} H \F_2 \ar[r]^-{\beta \vee \id} \ar[d]^{\theta \Phi^{\Z/2}(\red) \zeta^{-1}} & \Sigma H \Z  \vee \bigvee_{n \geq 1} \Sigma^{2n} H \F_2 \ar[d]^{\Sigma \theta_{F} (\Phi^{\Z/2}r) \zeta_{F}^{-1}} \\ \bigvee_{n \geq 1} \Sigma^{n-1} H \F_2  \ar[r]^-0 & H \F_2 \ar[r]^-{\incl} & H \F_2 \vee \bigvee_{n \geq 1} \Sigma^{n} H \F_2 \ar[r]^-{\proj} & \bigvee_{n \geq 1}  \Sigma^{n} H \F_2\rlap{\ .}}\]
From the commutativity of the right most square it follows that the composite
\[\xymatrix@C=35pt{H \F_2 \ar[r]^-{\incl} & H \F_2 \vee \bigvee_{n \geq 1} \Sigma^{2n} H \F_2 \ar[rr]^-{\theta \Phi^{\Z/2}(\red) \zeta^{-1}} & & H \F_2 \vee \bigvee_{n \geq 1} \Sigma^{n} H \F_2 \ar[r]^-{\proj_{n=1}} & \Sigma H \F_2 }\]
is equal to zero if and only if $\pi_1(\Sigma \theta_{F} \Phi^{\Z/2}(r) \zeta_{F}^{-1})= 0$ and hence $\pi_0(\theta_{F} \Phi^{\Z/2}(r) \zeta_{F}^{-1})=0$, or equivalently if $\pi_0(\Phi^{\Z/2}r)=0$. Since the Frobenius maps $f_{\Z}$ and $f_{\F_2}$ are isomorphisms on $\pi_0$ of underlying spectra, we must have that $\pi_0(F_{\Z})=0$ and $\pi_0(F_{\F_2})=0$. This implies that the horizontal maps in the commutative diagram
\[\xymatrix@C=50pt{\pi_0^{\Z/2}(F_{\Z}) \ar[r]^{\cong} \ar[d]_{\pi_0^{\Z/2}(r)} & \pi_0(\Phi^{\Z/2} F_{\Z}) \ar[d]^{\pi_0(\Phi^{\Z/2}r)} \\ \pi_0^{\Z/2}(F_{\F_2}) \ar[r]^{\cong} & \pi_0(\Phi^{\Z/2} F_{\F_2})}\]
are isomorphisms, and therefore we can equivalently show that $\pi^{\Z/2}_0(r)=0$.
From Proposition \ref{pi_0norm} we see that $\pi_0^{\Z/2}(r)$ coincides up to isomorphism with the map induced on the kernels in the diagram
\[\xymatrix{0 \ar[r] & \Z \ar@{-->}[d] \ar[r]^-i & W_2(\Z) \ar[d]^{\red} \ar[r]^-{w_1} & \Z \ar[d]^-{\red} \ar[r] & 0 \\ 0 \ar[r] & \F_2 \ar[r]^-j & W_2(\F_2) \ar[r]^-{w_1} & \F_2 \ar[r] & 0\rlap{\ ,} }\]
where $w_1$ denotes the first ghost map, given by $w_1(a,b)=a^2+2b$. It is easy to calculate that $i(m)=(2m, -2m^2)$ and $j(l)=(0,l)$. It follows that the dashed arrow is equal to zero, and hence $\pi_0^{\Z/2}(r)=0$.
\end{proof}

\begin{prop}\label{prop:piTHR(Z)}
There are isomorphisms of abelian groups
\begin{align*}\pi_0 (\THR(\Z)^{\Z/2}) \cong & \Z\\
\pi_1 (\THR(\Z)^{\Z/2}) \cong & \Z/2\\
\pi_2 (\THR(\Z)^{\Z/2}) \cong & \Z/2. 
\end{align*}
The generator of $\pi_1 (\THR(\Z)^{\Z/2})$ maps to $e \in \pi_1 \Phi^{\Z/2}\THR(\Z)$ and squares to $0$. The generator of $\pi_2 (\THR(\Z)^{\Z/2})$ maps to $b_1 + b_2 \in \pi_2 \Phi^{\Z/2}\THR(\Z)$ and has infinite multiplicative order.
\end{prop}

\begin{proof}
The $\pi_0$-statement follows immediately from Theorem \ref{pi_0THR}. For the remaining isomorphisms we consider the long exact sequence of homotopy groups for the isotropy separation sequence of $\THR(\Z)$. In the subsequence
\[ \pi_2 \Phi^{\Z/2}\THR(\Z)\twoheadrightarrow\! \pi_1 (\THR(\Z)_{h\Z/2}) \stackrel{0}{\to}\! \pi_1(\THR(\Z)^{\Z/2}) \stackrel{\cong}{\to}\! \pi_1 \Phi^{\Z/2}\THR(\Z) \stackrel{0}{\to}\! \pi_0 (\THR(\Z)_{h\Z/2}),\]
the rightmost map is $0$, because the transfer map on $\pi_0$ is  injective, by Theorem \ref{pi_0THR}. It follows that the middle right hand map is surjective. Using the unit map $H\Z \to \THR(\Z)$ we see that the leftmost map is surjective and so, by exactness, the middle right hand map is an isomorphism. By Theorem \ref{thm:phiTHR(Z)} the target group $\pi_1 \Phi^{\Z/2}\THR(\Z)$ is isomorphic to $\Z/2$.

The long exact sequence continues to the left with
\[  0=\pi_2 (\THR(\Z)_{h\Z/2}) \to \pi_2(\THR(\Z)^{\Z/2}) \hookrightarrow \pi_2 \Phi^{\Z/2}\THR(\Z) \twoheadrightarrow \pi_1 (\THR(\Z)_{h\Z/2})\cong\Z/2,\]
and we have already seen that the right hand map is surjective. From the homotopy orbit spectral sequence for $\THR(\Z)$ we see that $\pi_2 (\THR(\Z)_{h\Z/2}) = 0$ and $\pi_1 (\THR(\Z)_{h\Z/2}) \cong \Z/2$. Since $\pi_2\Phi^{\Z/2}\THR(\Z)$ is isomorphic to $\Z/2 \oplus \Z/2$ it follows by exactness that $\pi_2(\THR(\Z)^{\Z/2})$ is isomorphic to $\Z/2$.

Using that the canonical map $S^{1,1}_+ \wedge H\Z \to \THR(\Z)$ induces an isomorphism on $\pi_2 \Phi^{\Z/2}(-)$ one can check that the kernel of the right hand map is a copy of $\Z/2$ generated by $b_1 + b_2$. The generator $b$ of $ \pi_2(\THR(\Z)^{\Z/2})$ maps to the element $b_1 + b_2$ of $\pi_2 \Phi^{\Z/2}\THR(\Z)$, which has infinite multiplicative order. Since $\THR(\Z)^{\Z/2} \to  \Phi^{\Z/2}\THR(\Z) $ is a map of ring spectra, the element $b$ must also have infinite order. The generator $e'$of $\pi_1 (\THR(\Z)^{\Z/2})$ maps to $e \in \pi_1 \Phi^{\Z/2}\THR(\Z)$. Since $e$ does not divide $b_1 + b_2$ it follows that $e'$ cannot divide $b$. Hence there is no room for $(e')^2$ in $\pi_2 (\THR(\Z)^{\Z/2}) \cong \Z/2\{b\}$ to be non-zero.
\end{proof}

\begin{rem}
B{\"o}kstedt showed in \cite{Bok2} that $\THH(\Z)$ is equivalent to the wedge of Eilenberg-MacLane spectra
\[\THH(\Z)\simeq H\Z \vee \bigvee_{k \geq 1} \Sigma^{2k-1} H\Z/k.\]
Hence, for parity reasons the multiplication in $\pi_\ast \THH(\Z)$ must be zero in positive degrees. However, the multiplicative structure on the spectrum $\THH(\Z)$ is far from trivial, as can be seen by taking homology with $\F_p$ coefficients. Theorem \ref{thm:phiTHR(Z)} and Proposition \ref{prop:piTHR(Z)} show that in the $\Z/2$-equivariant case the complexity of the multiplicative structure is apparent already at the level of homotopy groups. In future work we hope to calculate the full bigraded homotopy ring $\pi_{\ast,\ast} \THR(\Z)$ and determine the $\Z/2$-equivariant homotopy type of $\THR(\Z)$. Based on the calculations of this section we conjecture that there is a stable equivalence of $\Z/2$-equivariant spectra
\[ \THR(\Z)\simeq H\Z \vee \bigvee_{k \geq 1} \Sigma^{2k-1,k} H\Z/k.\]
\end{rem}

We show now that this formula holds after localizing at an odd prime $p$.

\begin{theorem} \label{THRZp} Let $p$ be an odd prime. There is a stable equivalence of $\Z/2$-spectra
\[\THR(\Z)_{(p)} \simeq\THR(\Z_{(p)})\simeq  H\Z_{(p)} \vee \bigvee_{k \geq 1} \Sigma^{2k-1, k} H(\Z/p^{\nu_p(k)})\]
where $\nu_p(k)$ is the $p$-adic valuation of $k$.
\end{theorem}

Combining the results for the individual primes gives the following.

\begin{cor} \label{THRZ[1/2]} There is a stable equivalence of $\Z/2$-spectra
\[\THR(\Z)[\tfrac{1}{2}] \simeq\THR(\Z[\tfrac{1}{2}]) \simeq H\Z[\tfrac{1}{2}] \vee \bigvee_{k \geq 1} \Sigma^{2k-1, k} H(\Z/k)[\tfrac{1}{2}].\]
\end{cor}

The proof of this theorem is similar in sprit to the computation of $\THR(\mathbb{F}_p)$. The main difference is that since the wedge of Eilenberg MacLane spectra in question is not a free algebra, we need to define a map individually on each summand. We start by identifying the action of $\Z/2$ on the underlying homotopy type of $\THR(\Z)$.

\begin{lemma} \label{actionTHRZp} Let $p$ be an odd prime. The map $w\colon \THH(\Z_{(p)}) \to \THH(\Z_{(p)})$ is equivalent in the stable homotopy category to the wedge product
\[1 \vee \bigvee_{k \geq 1} (-1)^k \colon H\Z_{(p)} \vee \bigvee_{k \geq 1} \Sigma^{2k-1} H\Z/p^{\nu_p(k)} \to H\Z_{(p)} \vee \bigvee_{k \geq 1} \Sigma^{2k-1} H\Z/p^{\nu_p(k)}.\]
In other words, $w$ acts by the identity in degree $0$ and in degrees $3$ mod $4$, and by $-1$ in degrees $1$ mod $4$ (in the degrees where the $p$-adic valuation is zero, the groups are trivial and $-1=1$).
\end{lemma}

\begin{proof} Since $H\Z_{(p)}$ is a $\Z/2$-equivariant commutative ring spectrum, it follows from Corollary \ref{multTHR} that $\THR(\Z_{(p)})$ is  an $H\Z_{(p)}$-module. The involution on $\Z_{(p)}$ is trivial, so we see that $w\colon \THH(\Z_{(p)}) \to \THH(\Z_{(p)})$ is $H\Z_{(p)}$-linear. According to \cite{Bok2} the splitting
\[\THH(\Z_{(p)}) \simeq H\Z_{(p)} \vee \bigvee_{k \geq 1} \Sigma^{2k-1} H\Z/p^{\nu_p(k)}\]
is $H\Z_{(p)}$-linear, i.e. it is an equivalence in the homotopy category of $H\Z_{(p)}$-modules. But the latter is triangulated equivalent to the classical derived category of $\Z_{(p)}$ (see \cite{SchShip1}) and hence admits universal coefficient exact sequences. Given two $H\Z_{(p)}$-modules $X$ and $Y$, the sequence looks as follows
\[ \xymatrix{0 \ar[r] & \Ext^1_{\mathbb{Z}}(\pi_\ast X[1], \pi_\ast Y) \ar[r] & [X,Y] \ar[r] & \Hom_{\Z}(\pi_\ast X, \pi_\ast Y)  \ar[r] & 0,}\]
where $[X,Y]$ stands for morphisms in the homotopy category of $H\Z_{(p)}$-modules. This implies that $w$ splits in the homotopy category of $H\Z_{(p)}$-modules as a wedge
\[w_0 \vee \bigvee_{k \geq 1} w_k\colon H\Z_{(p)} \vee \bigvee_{k \geq 1} \Sigma^{2k-1} H\Z/p^{\nu_p(k)} \to H\Z_{(p)} \vee \bigvee_{k \geq 1} \Sigma^{2k-1} H\Z/p^{\nu_p(k)},\]
and that each $w_i$ is in fact an integer. Indeed, the potential non-diagonal terms correspond to elements in $\Ext$-groups, and since there is a shift involved these $\Ext$-terms vanish for $k \geq 2$ for degree reasons. Additionally, when $X=H\Z_{(p)}$ the $\Ext$-term obviously vanishes.

Now since $\Z_{(p)}$ has the trivial involution, we see that $w_0=1$. Moreover the identity $w^2=1$  implies that $w_k^2=1$ mod $p^{\nu_p(k)}$ for any $k \geq 1$. From now on we consider only those degrees $k$ for which $\nu_p(k) \neq 0$.
We observe that since $p$ is odd, if an integer $a$ satisfies the equation $a^2=1$ mod $p^l$ for some $l \geq 1$ we must have $a = \pm 1$ mod $p^l$. We conclude that $w_k=\pm 1$ for any $k \geq 1$.

To determine the sign of $w_k$ we use B\"okstedt's result \cite{Bok2}, which implies that the map $\THH(\Z_{(p)}) \to \THH(\mathbb{F}_p)$ is a product of Bocksteins. Moreover by Corollary \ref{poddTHRFp} the involution on $\pi_{2k} \THH(\mathbb{F}_p)$ is $(-1)^k$, and we obtain a commutative diagram
\[\xymatrix{\Sigma^{2k-1} H\Z/p^{\nu_p(k)} \ar[r]^-\beta \ar[d]^-{w_k} & \Sigma^{2k} H\mathbb{F}_p \ar[d]^-{(-1)^k} \\ \Sigma^{2k-1} H\Z/p^{\nu_p(k)} \ar[r]^-\beta & \Sigma^{2k} H\mathbb{F}_p, }\]
where $\beta$ is the Bockstein. It follows that $(w_k-(-1)^k)\beta=0$. Since $w_k=\pm 1$ the difference $w_k-(-1)^k$ can be $0$, $-2$ or $2$. But the latter two cases are impossible since $p$ is odd and the Bockstein is nontrivial. Hence $w_k-(-1)^k=0$ which proves our claim. \end{proof}

\begin{proof}[Proof of Theorem \ref{THRZp}] The equivalence
\[\THH(\Z)_{(p)} \simeq H\Z_{(p)} \vee \bigvee_{k \geq 1} \Sigma^{2k-1} H\Z/p^{\nu_p(k)}\]
gives maps $\incl_k\colon \Sigma^{2k-1} H\Z/p^{\nu_p(k)} \to \THH(\Z)_{(p)}$
for $k \geq 1$. By the restriction-induction adjunction we get maps in the $\Z/2$-equivariant stable homotopy category:
\[\xymatrix{\Z/2_+ \wedge \Sigma^{2k-1} H\Z/p^{\nu_p(k)} \ar[r]^-{\widetilde{\incl_k}} & \THR(\Z)_{(p)}}\]
for $k\geq 1$.
Composing with the morphisms 
\[S^{2k-1,k} \wedge H\Z/p^{\nu_p(k)} \to \Z/2_+ \wedge \Sigma^{2k-1} H\Z/p^{\nu_p(k)}\]
we obtain maps in the $\Z/2$-equivariant stable homotopy category
\[\overline{\incl_k}\colon S^{2k-1,k} \wedge H\Z/p^{\nu_p(k)} \to  \THR(\Z)_{(p)}. \]
Together with the unit map $H \Z_{(p)} \to \THR(\Z_{(p)})$ they assemble into a map
\[ H\Z_{(p)} \vee \bigvee_{k \geq 1} \Sigma^{2k-1, k} H\Z/p^{\nu_p(k)} \to \THR(\Z)_{(p)}. \]
We will now argue that this map is a genuine $\Z/2$-equivariant equivalence. The $\Z/2$-geometric fixed points of both sides vanish since $p$ is odd. Hence it suffices to show that this map is an underlying equivalence. We observe that on underlying spectra the map $\tran \colon S^0 \to \Z/2_+$ is the stable pinch map, and the map $S^{2k-1,k} \to  \Z/2_+ \wedge S^{2k-1}$ is equivalent as a map on non-equivariant spectra to the map
\[(1,(-1)^k) \colon S^{2k-1} \to S^{2k-1} \vee S^{2k-1}.\]
Using this and Lemma \ref{actionTHRZp}, we see that on underlying spectra the map $\overline{\incl_k}$ is the composite
\[
\Sigma^{2k-1} H\Z/p^{\nu_p(k)} \xrightarrow{(1,(-1)^k) } \Sigma^{2k-1} H\Z/p^{\nu_p(k)} \vee \Sigma^{2k-1} H\Z/p^{\nu_p(k)} \xrightarrow{(\incl_k,(-1)^k\incl_k)}\THH(\Z_{(p)})
\]
for all $k \geq 1$, which is equal to $2\incl_k$. By construction the maps $\incl_k$ assemble into an equivalence of spectra. Since $p$ is odd, $2$ is invertible and hence so do the maps $\overline{\incl_k}$. This completes the proof.
\end{proof}



\appendix
\section{Flatness  and ring spectra with anti-involution}\label{appendixflat}

In this appendix we prove some technical results on ring spectra with anti-involution and on flatness of orthogonal $G$-spectra. We recall some definitions, from  \cite{ManMay}, \cite{MMSS} and \cite{Sto, Schwedeglobal}.

Let $G$ be a compact Lie group. An orthogonal $G$-spectrum is a $\Top_\ast$-enriched functor
\[X \colon O \to \Top_\ast^G,\]
where $O$ is the category whose objects are the finite dimensional real inner product spaces and where the space of morphisms $O(V,W) $ is the Thom space of the orthogonal complement bundle over the Stiefel manifold $L(V,W)$ of linear isometric embeddings of $V$ into $W$. This space consists of the unique point $\infty$ at infinity and pairs $(\alpha,x)$ of a linear isometric embedding $\alpha\colon V\to W$ and $x \in W - \alpha(V)$. The category of orthogonal $G$-spectra is the $\Top_\ast$-enriched functor category $\Sp^G:=\Fun(O, \Top_\ast^G)$.
Let $O^{\leq m-1}$ denote the full subcategory of $O$ spanned by the vector spaces of dimension $\leq m-1$. The inclusion $\iota_{m-1} \colon O^{\leq m-1} \hookrightarrow O$
induces an adjunction
\[
\xymatrix@C=60pt{
\Sp^G=\Fun(O, \Top_\ast^G)\ar@<1ex>[r]^-{\iota_{m-1}^\ast} 
&
\Fun(O^{\leq m-1},\Top_\ast^G)
\ar@<1ex>[l]^-{(\iota_{m-1})_!}
},\]
where $(\iota_{m-1})_!$ is the $\Top_\ast$-enriched left Kan extension. We recall that the $(m-1)$-skeleton of a $G$-spectrum $X$ is defined by
\[\sk_{m-1}X := (\iota_{m-1})_! \iota_{m-1}^\ast X.\]
The counit of the adjunction is a map of orthogonal $G$-spectra $\varepsilon_X \colon \sk_{m-1}X \to X$, which evaluated at $\R^m$ this gives as $G \times O(m)$-equivariant map $(\sk_{m-1}X)(\R^m) \to X(\R^m)$.
\begin{defn} The $m$-th latching space of an orthogonal $G$-spectrum $Y$ is defined by
\[L_m X := (\sk_{m-1}X)(\R^m) .\]
The orthogonal  $G$-spectrum $X$ is called flat (or $\mathbb{S}$-cofibrant) if the latching map $L_mX \to X(\R^m)$ is a $G \times O(m)$-cofibration,
for every $m\geq 0$.
\end{defn}

The notion of $\mathbb{S}$-cofibrant $G$-spectra was first introduced in \cite[Section 2.3.3]{Sto}. It coincides with the notion of flat $G$-spectrum of \cite[Proposition 5.9]{Schwedeglobal}. We note that this definition does not depend on a choice of universe. The term ``flat'' is justified by the fact that the functor $X \wedge -$ preserves genuine $G$-equivariant stable equivalences whenever $X$ is a flat $G$-spectrum (see \cite[Theorem 5.10]{Schwedeglobal} and \cite[Proposition 2.10.1]{BrDuSt}).

\subsection{A model structure on ring spectra with anti-involution}\label{flatMS}

Let $f\colon (A,w)\to (B,\sigma)$ be a morphism of ring spectra with anti-involution. We say that $f$ is a \emph{stable $\Z/2$-equivalence} if the underlying map of orthogonal $\Z/2$-spectra is a stable equivalence with respect to a complete universe of $\Z/2$-representations.

\begin{prop}\label{antimodelstructure}
There is a cofibrantly generated model structure on the category of orthogonal ring spectra with anti-involution, where the weak equivalences are the stable $\Z/2$-equivalences and the cofibrations are the cofibrations of \cite[B.63]{HHR}.
\end{prop}

\begin{rem}\label{flatrepl}
It follows immediately from Proposition \ref{antimodelstructure} that any ring spectrum with anti-involution can be replaced up to equivalence by one whose underlying $\Z/2$-spectrum is cofibrant in the sense of \cite[B.63]{HHR}, and therefore flat.

A similar model structure where the cofibrations are the flat cofibrations would exist, if one can prove that the flat model structure on orthogonal $\Z/2$-spectra satisfies the indexed version of the pushout product axiom of \cite[B.97-B.102]{HHR}.
\end{rem}

The argument is essentially the one of \cite[4.1(3)]{SchwShip}. 
Let $T^\sigma \colon \Sp^{\mathbb{Z}/2} \to \Sp^{\mathbb{Z}/2}$ be the \emph{twisted tensor algebra functor}, which sends $X$ in $\Sp^{\mathbb{Z}/2}$ to
\[
T^\sigma(X)=\bigvee_{n\geq 0}X^{\wedge n}.
\]
The $\Z/2$-action on $T^\sigma(X)$ preserves the summands, and is given on the $n$-th summand by the map
\[
X^{\wedge n}\xrightarrow{w^{\wedge n}}X^{\wedge n}\xrightarrow{\tau_{n}}X^{\wedge n},
\]
where $w$ is the involution of $X$ and $\tau_n$ is the permutation of $\underline{n}=\{1,\dots, n\}$ that reverses the order. The usual concatenation product $\mu_X \colon  T^\sigma(T^\sigma(X)) \to T^\sigma(X)$ is $\Z/2$-equivariant, as is the rest of the monad structure on $T^\sigma$. It is not hard to see that ring spectra with anti-involution are precisely the algebras for $T^\sigma$ (cf. Remark \ref{operad}). We write $\Mon^{\sigma}(\Sp^{\mathbb{Z}/2})$ for the category of $T^\sigma$-algebras. The corresponding free-forgetful adjunction is
\[\xymatrix@C=50pt{\Mon^{\sigma}(\Sp^{\mathbb{Z}/2})\ar@<-.5ex>[r]_-U&\Sp^{\Z/2}\ar@<-.5ex>[l]_-{T^\sigma}}.\]


\begin{proof}[Proof of \ref{antimodelstructure}]
To show that there is such a model structure on $\Mon^{\sigma}(\Sp^{\mathbb{Z}/2})$ we follow the argument of  \cite[4.1(3)]{SchwShip} and test the conditions of \cite[6.2]{SchwShip}. The non-trivial condition to verify is the analog of \cite[6.2]{SchwShip}, which is based on the construction of a filtration of the pushouts of the form $X\leftarrow T^\sigma(K)\rightarrow T^\sigma(L)$ in the category of ring spectra with anti-involution, where $K\to L$ is a map of orthogonal $\Z/2$-spectra. The filtration
\[
X=:P_0\to P_1\to\dots\to P_n\to \dots\to P
\]
of the pushout $P$ is constructed in \cite[Section 6]{SchwShip} inductively as follows. Let $\mathcal{P}(\underline{n})$ denote the poset of subsets of $\underline{n}=\{1,\dots,n\}$ ordered by inclusion. The authors define an $n$-cube $W_n\colon \mathcal{P}(\underline{n})\to \Sp$ of spectra by sending a subset $S$ of $\underline{n}$ to the spectrum
\[
W_n(S):=X\wedge C_1\wedge X\wedge C_2\wedge X\wedge \dots\wedge X\wedge C_n\wedge X
\]
where $C_i=L$ if $i\in S$ and $C_i=K$ otherwise. The spectrum $P_n$ is then defined as the pushout
\[
\xymatrix{
Q_n\ar[r]\ar[d]&X^{\wedge n+1}\wedge L^{\wedge n}\ar[d]
\\
P_{n-1}\ar[r]&P_n
}
\]
where $Q_n:=\colim_{\mathcal{P}_1(\underline{n})}W_n$ and $\mathcal{P}_1(\underline{n})$ is the subposet of $\mathcal{P}(\underline{n})$ obtained by removing the terminal object. The left hand vertical map is defined in the proof of \cite[6.2]{SchwShip}. The authors show that the colimit spectrum $P=\colim_nP_n$ is a monoid, and that it satisfies the universal property for the pushout of $X\leftarrow T^\sigma(K)\rightarrow T^\sigma(L)$ in the category or ring spectra.

We show that the spectra $P_n$ carry $\Z/2$-actions that induce an anti-involution on $P$, and that this defines the pushout of $X\leftarrow T^\sigma(K)\rightarrow T^\sigma(L)$ in the category of ring spectra with anti-involution. The cubes $W_n$ are naturally $\Z/2$-equivariant cubes, where the poset $\mathcal{P}(\underline{n})$ has the involution induced by the involution $\tau_n$ of $\underline{n}$ that reverses the order. The $\Z/2$-structure is defined by the natural transformation
\[
W_n(S)=X^{\wedge n+1}\wedge \bigwedge_{i\in S}L\wedge\bigwedge_{i\notin S}K\longrightarrow X^{\wedge n+1}\wedge \bigwedge_{i\in \tau_n S}L\wedge\bigwedge_{i\notin \tau_nS}K= W_n(\tau_nS)
\]
which is the smash of the following maps. On the $X^{\wedge n+1}$-factor it is the map $\tau_{n+1}\circ w^{\wedge n+1}$, where $w$ is the involution of $X$. On the $S$-indexed smash it sends the $i$-factor to the $\tau_ni$-factor by the involution of $L$. On the smash factor indexed by the complement of $S$ it sends the $i$-factor to the $\tau_ni$-factor by the involution of $K$. The colimit $\colim_{\mathcal{P}_1(\underline{n})}W_n$ therefore inherits a $\mathbb{Z}/2$-action and the map to the terminal vertex
\[
\colim_{\mathcal{P}_1(\underline{n})}W_n\longrightarrow X^{\wedge n+1}\wedge L^{\wedge n}
\]
is equivariant. It is straightforward to show that the left vertical map of the square above is also equivariant, and therefore the pushout defining $P_n$ from $P_{n-1}$ is a pushout of $\Z/2$-spectra. This defines a filtration of $\Z/2$-spectra
\[
X=:P_0\to P_1\to\dots\to P_n\to \dots\to P:=\colim_nP_n.
\]
Let us now verify that the resulting involution on the colimit $P$ reverses the order of the multiplication. The multiplication on $P$ is defined in the proof of \cite[6.2]{SchwShip} by maps $\mu_{n,m} \colon P_n\wedge P_m\to P_{n+m}$. For $n + m = 0$ we have $P_n \wedge P_m = X \wedge X$ and we take $\mu_{0,0}$ to be the multiplication map of $X$. For $n + m > 0$ we define $\mu_{n,m}$ by induction on $n+m$ by expressing  $P_n\wedge P_m$ as an iterated pushout where the building blocks are themselves pushouts of terms that depend only on $P_k\wedge P_l$ where $k+l<n+m$. In order to verify the compatibility between these maps and the involution it is convenient to write $P_n\wedge P_m$ as the colimit of the diagram of $\Z/2$-spectra
\[
W_{n,m}=\left(
\vcenter{\xymatrix{
P_{n-1}\wedge P_{m-1}
&
Q_n\wedge P_{m-1}\ar[l]\ar[r]
&
X^{\wedge n+1}\wedge L^{\wedge n}\wedge P_{m-1}
\\
P_{n-1}\wedge Q_{m}\ar[u]\ar[d]
&
Q_n\wedge Q_m\ar[l]\ar[r]\ar[u]\ar[d]
&
X^{\wedge n+1}\wedge L^{\wedge n}\wedge Q_m\ar[u]\ar[d]
\\
P_{n-1}\wedge X^{\wedge m+1}\wedge L^{\wedge m}
&
Q_n\wedge X^{\wedge m+1}\wedge L^{\wedge m}\ar[l]\ar[r]
&
X^{\wedge n+1}\wedge L^{\wedge n}\wedge X^{\wedge m+1}\wedge L^{\wedge m}
}}
\right).
\]
The map $P_n\wedge P_m\to P_{n+m}$ is then determined by the following maps out of the four outer corners of $W_{n,m}$, where $\mu_{l,k}\colon P_l\wedge P_k\to P_{l+k}$ has already been inductively defined for $l+k<n+m$:
\[
\xymatrix@R=15pt{
P_{n-1}\wedge P_{m-1}\ar[rr]^-{\mu_{n-1,m-1}}&& P_{n+m-2}\ar[r]& P_{n+m}
\\
X^{\wedge n+1}\wedge L^{\wedge n}\wedge P_{m-1}\ar[r]& P_n\wedge P_{m-1}\ar[r]^-{\mu_{n,m-1}} &P_{n+m-1}\ar[r]& P_{n+m}
\\
P_{n-1} \wedge X^{\wedge m+1} \wedge L^{\wedge m} \ar[r] & P_{n-1}\wedge P_m \ar[r]^-{\mu_{n-1,m}} & P_{n-1 + m} \ar[r] &  P_{n+m}
\\
X^{\wedge n+1}\wedge L^{\wedge n}\wedge X^{\wedge m+1}\wedge L^{\wedge m}\ar[r]& X^{\wedge n+m+1}\wedge L^{\wedge n+m}\ar[rr]&& P_{n+m}.
}
\]
The first map in the last composite swaps the $L^{\wedge n}$ and $X^{\wedge m+1}$ factors and multiplies the last factor of $X^{\wedge n+1}$ and the first factor of $X^{\wedge m+1}$.  These maps define a monoid structure on the colimit $P$. The compatibility of this multiplication with the involution reduces to the commutativity of the following diagram
\[
\xymatrix{
\displaystyle P_{n}\wedge P_m\ar@<-7ex>[d]_-{\tau}\cong \colim_{\mathcal{P}_{1}(2)^{\times 2}}W_{n,m}
\ar[rr]^{w\wedge w}\ar@<5ex>[d]_-{}
&&
\displaystyle\colim_{\mathcal{P}_{1}(2)^{\times 2}}W_{n,m}
\ar[d]^{\mu_{n,m}}
\\
\displaystyle P_{m}\wedge P_n\cong \colim_{\mathcal{P}_{1}(2)^{\times 2}}W_{m,n}
\ar[r]_-{\mu_{m,n}}
&
P_{m+n}\ar[r]_-{w}&P_{n+m}
}
\]
where the middle vertical map is the composite of the maps
\[
\colim_{\mathcal{P}_{1}(2)^{\times 2}}W_{n,m}\longrightarrow \colim_{\mathcal{P}_{1}(2)^{\times 2}}\tau^{\ast}W_{m,n}\stackrel{\tau_\ast}{\longrightarrow} \colim_{\mathcal{P}_{1}(2)^{\times 2}}W_{m,n}
\]
of the push-forward by the functor $\tau\colon \mathcal{P}_{1}(2)^{\times 2}\to \mathcal{P}_{1}(2)^{\times 2}$ that switches the product factors, and of the natural transformation $W_{n,m}\to \tau^{\ast}W_{m,n}$ that permutes the two smash factors in each entry of $W_{n,m}$. It is immediate to see that the left hand square commutes. The right hand square commutes essentially because the order-reversing permutations $\tau_n\in \Sigma_n$ satisfy 
\[
\tau_n\times\tau_m=\tau_{n+m}\circ \chi_{n,m}
\]
where $\chi_{n,m}\in \Sigma_{n+m}$ switches the blocks of the first $n$-elements and the last $m$-elements.

This shows that the colimit $P$ is a ring spectrum with anti-involution. In order to show that $P$ is the pushout of $X\leftarrow T^\sigma(K)\rightarrow T^\sigma(L)$ in the category $\Mon^{\sigma}(\Sp^{\mathbb{Z}/2})$, we notice that the morphism sets in $\Mon^{\sigma}(\Sp^{\mathbb{Z}/2})$ from $(A,w)$ to $(B,\sigma)$ is the fixed-point set of the involution on the morphism set of underlying ring spectra that sends $f\colon A\to B$ to $\sigma f w$. Therefore it suffices to show that the bijection
\[
\Hom_{\Mon(\Sp)}(X\leftarrow T(K)\rightarrow T(L), A=A=A)\cong \Hom_{\Mon(\Sp)}(P,A)
\]
established in \cite[6.2]{SchwShip} is equivariant. This is readily verified.

We are now left with verifying the analog of \cite[6.2]{SchwShip}. That is, we need to prove that if $K\to L$ is an acyclic cofibration of orthogonal $\Z/2$-spectra with cofibrant source, the map $X\to P$ is an acyclic cofibration of underlying orthogonal $\Z/2$-spectra. The argument is identical to the one of \cite[6.2]{SchwShip}, by using the equivariant version of the pushout product axiom of \cite[B.102]{HHR}.
\end{proof}

\subsection{Shifts preserve flatness of orthogonal \texorpdfstring{$\Z/2$}{Z/2}-spectra}\label{shiftflat}

Let $G$ be a compact Lie group.
Recall that an orthogonal $G$-spectrum $Y$ can be evaluated on a real $n$-dimensional $G$-representation $V$ by the formula
\[Y(V) = L(\R^n,V)_+ \wedge_{O(n)} Y_n,\]
where $L(\R^n,V)_+$ is the space of linear isometries from $\R^n$ to $V$ with a disjoint basepoint, and the $G$-action is diagonal. The shift functor $\Sh^V \colon \Sp^G \to \Sp^G$ is defined by precomposing with the functor $(V\oplus -)$ on $O$, that is
\[(\Sh^VY)_n = Y(V \oplus \R^n).\]
The goal of this section is to prove the following.

\begin{prop}\label{prop:flatness}
Let $f$ be a flat cofibration of orthogonal $\Z/2$-spectra, and $V$ a finite dimensional orthogonal $\Z/2$-representation. Then the map $\Sh^Vf$ is a flat cofibration.
\end{prop}

\begin{rem}
The proof we propose works for any compact Lie group $G$, up until the description of the coequalizer in the proof of \ref{shiftsemifree}, where we use that the irreducible representations of $\Z/2$ are one-dimensional. We believe that with more work the argument can be modified to apply to any compact Lie group. 
\end{rem}
We begin with a lemma of point-set topology This is a stable analog of \cite[Proposition 3.11]{SagSchlichtdiag} and \cite[Proposition 1.3.6]{Schwedeglobal}.

\begin{lemma}\label{lemma:closed}
Let $Y$ be an orthogonal spectrum satisfying the following properties:
\begin{enumerate}
\item \label{lemma:closed-i} $Y(\R^n)$ is compact Hausdorff for all $n \geq 0$,
\item \label{lemma:closed-ii}For any morphism $(\alpha,x) \colon V \to W \textrm{ in } O$ the induced map $(\alpha,x)_\ast \colon Y(V) \to Y(W)$
is injective,
\item \label{lemma:closed-iii}Suppose we are given morphisms in $O$
\[ V \xrightarrow{(\alpha,x)}W \xleftarrow{(\alpha',x')} V' ,\]
where $\alpha$ and $\alpha'$ are subspace inclusions, and let $\iota\colon V\cap V'\to V$ and $\iota'\colon V\cap V'\to V'$ the inclusions. Then there exist $v\in V-V\cap V'$ and $v'\in V'-V\cap V'$ such that $x+\alpha(v)=x'+\alpha(v')$ and the canonical map into the pullback
\[\xymatrix@R=20pt{
Y(V\cap V')\ar@/^15pt/[rr]^{(\iota,v)_\ast}\ar@/_15pt/[dr]_-{(\iota',v')_\ast}\ar@{-->}[r]
&
P\ar[r]\ar[d]
&
Y(V)\ar[d]^-{(\alpha,x)_\ast}
\\
&
Y(V')\ar[r]_-{(\alpha',x')_\ast}
&Y(W)
}
\]
is surjective.
 \end{enumerate}
Then the map $L_m Y \to Y(\R^m)$ is a closed inclusion for all $m\geq 0$.
\end{lemma}

\begin{proof}
The latching space $L_mY$ can be calculated as a coequalizer
\[ \bigvee_{0 \leq s \leq l \leq m-1} O(\R^l,\R^m) \wedge O(\R^s,\R^l) \wedge Y(\R^s) \rightrightarrows \bigvee_{0 \leq l \leq m-1} O(\R^l,\R^m) \wedge Y(\R^l) \to L_mY, \]
hence $L_mY$ is compact. Since $Y(\R^m)$ is Hausdorff it suffices to show that the map $L_m Y \to Y(\R^m)$ is injective.
 Let $(\phi,t) \in O(\R^i,\R^m), \, y \in Y(\R^i)$ and $(\phi',t') \in O(\R^j,\R^m), \, y' \in Y(\R^j)$ be such that
\begin{equation}\label{eq:yy'}(\phi,t)_\ast(y) = (\phi',t')_\ast(y') \in Y(\R^m),\end{equation}
and let us show that the elements $((\phi,t),y)$ and $((\phi',t'),y')$ are identified in the coequalizer. We set $V = \phi(\R^i)$ and $V' = \phi'(\R^j)$ and factor $\phi$ as an isomorphism $\bar{\phi}$ onto $V$ followed by an inclusion $j$, and similarly $\phi' = j' \circ \bar{\phi}'$.

Consider the morphisms in $O$
\[ V \stackrel{(j,t)}{\longrightarrow} \R^m \stackrel{(j',t')}{\longleftarrow} V'.\]
By (\ref{eq:yy'}) it follows that $(j,t)_\ast(\bar{\phi},0)_\ast(y) =  (j',t')_\ast(\bar{\phi'},0)_\ast(y')$ as elements of $Y(\R^m)$.
Condition $\ref{lemma:closed-iii}$ now gives a commutative diagram
\[\xymatrix@C=50pt{V \cap V' \ar[r]^-{(\iota',v')} \ar[d]_{(\iota,v)} & V' \ar[d]^{(j',t')}\\
V \ar[r]_{(j,t)} & \R^m ,} \]
and a point $z \in Y(V \cap V')$ mapping to $(\bar{\phi},0)_\ast(y) \in Y(V)$ and $(\bar{\phi'},0)_\ast(y') \in Y(V')$ respectively by $(\iota,v)_\ast$ and $(\iota',v')_\ast$. Set $k = \dim ( V \cap V')$ and let $h \colon \R^k \to \R^i$ be a linear isometry for which there is a factorization
\[\xymatrix@C=50pt{\R^k \ar[d]_-{\exists \beta} ^\cong \ar[r]^h & \R^i \ar[d]^\phi \\
V \cap V' \ar@{^{(}->}[r] & \R^m.}\]
Let $h'$ be the composite $h'\colon \R^k \stackrel{\beta}{\to} V \cap V' \hookrightarrow V' \stackrel{\bar{\phi'}^{-1}}{\longrightarrow} \R^j$, and

take $w = \bar{\phi}^{-1}(v)$ and $w' = \bar{\phi'}^{-1}(v')$. Then there are commutative diagrams in $O$
\[\xymatrix{\R^k \ar[r]^-{(\beta,0)} \ar[rd]_{(h',w')} & V \cap V'  \ar[r]^-{(\iota',v')}& V' \ar[d]^{(j',t')}\\
& \R^j \ar[r]_-{(\phi',t')} & \R^m} \ \ \ \ \ \ \ \
\mbox{and}
\ \ \ \ \ \ \ \ \xymatrix{\R^k \ar[r]^-{(\beta,0)} \ar[rd]_{(h,w)} & V \cap V'  \ar[r]^-{(\iota,v)}& V \ar[d]^{(j,t)}\\
& \R^i \ar[r]_-{(\phi,t)} & \R^m.}\]
Setting $u = (\beta,0)_\ast^{-1}(z) \in Y(\R^k)$, we get
\[(\phi,t)_\ast(h,w)_\ast(u)= (j,t)_\ast(\iota,v)_\ast(z) = (j,t)_\ast(\bar{\phi},0)_\ast(y)=(\phi,t)_\ast(y).\]
By condition $\ref{lemma:closed-ii}$ we have $(h,w)_\ast (u) =y$, and similarly $(h',w')_\ast (u) =y'$. Finally we have
\[ ((\phi,t),y) =((\phi,t),(h,w)_\ast(u)) \sim ((\phi,t)\circ(h,w),u) = ((j,t)\circ(\iota,v) \circ \beta,u) = ((j',t')\circ(\iota',v') \circ \beta,u)\]
and by a similar argument $((j',t')\circ(i',v') \circ \beta,u) \sim ((\phi',t'),y')$.
\end{proof}

The evaluation functor $\ev_n \colon \Sp^G \to \Top_\ast^{G \times O(n)}$ that sends a $G$-spectrum $Y$ to $Y(\R^n)$ has a $\Top_\ast$-enriched left adjoint, which we denote by $G_n$. Spectra of the form $G_n(Z)$, for $Z$ a pointed $G\times O(n)$-space, are called semi-free. We recall that this left adjoint is described by the formula
\[G_n(Z)(V)= O(\R^n,V) \wedge_{O(n)} Z.\]

\begin{lemma}Let $Z$ be a finite $G\times O(n)$-CW complex and $U$ a $G$-representation. Then the latching maps for $\Sh^U G_n(Z)$ are closed inclusions.
\end{lemma}
\begin{proof}
We check that the underlying (non-equivariant) spectrum of $\Sh^U G_n(Z)$ satisfies the conditions of Lemma \ref{lemma:closed}.
The spaces
\[(\Sh^UG_n(Z))(\R^m)= O(\R^n,U \oplus \R^m) \wedge_{O(n)} Z\]
 are all compact Hausdorff, so condition $\ref{lemma:closed-i}$ is satisfied. Condition $\ref{lemma:closed-ii}$ is clear. It remains to check condition $\ref{lemma:closed-iii}$.

Suppose we are given maps
\[ V \xrightarrow{(\alpha,x)} W \xleftarrow{(\alpha',x')} V' ,\]
where $\alpha$ and $\alpha'$ are subspace inclusions, and points $[\phi,y,z] \in O(\R^n,U\oplus V ) \wedge_{O(n)} Z$ and $[\phi',y',z'] \in O(\R^n,U\oplus V' ) \wedge_{O(n)} Z$ such that
\[ (\alpha,x)_\ast[\phi,y,z] =  (\alpha',x')_\ast[\phi',y',z'] \in O(\R^n,U\oplus W) \wedge_{O(n)} Z. \]
This means that there is a $g \in O(n)$ such that
\[ ((1_U \oplus \alpha)\circ \phi, x + (1_U\oplus \alpha)(y), z)  = ((1_U \oplus \alpha')\circ \phi'\circ g^{-1}, x' + (1_U\oplus \alpha')(y'), gz')\] in $O(\R^n, U \oplus W) \wedge Z$.

Since $\alpha$ and $\alpha'$ are subspace inclusions there is a unique map $\bar{\phi} \colon \R^m \to V \cap V'$ such that the diagram
\[ \xymatrix@C=70pt{ & U \oplus V \ar[rd]^-{1\oplus \alpha} & \\
\R^n \ar[ur]^\phi \ar[r]^-{\bar{\phi}} \ar[dr]_-{\phi'\circ g^{-1}} & U \oplus (V\cap V') \ar@{^{(}->}[u]  \ar[r]  \ar@{^{(}->}[d] & U \oplus W \\
& U \oplus V' \ar[ur]_-{1\oplus \alpha'} & } \]
commutes. We have decompositions $y = e+v$ and $y' = e' + v'$ where $e,e' \in U \oplus V \cap V'$, $v \in U \oplus V - U \oplus (V\cap V')$ and $v' \in U \oplus V' - U \oplus (V\cap V')$. Note that $U \oplus V - U \oplus (V\cap V') = V - (V\cap V')$ and similarly for $V'$. From this we get the equation
\[x + (1\oplus \alpha)(e) + (1\oplus \alpha)(v) = x' + (1\oplus \alpha')(e') + (1\oplus \alpha')(v').\]
Applying the orthogonal projection onto $U \oplus (V\cap V')$ we get $(1\oplus \alpha)(e)=(1\oplus \alpha')(e') $ which implies $e =e'$.
This shows that the diagram in $O$
\[\xymatrix@C=50pt{V \cap V' \ar[r]^-{(\iota',v')} \ar[d]_{(\iota,v)} & V' \ar[d]^{(\alpha',x')}\\
V \ar[r]_{(\alpha,x)} & W} \]
commutes, and that the point $[\bar{\phi},e,z] \in O(\R^n,U \oplus V \cap V')\wedge_{O(n)}Z$ maps to $[\phi,y,z]$ and $[\phi',y',z']$ respectively by the maps $(\iota,v)_\ast$ and $(\iota',v')_\ast$.
\end{proof}

\begin{lemma}\label{lemma:emb} Let $\Gamma$ be a compact Lie group and let $i \colon M \to N$ be a smooth, $\Gamma$-equivariant embedding of smooth compact closed $\Gamma$-manifolds. Then $i$ is a $\Gamma$-cofibration.
\end{lemma}
\begin{proof}
We will factor $i$ as a composite of two relative $\Gamma$-CW complexes. The $\Gamma$-manifold $M$ is a $\Gamma$-CW complex by Illman's equivariant triangulation Theorem \cite[7.2]{Illman}. Choose an equivariant Riemannian metric on $N$ and an equivariant tubular neighborhood of $M$, identified with the normal bundle $\nu$ of the embedding (see e.g. \cite{Kank}). The closed unit disc bundle $D(\nu)$ inherits an equivariant cell structure from $M$ such that the zero-section $M \hookrightarrow D(\nu)$ is a relative $\Gamma$-CW complex (see e.g. \cite[Lemma 1.1]{LO}). The space $N$ can be written as pushout
\[\xymatrix{S(\nu) \ar[r] \ar[d] & D(\nu) \ar[d]\\
N \setminus \stackrel{\circ}{D}(\nu) \ar[r] & N ,} \]
where $S(\nu)$ is the unit sphere bundle of $\nu$ and $\stackrel{\circ}{D}(\nu)$ is the open disk bundle. The left vertical map is the inclusion of the boundary in a smooth, compact $\Gamma$-manifold and is therefore a relative $\Gamma$-CW complex by Illman's Theorem \cite[7.2]{Illman}. By cobase change it follows that the inclusion $D(\nu) \to N$ is also a relative $\Gamma$-CW complex. The map $i$ is the composite $M \to D(\nu) \to N$ and is therefore a $\Gamma$-cofibration.
\end{proof}

\begin{lemma} \label{shiftsemifree} Let $Z$ be a finite $\Z/2\times O(n)$-CW complex and $U$ a $\Z/2$-representation. Then the $\Z/2$-spectrum $\Sh^U G_n(Z)$ is flat.
\end{lemma}

\begin{proof}
Any $\Z/2$-representation $U$ decomposes as a direct sum of one-dimensional representations. An isomorphism $U \cong U_1 \oplus \cdots \oplus U_d$ induces an isomorphism of functors $\Sh^U \cong \Sh^{U_1}\circ \cdots \circ \Sh^{U_d}$, so we may assume that $U$ is one-dimensional.

Let $Y= \Sh^U G_n(Z)$. We must show that for each $m \geq 0$ the latching map $L_m Y \to Y(\R^m)$ is a $\Z/2 \times O(m)$-cofibration. The latching space $L_mY$ is the coequalizer of the diagram
\[  \bigvee_{0 \leq k \leq l \leq m-1}\!\!\! O(\R^l,\R^m) \wedge O(\R^k,\R^l) \wedge O(\R^n,U \oplus \R^k) \wedge_{O(n)} Z \rightrightarrows \!\!\!  \bigvee_{0 \leq l \leq m-1} \!\!\! O(\R^l,\R^m) \wedge O(\R^n,U \oplus \R^l) \wedge_{O(n)} Z,\]
which is isomorphic to the value of the functor $-\wedge_{O(n)} Z$ on the coequalizer $W_m$ of
\[  \bigvee_{0 \leq k \leq l \leq m-1} O(\R^l,\R^m) \wedge O(\R^k,\R^l) \wedge O(\R^n,U \oplus \R^k) \rightrightarrows \bigvee_{0 \leq l \leq m-1} O(\R^l,\R^m) \wedge O(\R^n,U \oplus \R^l) .\]
The latching map can be written as the value of the functor $-\wedge_{O(n)} Z$ on a $G \times O(n) \times O(m)$-equivariant map $f_m \colon W_m \to O(\R^n,U \oplus \R^m)$. Since $-\wedge_{O(n)} Z$ preserves cofibrations, it suffices to show that $f_m$ is a $G \times O(n) \times O(m)$-cofibration.

If $m < n-1$, then there are no linear embeddings of $\R^n$ into $U \oplus \R^m$ and $f_m$ is just $\ast \to \ast$ which is a cofibration. If $m = n-1$ then $W_m = \ast$ and the map $f_m$ is the inclusion of the basepoint into $O(\R^n,U \oplus \R^{n-1}) = L(\R^n,U \oplus \R^{n-1})_+$ which is a cofibration, since the space $L(\R^n,U \oplus \R^{n-1})$ of linear isometries is a $G \times O(n) \times O(n-1)$-manifold, and hence cofibrant, by Illman's Theorem \cite{Illman}.

On the other hand if $m > n$, then we claim that $f_m$ is a homeomorphism. To see this, first note that the map $f_m$ agrees at the point set level with the latching map for the $G$-spectrum $\Sh^UG_n(O(n)_+)$, and is therefore a closed embedding by Lemma \ref{lemma:closed}. Since all the spaces involved are compact Hausdorff it now suffices to show that $f_m$ is surjective. Take any point
\[(\alpha,x) \in O(\R^n, U \oplus \R^m)\]
and let $V$ denote the image of the map $\R^n \stackrel{\alpha}{\longrightarrow} U \oplus \R^m \stackrel{\proj}{\longrightarrow} \R^m$.
Let $v = \dim V$, then $v \leq n \leq m-1$. The map $\alpha$ factors as
\[\xymatrix{\R^n \ar[dr]_-{\bar{\alpha}} \ar[rr]^-\alpha && U \oplus \R^m \\
& U \oplus \R^v \ar[ur]_-{id_U \oplus \beta} &,}
\]
where $\beta$ is an isometry with image $V$. The vector $x$ can be decomposed as a sum $x = x' +x''$, where $x' \in U \oplus V$ and $x'' \in U \oplus \R^m - U \oplus V = \R^m - V$.  Now if $y'$ is a vector in $U\oplus \R^v$ which maps to $x'$ then the class in $W_m$ of the point
\[ (\beta,x'') \wedge (\bar{\alpha},y') \in O(\R^v, \R^m) \wedge O(\R^n, U \oplus \R^v)\]
is mapped to $(\alpha,x)$ by $f_m$.


We are left with the case $m= n$. The space $W_m$ is homeomorphic to
\[O(\R^{n-1},\R^n) \wedge_{O(n-1)} O(\R^n, U \oplus \R^{n-1})\]
and $f_m$ is the map on Thom spaces induced by map of vector bundles which is also a pullback diagram
\[\xymatrix{E \ar[d] \ar[rr] && \gamma(U\oplus \R^n)^\perp \ar[d] \\
L(\R^{n-1},\R^n) \times_{O(n-1)} L(\R^n, U \oplus \R^{n-1})  \ar[rr] && L(\R^n, U \oplus \R^{n})\rlap{\ ,}}\]
where $\gamma(U\oplus\mathbb{R}^n)^\perp$ is the orthogonal complement bundle.
The lower horizontal map is a smooth map between smooth manifolds which is $G\times O(n) \times O(n)$-equivariant and also a closed embedding. By Lemma \ref{lemma:emb} it is therefore an equivariant cofibration and it follows that the induced map on Thom spaces is a cofibration as well (see e.g. \cite[Lemma 1.1]{LO}).
\end{proof}

\begin{proof}[Proof of Proposition \ref{prop:flatness}]
We will now show that shifts preserve flat cofibrations ($\mathbb{S}$-cofibrations) of orthogonal $\Z/2$-spectra. The functor $\Sh^U$-commutes with colimits, hence it suffices to show that $\Sh^U$ sends generating flat cofibrations to flat cofibrations. The generating flat cofibrations have the form
\[ i \wedge 1\colon  A \wedge G_n(Z) \longrightarrow B \wedge G_n(Z), \]
where $i\colon  A \rightarrow B$ is a (non-equivariant) cofibration of spaces and $Z$ is a finite $G\times O(n)$-CW complex. The map $\Sh^U(i \wedge 1)$ is isomorphic to the map
\[i \wedge 1\colon  A \wedge \Sh^U(G_n(Z)) \longrightarrow B \wedge \Sh^U(G_n(Z)). \]
Lemma \ref{shiftsemifree} implies that $\Sh^U(G_n(Z))$ is flat, and since the $\mathbb{S}$-model structure is topological \cite{Sto}, the latter map is a flat cofibration.
\end{proof}

\phantomsection\addcontentsline{toc}{section}{References} 
\bibliographystyle{amsalpha}
\bibliography{bib}

\providecommand{\bysame}{\leavevmode\hbox to3em{\hrulefill}\thinspace}
\providecommand{\MR}{\relax\ifhmode\unskip\space\fi MR }
\providecommand{\MRhref}[2]{%
  \href{http://www.ams.org/mathscinet-getitem?mr=#1}{#2}
}
\providecommand{\href}[2]{#2}
\begin{thebibliography}{EKMM97}

\bibitem[AR05]{angrog}
Vigleik Angeltveit and John Rognes, \emph{Hopf algebra structure on topological
  {H}ochschild homology}, Algebr. Geom. Topol. \textbf{5} (2005), 1223--1290.
  \MR{2171809}

\bibitem[BDS16]{BrDuSt}
Morten Brun, Bj{\o}rn~Ian Dundas, and Martin Stolz, \emph{Equivariant structure
  on smash powers}, arXiv:1604.05939, 2016.

\bibitem[BF84]{BF}
D.~Burghelea and Z.~Fiedorowicz, \emph{Hermitian algebraic {$K$}-theory of
  topological spaces}, Algebraic {$K$}-theory, number theory, geometry and
  analysis ({B}ielefeld, 1982), Lecture Notes in Math., vol. 1046, Springer,
  Berlin, 1984, pp.~32--46. \MR{750675}

\bibitem[BHM93]{BHM}
M.~B{\"o}kstedt, W.~C. Hsiang, and I.~Madsen, \emph{The cyclotomic trace and
  algebraic {$K$}-theory of spaces}, Invent. Math. \textbf{111} (1993), no.~3,
  465--539. \MR{1202133 (94g:55011)}

\bibitem[BK72]{BK}
A.~K. Bousfield and D.~M. Kan, \emph{Homotopy limits, completions and
  localizations}, Lecture Notes in Mathematics, Vol. 304, Springer-Verlag,
  Berlin, 1972. \MR{0365573 (51 \#1825)}

\bibitem[BM12]{BMlocTHH}
Andrew~J. Blumberg and Michael~A. Mandell, \emph{Localization theorems in
  topological {H}ochschild homology and topological cyclic homology}, Geom.
  Topol. \textbf{16} (2012), no.~2, 1053--1120. \MR{2928988}

\bibitem[BM15]{BMcyc}
\bysame, \emph{The homotopy theory of cyclotomic spectra}, Geom. Topol.
  \textbf{19} (2015), no.~6, 3105--3147. \MR{3447100}

\bibitem[B{\"o}k]{Bok2}
Marcel B{\"o}kstedt, \emph{The topological {H}ochschild homology of
  $\mathbb{Z}$ and $\mathbb{Z}/p$}, preprint.

\bibitem[B{\"o}k86]{Bok}
\bysame, \emph{Topological {H}ochschild homology}, preprint, 1986.

\bibitem[Bou97]{Bouc}
Serge Bouc, \emph{Green functors and {$G$}-sets}, Lecture Notes in Mathematics,
  vol. 1671, Springer-Verlag, Berlin, 1997. \MR{1483069}

\bibitem[BP17]{BP}
Peter Bonventre and Lu\'{i}s~A. Pereira, \emph{Genuine equivariant operads},
  Arxiv:1707.02226, 2017.

\bibitem[Bre78]{breen}
Lawrence Breen, \emph{Extensions du groupe additif}, Inst. Hautes \'Etudes Sci.
  Publ. Math. (1978), no.~48, 39--125. \MR{516914}

\bibitem[Bru07]{Brun}
M.~Brun, \emph{Witt vectors and equivariant ring spectra applied to cobordism},
  Proc. Lond. Math. Soc. (3) \textbf{94} (2007), no.~2, 351--385. \MR{2308231}

\bibitem[Cn93]{gc}
Guillermo Corti\~nas, \emph{{$L$}-theory and dihedral homology. {II}}, Topology
  Appl. \textbf{51} (1993), no.~1, 53--69. \MR{1229500}

\bibitem[DGM13]{DGM}
Bj{\o}rn~Ian Dundas, Thomas~G. Goodwillie, and Randy McCarthy, \emph{The local
  structure of algebraic {K}-theory}, Algebra and Applications, vol.~18,
  Springer-Verlag London Ltd., London, 2013. \MR{3013261}

\bibitem[DM94]{DM}
Bj{\o}rn~Ian Dundas and Randy McCarthy, \emph{Stable {$K$}-theory and
  topological {H}ochschild homology}, Ann. of Math. (2) \textbf{140} (1994),
  no.~3, 685--701. \MR{1307900 (96e:19005a)}

\bibitem[DM96]{ringfctrs}
\bysame, \emph{Topological {H}ochschild homology of ring functors and exact
  categories}, J. Pure Appl. Algebra \textbf{109} (1996), no.~3, 231--294.
  \MR{1388700 (97i:19001)}

\bibitem[DM16]{Gdiags}
Emanuele Dotto and Kristian Moi, \emph{Homotopy theory of {$G$}-diagrams and
  equivariant excision}, Algebr. Geom. Topol. \textbf{16} (2016), no.~1,
  325--395. \MR{3470703}

\bibitem[DO17]{DO}
Emanuele Dotto and Crichton Ogle, \emph{${K}$-theory of {H}ermitian {M}ackey
  functors and a reformulation of the {N}ovikov conjecture}, Arxiv:1703.09523,
  2017.

\bibitem[Dot12]{thesis}
Emanuele Dotto, \emph{Stable {R}eal {$K$}-theory and {R}eal topological
  {H}ochschild homology}, Ph.D. thesis, University of Copenhagen, 2012,
  arXiv:1212.4310.

\bibitem[Dot16]{Gcalc}
Emanuele Dotto, \emph{Equivariant calculus of functors and
  $\mathbb{Z}/2$-analyticity of {R}eal algebraic {$K$}-theory}, Journal of the
  Institute of Mathematics of Jussieu \textbf{15(4)} (2016), pp. 829--883.

\bibitem[Dun97]{Dundas}
Bj{\o}rn~Ian Dundas, \emph{Relative {$K$}-theory and topological cyclic
  homology}, Acta Math. \textbf{179} (1997), no.~2, 223--242. \MR{1607556
  (99e:19007)}

\bibitem[EKMM97]{EKMM}
A.~D. Elmendorf, I.~Kriz, M.~A. Mandell, and J.~P. May, \emph{Rings, modules,
  and algebras in stable homotopy theory}, Mathematical Surveys and Monographs,
  vol.~47, American Mathematical Society, Providence, RI, 1997, With an
  appendix by M. Cole. \MR{1417719 (97h:55006)}

\bibitem[ERW17]{ERW}
Johannes Ebert and Oscar Randal-Williams, \emph{Semi-simplicial spaces},
  arXiv:1705.03774, 2017.

\bibitem[GMM15]{MonaEG}
B.~J. Guillou, J.~P. May, and M.~Merling, \emph{Categorical models for
  equivariant classifying spaces}, arXiv: 1505.07562v1, 2015.

\bibitem[Goo85]{Goodwilliecyclic}
Thomas~G. Goodwillie, \emph{Cyclic homology, derivations, and the free
  loopspace}, Topology \textbf{24} (1985), no.~2, 187--215. \MR{793184}

\bibitem[Hes16]{LarsZeta}
Lars Hesselholt, \emph{Topological {H}ochschild homology and the {H}asse-{W}eil
  zeta function}, Alpine Algebraic and Applied Topology (Saas Almagell,
  Switzerland), Contemp. Math., Amer. Math. Soc., Providence, RI, 2016, To
  appear.

\bibitem[HHR16]{HHR}
M.~A. Hill, M.~J. Hopkins, and D.~C. Ravenel, \emph{On the nonexistence of
  elements of {K}ervaire invariant one}, Ann. of Math. (2) \textbf{184} (2016),
  no.~1, 1--262. \MR{3505179}

\bibitem[HK01]{HuKriz}
Po~Hu and Igor Kriz, \emph{Real-oriented homotopy theory and an analogue of the
  {A}dams-{N}ovikov spectral sequence}, Topology \textbf{40} (2001), no.~2,
  317--399. \MR{1808224}

\bibitem[HM97]{Wittvect}
Lars Hesselholt and Ib~Madsen, \emph{On the {$K$}-theory of finite algebras
  over {W}itt vectors of perfect fields}, Topology \textbf{36} (1997), no.~1,
  29--101. \MR{1410465 (97i:19002)}

\bibitem[HM03]{IbLarsLocalfields}
\bysame, \emph{On the {$K$}-theory of local fields}, Ann. of Math. (2)
  \textbf{158} (2003), no.~1, 1--113. \MR{1998478}

\bibitem[HM04]{IbLarsDeRhamMixed}
\bysame, \emph{On the {D}e {R}ham-{W}itt complex in mixed characteristic}, Ann.
  Sci. \'Ecole Norm. Sup. (4) \textbf{37} (2004), no.~1, 1--43. \MR{2050204}

\bibitem[HM15]{IbLars}
\bysame, \emph{Real algebraic {$K$}-theory},
  http://www.math.ku.dk/~larsh/papers/s05/, 2015.

\bibitem[H{\o}g16]{Amalie}
Amalie H{\o}genhaven, \emph{Real topological cyclic homology of spherical group
  rings}, arXiv: 1611.01204, 2016.

\bibitem[H{\o}g17]{AmalieGeom}
\bysame, \emph{On the geometric fixed points of real topological {H}ochschild
  homology}, arXiv: 1710.01817, 2017.

\bibitem[Hoy14]{Hoy}
Rolf Hoyer, \emph{Two topics in stable homotopy theory}, Ph.D. thesis,
  University of Chicago, 2014.

\bibitem[Ill83]{Illman}
S\"oren Illman, \emph{The equivariant triangulation theorem for actions of
  compact {L}ie groups}, Math. Ann. \textbf{262} (1983), no.~4, 487--501.
  \MR{696520}

\bibitem[Kan07]{Kank}
Marja Kankaanrinta, \emph{Equivariant collaring, tubular neighbourhood and
  gluing theorems for proper {L}ie group actions}, Algebr. Geom. Topol.
  \textbf{7} (2007), 1--27. \MR{2289802}

\bibitem[Kro05]{Kroinv}
Tore~August Kro, \emph{Involutions on $\mathbb{S}[{\Omega} {M}]$},
  arXiv:math/0510221, 2005.

\bibitem[Lew81]{LewisGreen}
L.~Gaunce Lewis, Jr., \emph{The theory of {G}reen functors}, 1981.

\bibitem[LM06]{LM06}
L.~Gaunce Lewis, Jr. and Michael~A. Mandell, \emph{Equivariant universal
  coefficient and {K}\"unneth spectral sequences}, Proc. London Math. Soc. (3)
  \textbf{92} (2006), no.~2, 505--544. \MR{2205726}

\bibitem[LMSM86]{LMS}
L.~G. Lewis, Jr., J.~P. May, M.~Steinberger, and J.~E. McClure,
  \emph{Equivariant stable homotopy theory}, Lecture Notes in Mathematics, vol.
  1213, Springer-Verlag, Berlin, 1986, With contributions by J. E. McClure.
  \MR{866482 (88e:55002)}

\bibitem[LO01]{LO}
Wolfgang L\"uck and Bob Oliver, \emph{The completion theorem in {$K$}-theory
  for proper actions of a discrete group}, Topology \textbf{40} (2001), no.~3,
  585--616. \MR{1838997}

\bibitem[Lod76]{jl}
Jean-Louis Loday, \emph{{$K$}-th\'eorie alg\'ebrique et repr\'esentations de
  groupes}, Ann. Sci. \'Ecole Norm. Sup. (4) \textbf{9} (1976), no.~3,
  309--377. \MR{0447373}

\bibitem[Lod98]{Loday}
\bysame, \emph{Cyclic homology}, second ed., Grundlehren der Mathematischen
  Wissenschaften [Fundamental Principles of Mathematical Sciences], vol. 301,
  Springer-Verlag, Berlin, 1998, Appendix E by Mar\'\i a O. Ronco, Chapter 13
  by the author in collaboration with Teimuraz Pirashvili. \MR{1600246}

\bibitem[Mal16]{Cary}
Cary Malkiewich, \emph{Cyclotomic structure in the topological {H}ochschild
  homology of {DX}}, arXiv:1505.06778, 2016.

\bibitem[May92]{MaySimp}
J.~Peter May, \emph{Simplicial objects in algebraic topology}, Chicago Lectures
  in Mathematics, University of Chicago Press, Chicago, IL, 1992, Reprint of
  the 1967 original. \MR{1206474}

\bibitem[Maz16]{Maz}
Kristen Mazur, \emph{An equivariant tensor product on {M}ackey functors},
  arXiv:1508.04062v3, 2016.

\bibitem[McC97]{McCarthy}
Randy McCarthy, \emph{Relative algebraic {$K$}-theory and topological cyclic
  homology}, Acta Math. \textbf{179} (1997), no.~2, 197--222. \MR{1607555
  (99e:19006)}

\bibitem[MM02]{ManMay}
M.~A. Mandell and J.~P. May, \emph{Equivariant orthogonal spectra and
  {$S$}-modules}, Mem. Amer. Math. Soc. \textbf{159} (2002), no.~755, x+108.
  \MR{1922205 (2003i:55012)}

\bibitem[MMSS01]{MMSS}
M.~A. Mandell, J.~P. May, S.~Schwede, and B.~Shipley, \emph{Model categories of
  diagram spectra}, Proc. London Math. Soc. (3) \textbf{82} (2001), no.~2,
  441--512. \MR{1806878}

\bibitem[NS17]{NikSchol}
Thomas Nikolaus and Peter Scholze, \emph{On topological cyclic homology},
  Arxiv:1707.01799, 2017.

\bibitem[PS16]{PatchkSagave}
Irakli Patchkoria and Steffen Sagave, \emph{Topological {H}ochschild homology
  and the cyclic bar construction in symmetric spectra}, Proc. Amer. Math. Soc.
  \textbf{144} (2016), no.~9, 4099--4106. \MR{3513565}

\bibitem[Qui73]{Quillen}
Daniel Quillen, \emph{Higher algebraic {$K$}-theory. {I}}, Algebraic
  {$K$}-theory, {I}: {H}igher {$K$}-theories ({P}roc. {C}onf., {B}attelle
  {M}emorial {I}nst., {S}eattle, {W}ash., 1972), Springer, Berlin, 1973,
  pp.~85--147. Lecture Notes in Math., Vol. 341. \MR{0338129 (49 \#2895)}

\bibitem[Sch04]{SchlichtkrullUnits}
Christian Schlichtkrull, \emph{Units of ring spectra and their traces in
  algebraic {$K$}-theory}, Geom. Topol. \textbf{8} (2004), 645--673.
  \MR{2057776}

\bibitem[Sch07]{Schwedesym}
Stefan Schwede, \emph{Symmetric spectra}, 2007.

\bibitem[Sch08]{Schwedehtpygps}
\bysame, \emph{On the homotopy groups of symmetric spectra}, Geom. Topol.
  \textbf{12} (2008), no.~3, 1313--1344. \MR{2421129}

\bibitem[Sch10]{Schlichting}
Marco Schlichting, \emph{Hermitian {$K$}-theory of exact categories}, J.
  K-Theory \textbf{5} (2010), no.~1, 105--165. \MR{2600285 (2011b:19007)}

\bibitem[Sch13]{Schwede}
Stefan Schwede, \emph{Lectures on equivariant stable homotopy theory}, 2013.

\bibitem[Sch17]{Schwedeglobal}
\bysame, \emph{Global homotopy theory}, to appear in New Mathematical
  Monographs, Cambridge University Press, 2017.

\bibitem[Seg73]{Segalsub}
Graeme Segal, \emph{Configuration-spaces and iterated loop-spaces}, Invent.
  Math. \textbf{21} (1973), 213--221. \MR{0331377}

\bibitem[Shi00]{Sh00}
Brooke Shipley, \emph{Symmetric spectra and topological {H}ochschild homology},
  $K$-Theory \textbf{19} (2000), no.~2, 155--183. \MR{1740756}

\bibitem[SS00]{SchwShip}
Stefan Schwede and Brooke~E. Shipley, \emph{Algebras and modules in monoidal
  model categories}, Proc. London Math. Soc. (3) \textbf{80} (2000), no.~2,
  491--511. \MR{1734325}

\bibitem[SS03]{SchShip1}
Stefan Schwede and Brooke Shipley, \emph{Stable model categories are categories
  of modules}, Topology \textbf{42} (2003), no.~1, 103--153. \MR{1928647}

\bibitem[SS12]{SagSchlichtdiag}
Steffen Sagave and Christian Schlichtkrull, \emph{Diagram spaces and symmetric
  spectra}, Adv. Math. \textbf{231} (2012), no.~3-4, 2116--2193. \MR{2964635}

\bibitem[SS13]{SagSchCpletion}
\bysame, \emph{Group completion and units in {$I$}-spaces}, Algebr. Geom.
  Topol. \textbf{13} (2013), no.~2, 625--686. \MR{3044590}

\bibitem[Sti13]{Stiennon}
Nisan Stiennon, \emph{The moduli space of real curves and a
  $\mathbb{Z}/2$-equivariant {M}adsen-{W}eiss theorem}, Ph.D. thesis, Stanford,
  2013.

\bibitem[Sto11]{Sto}
Martin Stolz, \emph{Equivariant structure on smash powers of commutative ring
  spectra}, Ph.D. thesis, University of Bergen, 2011.

\bibitem[Str12]{Strickland}
Niel Strickland, \emph{Tambara functors}, arXiv: 1205.2516, 2012.

\bibitem[Tho79]{Thomason}
R.~W. Thomason, \emph{Homotopy colimits in the category of small categories},
  Math. Proc. Cambridge Philos. Soc. \textbf{85} (1979), no.~1, 91--109.
  \MR{510404 (80b:18015)}

\bibitem[Til16]{Tilson}
Sean Tilson, \emph{Power operations in the {K}\"{u}nneth spectral sequence and
  commutative {$H\mathbb{F}_p$}-algebras}, arXiv:1602.06736, 2016.

\bibitem[Ull13]{Ul2}
John Ullman, \emph{Symmetric powers and norms of {M}ackey functors},
  arXiv:1304.5648v2, 2013.

\bibitem[Wal70]{Wall}
C.~T.~C. Wall, \emph{On the axiomatic foundations of the theory of {H}ermitian
  forms}, Proc. Cambridge Philos. Soc. \textbf{67} (1970), 243--250.
  \MR{0251054 (40 \#4285)}

\bibitem[Wil17]{Wilson}
Dylan Wilson, \emph{Power operations for {$H\underline{\mathbb{F}}_2$} and a
  cellular construction of {$BP\mathbf{R}$}}, arXiv:1611.06958v2, 2017.

\end{thebibliography}

\end{document}